%% file: article.tex
\documentclass[final,onefignum,onetabnum]{siamonline220329}

\input{shared}

\ifpdf
\hypersetup{
  pdftitle={Precise Asymptotics for Spectral Methods in Mixed Generalized Linear Models},
  pdfauthor={Y.~Zhang, M.~Mondelli, and R.~Venkataramanan}
}
\fi

\begin{document}

\maketitle

\begin{abstract}
In a mixed generalized linear model, the goal is to learn multiple signals from unlabeled observations: each sample comes from exactly one signal, but it is not known which one. We consider the prototypical problem of estimating two statistically independent signals in a mixed generalized linear model with Gaussian covariates. Spectral methods are a popular class of estimators which output the top two eigenvectors of a suitable data-dependent matrix. However, despite the wide applicability, their design is still obtained via heuristic considerations, and the number of samples $n$ needed to guarantee recovery is super-linear in the signal dimension $d$. In this paper, we develop exact asymptotics on spectral methods in the challenging proportional regime in which $n, d$ grow large and their ratio converges to a finite constant. This allows us optimize the design of the spectral method, and combine it with a simple linear estimator, to minimize the estimation error. Our characterization exploits a mix of tools from random matrices, free probability and the theory of approximate message passing algorithms. Numerical simulations for mixed linear regression and phase retrieval demonstrate the advantage enabled by our analysis over existing designs of spectral methods. 
\end{abstract}

\begin{keywords}
Spectral estimator, generalized linear models, mixed regression, high dimensional asymptotics, random matrix theory, Approximate Message Passing (AMP).
\end{keywords}

\begin{MSCcodes}
62E20, 62J05, 62J12.
\end{MSCcodes}

\section{Introduction}

We consider the problem of learning multiple $d$-dimensional vectors from $n$ unlabeled observations coming from a \emph{mixed} generalized linear model (GLM):
\begin{align}
y_i &= q\paren{\inprod{a_i}{x_{\upsilon_i}^*} ,\eps_i}, \qquad i\in [n]=\{1, \ldots, n\}.  \label{eqn:def-y-glmm-intro} 
\end{align}
Here, $x_1^*, \ldots, x_\ell^*\in\mathbb R^d$ are the $\ell$ signals (regression vectors) to be recovered from the observation vector $y=(y_1, \ldots, y_n)\in \mathbb R^n$ and the known design matrix $A=[a_1, \ldots, a_n]^\top\in \mathbb R^{n\times d}$. For $i \in [n]$, $\eps_i$  is a noise variable, and 
$\upsilon_i$ is an $[\ell]$-valued latent variable, i.e., it indicates which signal each observation comes from, and is unknown to the statistician. The notation $\inprod{\cdot}{\cdot}$ denotes the Euclidean inner product, and $q:\mathbb R^2\to\mathbb R$ is a known link function. For $\ell=1$, \Cref{eqn:def-y-glmm-intro} reduces to a generalized linear model \cite{mccullagh2018generalized}, which covers many widely studied problems in statistical estimation including linear regression, logistic regression,  
phase retrieval \cite{shechtman2015phase,fannjiang2020numerics}, and 1-bit compressed sensing \cite{boufounos20081}. The regression model with $\ell=1$ implicitly assumes a homogeneous population, in which a single regression vector suffices to capture the features of the entire sample. In practice, it is often the case that the observations may come from multiple sub-populations. Mixed GLMs offer a flexible solution in settings with unlabeled heterogeneous data, and have found applications in a variety of fields including biology, physics, and economics \cite{Mcl04,Gru07,Li19,Dev20}. When $q(g, \eps)=g+\eps$, \Cref{eqn:def-y-glmm-intro} reduces to the widely studied mixture of linear regressions  \cite{Vie02, Far10, Sta10, Cha13,  yi-2014-mixed-linear-regression, Zho16, She19, Zha20, Gho20}.

A natural approach to estimate the vectors $x_1^*, \ldots, x_\ell^*$ from $y$ and $A$ is via the maximum-likelihood estimator (assuming a statistical model  for $(\eps_i)_{i \in [n]}$ is available).
However, the corresponding optimization problem is non-convex and NP-hard \cite{yi-2014-mixed-linear-regression}. Thus, various low-complexity alternatives --- mostly focusing on mixed linear regression --- have been proposed: examples include expectation-maximization (EM) \cite{Kha07,Far10, Sta10}, alternating minimization \cite{yi-2014-mixed-linear-regression, She19, Gho20}, convex relaxation \cite{Che14},  moment descent methods \cite{Li18,Che20}, and the use of tractable non-convex objectives \cite{Zho16,Bar22}. Many of these methods are iterative in nature and require a ``warm start'' with an initial guess correlated with the ground truth. Spectral methods are a popular way to provide such initialization  \cite{yi-2014-mixed-linear-regression}. A variety of estimators based on the spectral decomposition of data-dependent  matrices or tensors have been proposed for mixed GLMs \cite{Cha13, yi-2014-mixed-linear-regression, sedghi2016provable}. In this paper, we focus on a spectral method that estimates the $\ell$ signals via the top-$\ell$ principal eigenvectors of the following data-dependent matrix:
\begin{align}
    D = \frac{1}{n} \sum_{i = 1}^n \cT(y_i) a_i a_i^\top \in\bbR^{d\times d} ,\label{eqn:def-mtx-t-and-d-intro} 
\end{align}
where $\mathcal T:\mathbb R\to\mathbb R$ is a suitably chosen preprocessing function. This spectral estimator with the preprocessing function $\cT(y) = y^2$ was studied for mixed linear regression by Yi et al.\ \cite{yi-2014-mixed-linear-regression}, who showed that the signals can be accurately recovered when the number of observations $n$ is of order $d\log d$. Furthermore,  existing theoretical results for all estimators (including spectral, alternating minimization and EM)  require 
$n$ to be  of order at least $d\log d$ to guarantee accurate recovery \cite{Cha13,yi-2014-mixed-linear-regression,sedghi2016provable,Li18,Che20}. This leads to the following natural questions:

\vspace{.5em}

\begin{center}

\emph{What is the optimal sample complexity of a spectral estimator based on \Cref{eqn:def-mtx-t-and-d-intro}? \\
\vspace{.25em}
Can we carry out a principled optimization of the preprocessing function $\cT$?}
\end{center}

\vspace{.5em}

\noindent A simpler alternative to obtain an initial estimate is to use the linear estimator
\begin{align}
    \frac{1}{n} \sum_{i = 1}^n \cL(y_i) a_i \ \in\bbR^d, \label{eqn:def-lin-estimator-intro}
\end{align}
where $\cL: \reals \to \reals$ is a suitable preprocessing function. 
The   performance analysis of this linear estimator  for the mixed GLM can be carried out similarly to that for the non-mixed case ($\ell=1$); the analysis for the latter is given in \cite[Proposition 1]{plan2017high} and  in \cite[Lemma 2.1]{mondelli2021optimalcombination}. Thus, another natural question is:

\vspace{.5em}

\begin{center}
\emph{What is the optimal way to combine a spectral estimator based on\! \Cref{eqn:def-mtx-t-and-d-intro} and the linear estimator in \Cref{eqn:def-lin-estimator-intro}?}
\end{center}

\vspace{.5em}


\subsection{Main contributions}
\label{sec:sum-results}

In this paper, we resolve the questions above for the 
recovery of two independent signals $x_1^*, x_2^*$ with a Gaussian design matrix $A$.
This is achieved by characterizing the high-dimensional limit of the joint empirical distribution of \emph{(i)} the signals $x_1^*, x_2^*$, \emph{(ii)} the linear estimator in \Cref{eqn:def-lin-estimator-intro}, and \emph{(iii)} spectral estimators based on the matrix in \Cref{eqn:def-mtx-t-and-d-intro}. Our analysis holds in the proportional setting where $n,d\to\infty$ with $n/d\to\delta\in(0,\infty)$.
That is, we consider the regime where the ratio between sample size and signal dimension tends to a constant, as opposed to most analyses of mixed GLMs in the literature which assume $n = \Omega (d \log d)$. 
Our major findings are summarized as follows. 

\vspace{.5em}

\begin{itemize}[leftmargin=3mm]

    \item Our master theorem (\Cref{thm:main-thm-joint-dist}) characterizes the joint distribution of the linear estimator, the spectral estimator, and the signals in the high-dimensional limit.  This joint distribution characterization holds for arbitrary preprocessing functions $ \cL,\cT\colon\bbR\to\bbR $  in \Cref{eqn:def-mtx-t-and-d-intro,eqn:def-lin-estimator-intro} (subject to 
    mild regularity conditions). 
    The limiting joint distribution is expressed as the law of a set of jointly Gaussian random variables whose covariance structure is explicitly derived in terms of the model and the preprocessing functions. 

\vspace{.25em}
    
    \item As an immediate consequence of the distributional characterization, we derive the normalized correlations (or `overlaps') between the linear/spectral estimator and the signals (\Cref{lem:linear-overlap}/\Cref{thm:overlap-spectral}).
    The linear estimator achieves a strictly positive overlap    with each signal for any $\delta > 0$, provided  a strictly positive overlap can be attained for some $\delta>0$. In contrast, for the spectral estimator, we identify a threshold (depending on the preprocessing function $\cT$) such that strictly positive overlap is attained as soon as $\delta$ exceeds this threshold. In general, there is no clear winner between the spectral and the linear estimator, and which one performs better depends on the setting. 

\vspace{.25em}
    
    \item In fact, it is best to combine the linear and spectral estimators: our master theorem also allows us to compute the limiting overlap of a class of such combinations. 
    In particular, the Bayes-optimal combination can be derived, which turns out to be linear in the two estimators due to the Gaussianity of their high-dimensional limits (\Cref{cor:opt-combo}). 
    
    \vspace{.25em}

    \item 
    We determine the optimal preprocessing functions $ \cL^*, \cT_1^*,\cT_2^*\colon\bbR\to\bbR $ for the linear and spectral estimators that maximize the overlap between the estimator and each signal (\Cref{lem:linear-optimal-overlap,thm:opt-spec}). 
    The optimal overlaps of linear and spectral estimators reveal intriguing behaviors of mixed models. 
    In particular, there is a \emph{single} function $\cL^*$ that simultaneously maximizes the overlap between the linear estimator and each signal. 
    In contrast, for the spectral method,  one needs to employ two \emph{different} functions $ \cT_1^*,\cT_2^* $ 
    in order to achieve the maximal overlaps with $x_1^*$, $x_2^*$, respectively. 
    Furthermore, the optimal overlap of the spectral estimator with each signal approaches $1$ --- the best possible value --- as the aspect ratio $\delta$ grows. 
    We remark that the same is not true for the  linear estimator: the optimal overlap with each signal remains strictly less than 1 even as $\delta \to \infty$, as long as there is a strictly positive fraction of observations corresponding to each signal.
    

    
\end{itemize}

    \vspace{.5em}

Our precise asymptotic analysis leads to a significant improvement over previous designs of spectral methods, as showcased in \Cref{fig:all} for noiseless mixed linear regression. The continuous lines correspond to our theoretical predictions (``pred.''), which closely match the points coming from the simulations (``sim.''). The following methods are compared: \emph{(i)} optimal spectral method (black), obtained from \Cref{thm:opt-spec}; \emph{(ii)} optimal linear method (blue), obtained from \Cref{lem:linear-optimal-overlap}; \emph{(iii)} combined estimator (``combo'') (red), obtained from \Cref{cor:opt-combo};  \emph{(iv)} spectral estimator for mixed linear regression proposed in \cite{yi-2014-mixed-linear-regression} (yellow); \emph{(v)} spectral estimator which optimizes the overlap in the non-mixed setting (green), proposed in \cite{luo-2019-opt-preprocessing}. The spectral methods resulting from our sharp analysis (red, black) significantly outperform existing methods (green, yellow), especially for low values of $\delta$. More details on the experimental setup and additional simulation results can be found in \Cref{sec:experiments}. 

\begin{figure}[tbp]
    \centering
    \begin{subfigure}{0.49\linewidth}
        \centering
        \includegraphics[width=\linewidth]{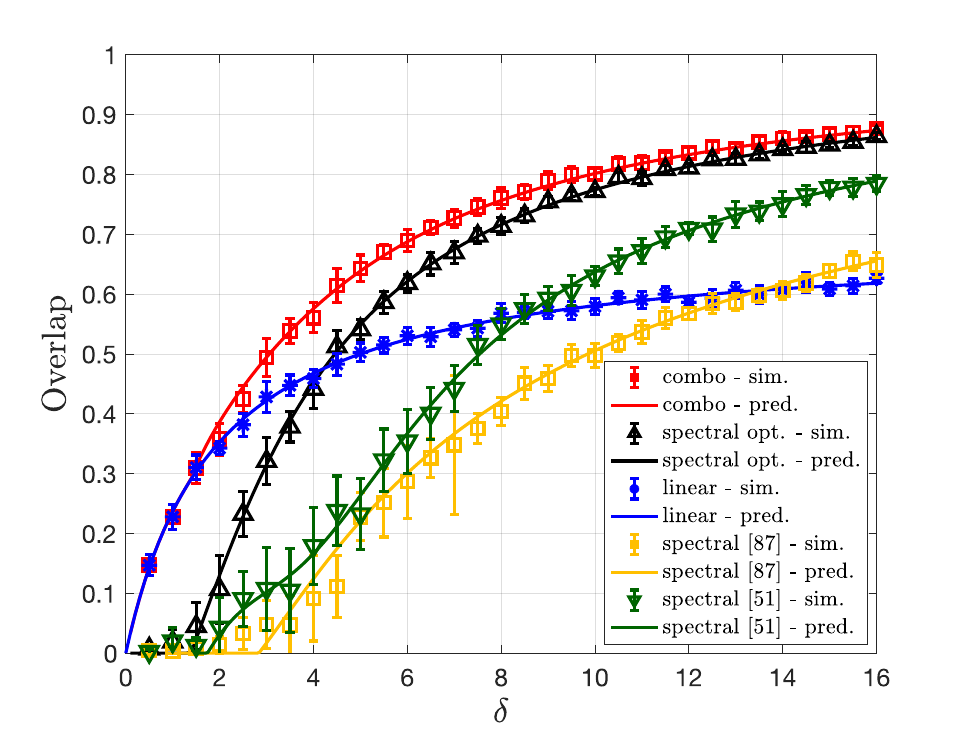}
        \caption{Recovery of $ x_1^* $}
        \label{fig:all1}
    \end{subfigure}
    \hfill
    \begin{subfigure}{0.49\linewidth}
        \centering
        \includegraphics[width=\linewidth]{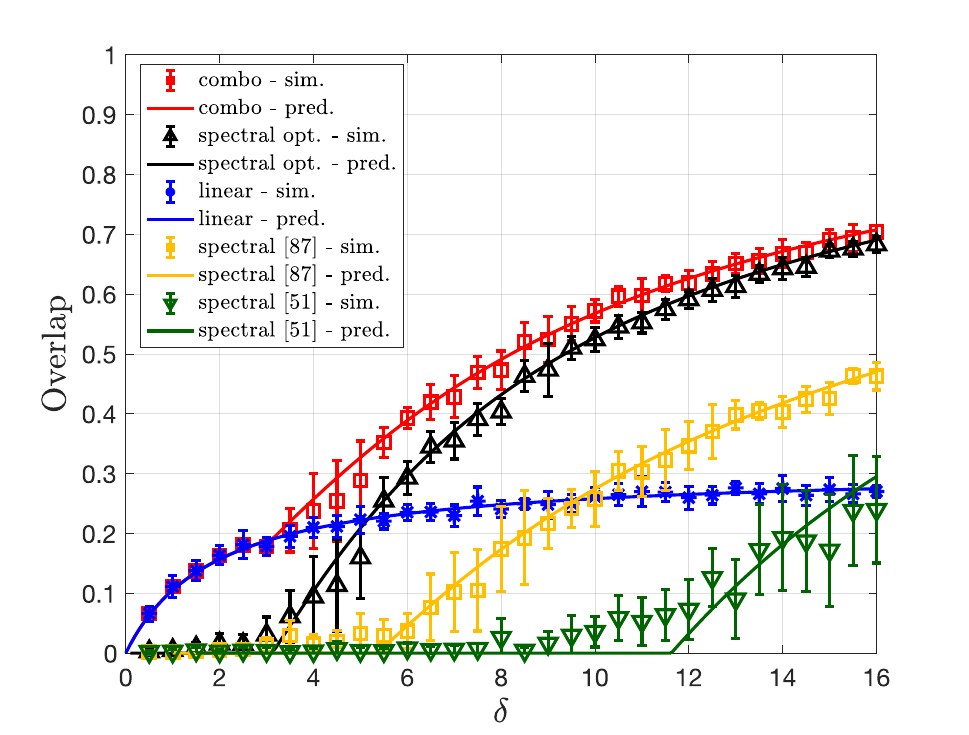}
        \caption{Recovery of $ x_2^* $}
        \label{fig:all2}
    \end{subfigure}
    \caption{Noiseless mixed linear regression with mixing parameter (i.e., probability that a sample corresponds to $x_1^*$) $ \alpha = 0.6 $. 
    Overlaps with the first signal $ x_1^*$ (left) and the second signal $ x_2^*$ (right),  computed via simulation (``sim.'') and the theoretical prediction (``pred.''),  are plotted as a function of the aspect ratio $ \delta =n/d$. The signal dimension is $d=2000$.
 We note that our optimal spectral estimator enables weak recovery of both signals at a smaller $\delta$ than existing spectral estimators designed for non-mixed data. 
    E.g., in the right panel, our optimal spectral estimator starts weakly recovering $ x_2^* $ when $ \delta > 3.1 $ while other spectral estimators require at least $ \delta > 5.3 $. 
}
    \label{fig:all}
\end{figure}

\vspace{.5em}

\paragraph{Proof techniques} We exploit a combination of tools from free probability, random matrices and the theory of approximate message passing (AMP).  Generalized approximate message passing (GAMP) refers to a family of iterative algorithms \cite{RanganGAMP} with the following key feature: 
the joint distribution of the iterates is accurately tracked by a simple deterministic recursion, called \emph{state evolution}. Our strategy to obtain the joint distribution of the linear/spectral estimators and the signals in the master theorem (\Cref{thm:main-thm-joint-dist}) is to design a GAMP that \emph{(i)} outputs the linear estimator as the first iterate and \emph{(ii)} then implements a power method, so that its fixed point corresponds to the spectral estimator. One challenge in the implementation of this strategy is that the state evolution of GAMP, in its original form for vanilla (non-mixed) GLMs, only records the correlation of its iterates with a single signal. To circumvent this issue, we equip GAMP with a state evolution recursion involving both signals, and run a \emph{pair} of GAMP iterations converging to the first and second top eigenvector of the spectral matrix $D$ in \Cref{eqn:def-mtx-t-and-d-intro}, respectively. A second, even more fundamental, challenge is that, for the power method to converge to the desired eigenvector, a spectral gap between the corresponding eigenvalue and the rest of the spectrum is required. For non-mixed GLMs, the spectral analysis was carried out in earlier work \cite{lu2020phase,mondelli-montanari-2018-fundamental}, which  characterized the limiting  eigenvalues of $D$ as well as the overlaps  using tools from random matrix theory. Here, the difficulty comes from the mixed effect of the model, leading to additional matrix terms which appear challenging to control.
Our approach is to decompose $D$ into the sum of two matrices, $D_1$ and $D_2$, each consisting of components only pertaining to the first and second signal, respectively. 
Now, $D_1,D_2$ can be individually viewed as generated from a non-mixed GLM, 
hence their 
limiting spectra are well understood.
The key observation is then that, by assuming both signals to be independent and uniformly distributed on the sphere, $D_1$ and $D_2$ become \emph{asymptotically free}\footnote{Asymptotic freeness can be thought of as the random matrix analogue of \emph{independence} of random variables.}. Thus, we are able to characterize the sum of these two spiked matrices by using techniques from free probability. 

\subsection{Related work}\label{subsec:rel}


 Mixtures of generalized linear models have been studied in machine learning  as `hierarchical mixtures of experts' 
 \cite{Jor04}. Bayesian methods for this problem
 were investigated in \cite{Pen96,  Wat95,Vie02}. 
Khalili and Chen \cite{Kha07} proposed a penalized likelihood approach for variable selection in mixed GLMs, 
showing consistency 
in the low-dimensional setting (where the dimension $d$ is fixed as $n$ grows). St\"{a}dler et al.\ \cite{Sta10} analyzed $\ell_1$-penalized estimators for high-dimensional mixed linear regression (MLR). Zhang et al.\ \cite{Zha20} considered MLR 
with two sparse components, 
when the mixing proportion and the covariance structure of the covariates are unknown.  The works \cite{Kha07, Sta10, Zha20} all use variants of the EM algorithm for optimizing a suitable penalized likelihood function.
Balakrishnan et al.\ \cite{Bal17} and Klusowski et al.\ \cite{Klu19} obtained statistical guarantees on the performance of EM for a class of problems, including
symmetric MLR with $x_1^* = - x_2^*$. Variants of 
EM 
for symmetric MLR 
were also analyzed in \cite{Wan15,Yi15,Zhu17}. Minimax lower bounds 
were obtained in \cite{Fan18}, and statistical-computational gaps 
were recently studied in \cite{arpino2023statistical}.
Kong et al.\ \cite{Kon20} studied MLR as a canonical example of meta-learning: in the setting where the number of signals ($\ell$) is large,  they derived conditions under which a large number of signals with a few observations can compensate for the lack of signals with abundantly many observations. 
The prediction error of MLR in the non-realizable setting, where no generative model is assumed for the data, was studied in \cite{Pal22}.
Chandrasekher et al.\ \cite{CPT21} analyzed the performance of 
iterative algorithms (not including AMP) for mixtures of GLMs. They provide a sharp characterization of the  per-iteration error with sample-splitting in the regime $n \asymp d\, \mathrm{polylog}(d)$, assuming  a Gaussian design and a random initialization. 
An AMP estimator for mixed GLMs was recently studied in  \cite{tan2023mixed}. We emphasize that the focus of the current paper is not on using the AMP algorithm as an estimator for mixed GLMs. Rather, we use AMP as a proof technique to obtain a precise distributional characterization of the spectral estimator, and use this characterization to optimize its accuracy.

Spectral methods based on \Cref{eqn:def-mtx-t-and-d-intro} were introduced in \cite{li1992principal} for standard  GLMs  (non-mixed, with $\ell=1$). For the special case of phase retrieval, a series of works has provided increasingly refined bounds on the number of samples needed  to guarantee signal recovery via the spectral method \cite{netrapalli2013phase,candes2013phaselift,chen2015solving}. This type of analysis is based on matrix concentration inequalities, a technique that typically does not return exact values for the overlap between the signal and the estimate.
More recently, an exact high-dimensional analysis for generalized linear models 
was carried out in \cite{lu2020phase,mondelli-montanari-2018-fundamental}. These works focus on the regime of interest in this paper: $n$ and $d$ growing at a proportional rate $\delta$. 
This sharp analysis allows for the optimization of the preprocessing function: the choice of $\cT$ minimizing the value of $\delta$ (and, hence, the amount of data) needed to achieve a strictly positive overlap was provided in \cite{mondelli-montanari-2018-fundamental}; furthermore, the choice of $\cT$ maximizing the overlap was provided in \cite{luo-2019-opt-preprocessing}. Going beyond the proportional regime in which $n$ is linear in $d$, bounds on the sample complexity required for moment methods (including spectral) to achieve non-vanishing overlap were recently  obtained in \cite{damian2024computational}. The aforementioned analyses 
assume a Gaussian design matrix.  
Beyond this assumption, \cite{dudeja2020analysis} provides precise asymptotics for design matrices sampled from the Haar distribution, and \cite{maillard2020construction} studies rotationally invariant designs.
Moving to the mixed regression setting ($\ell>1$), 
Yi et al.\ \cite{yi-2014-mixed-linear-regression} proposed a spectral estimator based on \Cref{eqn:def-mtx-t-and-d-intro} with $\cT(y) = y^2$. The analysis is based on concentration inequalities and requires the number of samples $n$ to be of order $d\log d$ for accurate recovery.  Estimators based on spectral decomposition of data-dependent tensors were proposed for MLR in \cite{Cha13} and for mixed GLMs in \cite{sedghi2016provable}. However, these methods require $n$ to be of order at least $d^3$ for accurate recovery. Our work is the \emph{first} to establish exact asymptotics for a mixed GLM in the linear sample-size regime: $n, d\to\infty$ with $n/d\to\delta\in (0, \infty)$. To achieve this goal, our strategy differs from analyses of spectral methods in the non-mixed setting \cite{lu2020phase,mondelli-montanari-2018-fundamental} which reduce the study of the spectrum of $D$ to that of a rank-1 perturbation. In contrast, our analysis is based on a combination of techniques from free probability and approximate message passing (AMP).



AMP is a family of iterative algorithms that has been applied to several problems in high-dimensional statistics, including estimation in linear models \cite{DMM09, BM-MPCS-2011, krzakala2012}, generalized linear models \cite{RanganGAMP,schniter2014compressive,sur2019modern}, and low-rank matrix estimation \cite{deshpande2014information, RanganFletcherGoyal, lesieur2017constrained}, see also the review 
\cite{Feng22AMPTutorial}. A key feature of AMP algorithms is that under suitable model assumptions, the empirical joint distribution of their iterates can be exactly characterized in the high-dimensional limit, in  terms of a simple scalar recursion called \emph{state evolution}. By taking advantage of this characterization, AMP methods have been used to derive exact high-dimensional asymptotics for convex penalized estimators such as LASSO \cite{BayatiMontanariLASSO}, M-estimators \cite{donoho2016high}, logistic regression \cite{sur2019modern}, and SLOPE \cite{bu2020algorithmic}. AMP algorithms have been initialized via spectral methods in the context of low-rank matrix estimation \cite{montanari2017estimation} and generalized linear models \cite{mondelli-2021-amp-spec-glm}. Furthermore, they have been used -- in a non-mixed setting -- to combine linear and spectral estimators \cite{mondelli2021optimalcombination}. 
A finite-sample analysis which allows the number of iterations to grow roughly as $\log n$ ($n$ being the ambient dimension) was put forward in \cite{rush-venkat-finite-sample}, and the recent paper \cite{li2022non} improves this guarantee to a linear (in $n$) number of iterations. This could potentially allow to study settings in which $\delta=n/d$ approaches the spectral threshold. The works on AMP discussed above  all assume i.i.d. Gaussian matrices. A number of recent papers have proposed generalizations of AMP for the much broader class of rotationally invariant matrices, e.g., \cite{opper2016theory,ma2017orthogonal, rangan2019vector, takeuchi2020rigorous,Zho21,Fan22,mondelli2021pca,venkataramanan2021estimation}.

\hl{Finally, we mention the recent paper \mbox{\cite{Kovacevic}} 
that derived precise asymptotics of spectral estimators for multi-index models by generalizing the techniques in \mbox{\cite{lu2020phase,mondelli-montanari-2018-fundamental}}.
However, \mbox{\cite{Kovacevic}} did not derive the \emph{joint} distribution of spectral and linear estimators, or the (optimal) combination of the two. 
To the best of our knowledge, such results do not follow immediately from the pure random matrix theory-type results in \mbox{\cite{Kovacevic}}, but require additional work to handle the correlation between spectral and linear estimators. 
These are achieved in the present work using a mix of tools from free probability and Approximate Message Passing (AMP).}

\section{Preliminaries}
\label{sec:prelim}

The $i$-th element in $a\in\bbR^p$ is denoted by $a_i$. 
If a vector has multiple subscripts, the component index is the last one.  For a symmetric $ M\in\bbR^{p\times p} $, we denote by $ \mu_M $ its empirical spectral distribution. 
The (real) eigenvalues of $M$ are $ \lambda_1(M)\ge\lambda_2(M)\ge\cdots\ge\lambda_p(M) $, and the corresponding eigenvectors 
are 
$ v_1(M),v_2(M),\cdots,v_p(M) $. 
The $(i,j)$-th entry of $M$ is denoted by $ M_{i,j} $. 
For a random variable $X$, $ \supp(X) $ denotes the support of its density function. 
The orthogonal group in dimension $p$ is $ \bbO(p)\coloneqq\curbrkt{O\in\bbR^{p\times p} : OO^\top = O^\top O = I_p} $. 
The unit sphere in dimension $p$ is $ \bbS^{p-1} \coloneqq \curbrkt{x\in\bbR^p : \normtwo{x} = 1} $. 
For two distributions $P$ and $Q$, $P\ot Q$ is their product distribution, and 
$P^{\ot k}$ is the $k$-fold product distribution of $P$. 

\vspace{.5em}

\paragraph{Model}
We consider a two-component 
mixed GLM with signal vectors 
$ x_1^*,x_2^*\in\bbS^{d-1} $, covariate vectors 
$ a_1,\cdots,a_n\in\bbR^d $ , and a known link function $ q\colon\bbR^2\to\bbR $.
 Let $ P_\eps $ be a noise distribution over $ \bbR $. 
The $n$ observations $ y_1,\cdots,y_n\in\bbR $ are generated as:
\begin{align}
y_i &= q\paren{\inprod{a_i}{\eta_ix_1^* + (1-\eta_i) x_2^*} ,\, \eps_i}, \qquad i \in [n] .  \label{eqn:def-y-glmm} 
\end{align}
Here, the vector of latent variables $\veta \coloneqq (\eta_1,\cdots,\eta_n) \sim \bern(\alpha)^\tn $ indicates which signal is selected by each observation, and is \emph{unobserved}. The latent variable vector $\veta$, the signals $x_1^*, x_2^*$, the covariate vectors $a_1, \ldots, a_n$, and 
the noise vector $\veps \coloneqq (\eps_1,\cdots,\eps_n)\sim P_\eps^\tn$ are mutually independent. Then, \Cref{eqn:def-y-glmm} is equivalent to
\begin{equation}
    y_i\mid \inprod{a_i}{\eta_ix_1^* + (1-\eta_i) x_2^*} \sim p(\cdot \mid \inprod{a_i}{\eta_ix_1^* + (1-\eta_i) x_2^*}),\label{eq:condlaw}
\end{equation}
where $p(\cdot |g)$ denotes the distribution of $q(g, \eps)$ for a fixed $g\in\mathbb R$ and $\eps\sim P_\eps$ independent of $g$. The design matrix 
is $A=[a_1^\top, \ldots, a_n^\top]^\top\in\bbR^{n\times d}$. 
Given $A$, upon observing $ y=( y_1,\cdots,y_n)\in\bbR^n $, our goal is to estimate $ x_1^* $ and $ x_2^* $. 
Given a pair of estimators $\wh{x}_1 = \wh{x}_1(y,A), \, \wh{x}_2 = \wh{x}_2(y,A)\in\bbR^d$, we measure the  performance via their \emph{overlap} with the respective signals: 
\begin{align}
    \lim_{d\to\infty} \frac{\abs{\inprod{\wh{x}_1}{x_1^*}}}{\normtwo{\wh{x}_1}\normtwo{x_1^*}} , \quad 
    \lim_{d\to\infty} \frac{\abs{\inprod{\wh{x}_2}{x_2^*}}}{\normtwo{\wh{x}_2}\normtwo{x_2^*}} . \notag
\end{align}
%
Throughout the paper, the following assumptions are imposed. 

\vspace{.25em}

\begin{enumerate}[label=(A\arabic*)]
\setcounter{enumi}{\value{assctr}}
    \item \label[ass]{itm:assump-signal-distr} $ x_1^*,x_2^* $ are independent and uniform on the unit sphere, $ (x_1^*,x_2^*)\sim\unif(\bbS^{d-1})^{\ot2} $. 
\vspace{.25em}
    \item \label[ass]{itm:assump-alpha} $ \alpha\in (1/2,1) $. 
\vspace{.25em}
    \item \label[ass]{itm:assump-noise-distr} The noise sequence $\veps\in\bbR^n$ is i.i.d.\ according to $ \veps \sim P_\eps^{\ot n} $, and $P_\eps$ has finite second moment. 
\vspace{.25em}
    \item \label[ass]{itm:assump-gaussian-design} $ a_1,\cdots,a_n\in\bbR^d $ are i.i.d., each distributed according to $ a_i\iid\cN(0_d,I_d) $. 
\vspace{.25em}
    \item \label[ass]{itm:assump-proportional} We consider the proportional regime where $ n,d\to\infty $ and $ n/d\to\delta $ for some constant $ \delta> 0 $ which we call \emph{aspect ratio}. 
\setcounter{assctr}{\value{enumi}}
\end{enumerate}

\vspace{.25em}

As for \Cref{itm:assump-signal-distr}, 
choosing signals uniform 
on the sphere corresponds to having 
no structural information about them. 
This requirement is natural, since 
spectral methods 
are typically unable to exploit prior information about the signal. 
\hl{We expect that all results of the present paper hold for the more relaxed setting where the signals $ x_1^*, x_2^* $ are independent of the design matrix $A$ and the noise vector $\veps$, and satisfy }
\begin{align}
&&
    \lim_{d\to\infty} \normtwo{x_1^*} = \lim_{d\to\infty} \normtwo{x_2^*} &= 1 , & 
    \lim_{d\to\infty} \inprod{x_1^*}{x_2^*} &= 0 . & 
& \notag 
\end{align}
\hl{In particular, we believe that the assumption that the  signals are uniformly distributed on the sphere is not required. 
In that case, \mbox{\Cref{eqn:use-indep-signal-tmp}} in the argument of the reduction to the free sum of two random matrices is no longer an exact equality (in distribution). 
However, we still expect asymptotic freeness to hold in the proportional limit. 
We leave its formal justification to future work. }
Understanding the effect of correlation on the performance of spectral estimators and the design of the optimal preprocessing function is an exciting 
future direction. 
\Cref{itm:assump-alpha} 
is without loss of generality: 
if $0<\alpha<1/2$, one can simply interchange the roles of $x_1^*$ and $x_2^*$. When $\alpha = 1/2$, 
the top two eigenvectors given by the spectral method correspond to the same limiting eigenvalue as $n\to \infty$. These eigenvectors provide an estimate on the space spanned by $x_1^*, x_2^*$, and in order to estimate the individual signals, an additional $1$-dimensional grid search is required. Provided this extra step is carried out, our results 
still apply, 
see \Cref{rk:alpha-half-master,rk:alpha-half-rmt,rk:alpha-half-amp}. 
\Cref{itm:assump-gaussian-design} is common in the related literature 
\cite{yi-2014-mixed-linear-regression,mondelli-montanari-2018-fundamental,lu2020phase,luo-2019-opt-preprocessing}, and the potential universality beyond 
Gaussian design matrices is discussed in \Cref{sec:discussion}.

\vspace{.5em}




\paragraph{Linear estimator}
Given the preprocessing function $ \cL\colon\bbR\to\bbR $, 
the \emph{linear estimator} is 
\begin{align}
    \xlin &\coloneqq \frac{1}{n} A^\top \cL(y) = \frac{1}{n} \sum_{i = 1}^n \cL(y_i) a_i \in\bbR^d , \label{eqn:def-lin-estimator}
\end{align}
where $\cL$ is applied component-wise, 
i.e., $\cL(y) = (\cL(y_1),\cdots,\cL(y_n))$. Let 
$ Y $ be defined as  
\begin{align}
    Y=q(G,\eps), \quad \text{ where } (G,\eps)\sim\cN(0,1)\otimes P_\eps . \label{eqn:def-rv-y}
\end{align}
We make the following assumption on 
$\cL$. 

\vspace{.25em}

\begin{enumerate}[label=(A\arabic*)]
\setcounter{enumi}{\value{assctr}}
    \item \label[ass]{itm:assump-preproc-lin} $ \cL\colon\bbR\to\bbR $ is Lipschitz and satisfies 
    \begin{align}
        \expt{G\cL(Y)} &\ne 0 , \quad 
        \expt{|G\cL(Y)|} < \infty . \notag 
    \end{align}
\setcounter{assctr}{\value{enumi}}
\end{enumerate}

The first condition guarantees that the linear method w.r.t.\ $\cL$ attains positive overlaps with both signals, and the second condition is rather mild and purely technical. 

\vspace{.5em}

\paragraph{Spectral estimator}
Let $ \cT\colon\bbR\to\bbR $ be a preprocessing function, and consider
\begin{align}
    T &\coloneqq \diag(\cT(y)) \in\bbR^{n\times n} , \quad
    D \coloneqq \frac{1}{n} A^\top TA = \frac{1}{n} \sum_{i = 1}^n \cT(y_i) a_ia_i^\top \in\bbR^{d\times d} , \label{eqn:def-mtx-t-and-d} 
\end{align}
where 
$\cT(y) = (\cT(y_1),\cdots,\cT(y_n))$. Then, the spectral method computes the top two eigenvectors $ v_1(D),v_2(D) $ of $D$ as estimates of $ x_1^*,x_2^* $. 
We make the following assumption on 
$\cT$. 

\vspace{.25em}

\begin{enumerate}[label=(A\arabic*)]
\setcounter{enumi}{\value{assctr}}
    \item \label[ass]{itm:assump-preproc-spec}  Let $Y$ be defined in \Cref{eqn:def-rv-y}. Then,  $\cT(Y)$ is not almost surely zero, i.e.,  $ \prob{\cT(Y) = 0}<1 $, 
$\cT$ is Lipschitz and satisfies  
    \begin{align}
        \inf_{y\in\supp(Y)} \cT(y)  > -\infty , \quad  \text{and} \quad
        0 < \sup_{y\in\supp(Y)} \cT(y) < \infty . \notag
    \end{align}
\setcounter{assctr}{\value{enumi}}
\end{enumerate}

\vspace{.25em}

In words, we 
require $\cT$ to be bounded, with strictly positive 
upper edge of its range. A bounded preprocessing function is also required in the non-mixed setting \cite{mondelli-montanari-2018-fundamental,lu2020phase}. The requirement on the $\sup$ to be strictly positive is purely technical, and it simply  
 rules out the trivial cases in which the spectral matrix $D$ is all-zero with high probability. 
 \Cref{itm:assump-preproc-spec} is satisfied by the preprocessing function that maximizes the overlap (cf.\ \Cref{thm:opt-spec}). 



\section{Main results}
\label{sec:results}


We start by defining a few auxiliary quantities. 
Let $ \delta_1 = \alpha\delta,\delta_2 = (1-\alpha)\delta $, and $Z = \cT(Y)$, with $Y$ as defined in \Cref{eqn:def-rv-y}.  Define $ \phi\colon(\sup\supp(Z),\infty)\to\bbR $ and $ \psi\colon(\sup\supp(Z),\infty)\times(0,\infty)\to\bbR $ as 
\begin{align}
\phi(\lambda) &\coloneqq \lambda\,\expt{\frac{ZG^2}{\lambda - Z}} ,  \label{eqn:def-phi-fn} \\
\psi(\lambda;\Delta) &\coloneqq \lambda\paren{\frac{1}{\Delta} + \expt{\frac{Z}{\lambda - Z}}} . \label{eqn:def-psi-fn} 
\end{align}
In what follows, we will set the second argument $\Delta$ of $\psi$ to $\delta, \delta_1$ and $\delta_2$. For $ \Delta\in\{\delta, \delta_1, \delta_2\}$, let $ \ol\lambda(\Delta)>\sup\supp(Z) $ be the minimum point of $ \psi(\cdot;\Delta) $, i.e., 
\begin{align}
\ol\lambda(\Delta) &\coloneqq \argmin_{\lambda>\sup\supp(Z)} \psi(\lambda;\Delta) . \label{eqn:def-lam-bar} 
\end{align}
Since $\psi$ is convex in its first argument (see \Cref{lem:properties-phi-psi}), this minimum point is obtained by setting the derivative to $0$. Furthermore, define 
$\zeta\colon(\sup\supp(Z),\infty)\times(0,\infty)\to\bbR $ as
\begin{align}
\zeta(\lambda; \Delta) &\coloneqq \psi(\max\{\lambda,\ol\lambda(\Delta)\}; \Delta) . \label{eqn:def-zeta-i}
\end{align}
Finally, for $i\in\{1, 2\}$, by \cite[Lemma 2]{mondelli-montanari-2018-fundamental}, the equation $ \zeta(\lambda;\delta_i) = \phi(\lambda) $ admits a unique solution in $ \lambda\in(\sup\supp(Z),\infty) $ which we call $ \lambda^*(\delta_i) $. \label{def:lambdastard}
The functions $ \psi(\lambda;\Delta),\phi(\lambda),\zeta(\lambda;\Delta) $ together with the parameters $ \lambda^*(\Delta),\ol\lambda(\Delta) $ are plotted in \Cref{fig:psi-phi-zeta} for $ \Delta\in\{\delta,\delta_1,\delta_2\} $. 
Some convexity and monotonicity properties 
of these functions can be found in \Cref{lem:properties-phi-psi}. 

\begin{figure}[htbp]
    \centering
    \includegraphics[width=0.5\linewidth]{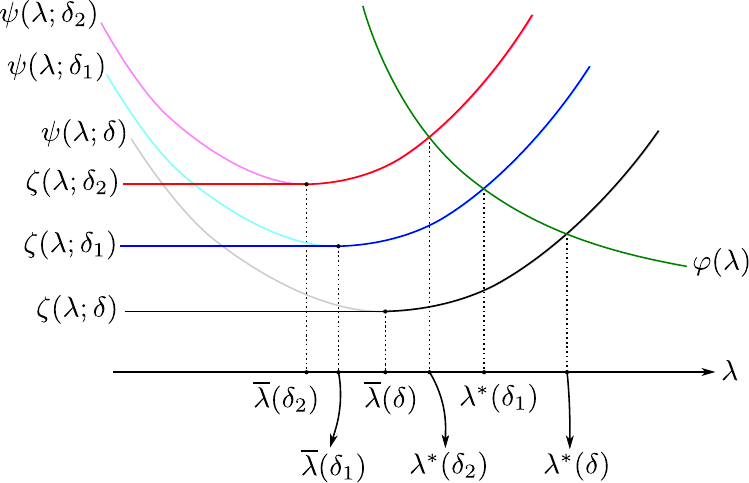}
    \caption{Plot of $ \psi(\lambda;\Delta),\phi(\lambda),\zeta(\lambda;\Delta) $ as functions of $\lambda$ with $ \Delta\in\{\delta,\delta_1,\delta_2\} $. }
    \label{fig:psi-phi-zeta}
\end{figure}

The empirical distribution of  a vector $u \in \reals^d$ is given by $ \frac{1}{d}\sum_{i=1}^d \delta_{u_i}$, where $\delta_{u_i}$ denotes a Dirac delta mass on $u_i$. Similarly, the joint empirical  distribution of the rows of a matrix $(u^1, u^2, \ldots, u^t) \in \reals^{d \times t}$ is $\frac{1}{d} \sum_{i=1}^d \delta_{(u^1_i, \ldots, u^t_i)}$. Our master theorem is an exact characterization in the high-dimensional limit of the joint empirical distribution of the rows of the signals, the linear estimator, and the spectral estimators. In particular, we show that this joint empirical distribution converges to the law of a Gaussian random vector with a specified covariance matrix. The result is stated in terms of the  following parameters: the asymptotic correlations $\rho_1^{\lin}, \rho_2^{\lin}$ between the linear estimator and the two signals;  the asymptotic normalized Euclidean norm $n^{\lin}$  of the linear estimator; and the asymptotic correlations $\rho_1^{\spec}, \rho_2^{\spec}$ between the spectral estimators and the two signals. The formulas for these quantities are:
\begin{subequations}
\begin{align}
& n^{\lin} \coloneqq \bigg( (\alpha^2+(1-\alpha)^2) \expt{G\cL(Y)}^2 + \frac{\expt{\cL(Y)^2}}{\delta} \bigg)^{\frac{1}{2}}, \\
& \rho_1^{\lin} \coloneqq \frac{\alpha \expt{G\cL(Y)}}{n^{\lin}}, \quad \rho_2^{\lin} \coloneqq  \frac{ (1-\alpha) \expt{G\cL(Y)}}{n^{\lin}}, \label{eqn:def-rho-lin-12} \\
\begin{split}
& \rho_1^{\spec} \coloneqq \paren{ \frac{\frac{1}{\delta} - \expt{\paren{\frac{Z}{\lambda^*(\delta_1) - Z}}^2}}{\frac{1}{\delta} + \alpha\expt{\paren{\frac{Z}{\lambda^*(\delta_1) - Z}}^2(G^2 - 1)}} }^{\frac{1}{2}} , \\
&\rho_2^{\spec} \coloneqq \paren{ \frac{\frac{1}{\delta} - \expt{\paren{\frac{Z}{\lambda^*(\delta_2) - Z}}^2}}{\frac{1}{\delta} + (1-\alpha)\expt{\paren{\frac{Z}{\lambda^*(\delta_2) - Z}}^2(G^2 - 1)}} }^{\frac{1}{2}} . 
\end{split}
\label{eqn:def-rho-spec-12}
\end{align}
\label{eqn:def-asymp-corr}
\end{subequations}
\Cref{thm:main-thm-joint-dist} is stated in terms of \emph{pseudo-Lipschitz} test functions.
%
A function $\Psi: \reals^m \to \reals$ is pseudo-Lipschitz of order $k \geq 1$, denoted $\Psi \in \PL(k)$, if there is a constant $C > 0$  such that 
\begin{equation}
\normtwo{\Psi(x)-\Psi(y)} \le C (1 + \| x \|_2^{k-1} + \| y \|_2^{k-1} )\normtwo{x - y},
\label{eq:PL2_prop} 
\end{equation}
for all $x,y \in \mathbb{R}^m$. Examples of pseudo-Lipschitz functions of order two are: $\Psi(u)=u^2$ and $\Psi(u,v) = \abs{uv}$, for $u, v \in \reals$.
We consider pseudo-Lipschitz test functions of order two, as 
those suffice to compute the asymptotic overlaps between the signals and the various estimators. 
One could extend \Cref{thm:main-thm-joint-dist} to test functions in $\PL(k)$ for $k>2$, at the cost of a more involved argument and an additional assumption on the finiteness of the moments of $P_\eps$.


\begin{theorem}[Master theorem on joint distribution]
\label{thm:main-thm-joint-dist}
Consider the setting of \Cref{sec:prelim}, and let \Cref{itm:assump-signal-distr,itm:assump-alpha,itm:assump-noise-distr,itm:assump-gaussian-design,itm:assump-proportional,itm:assump-preproc-lin,itm:assump-preproc-spec} hold.
Define the following rescaled vectors of Euclidean norm $\sqrt{d}$: $x^{\lin} = \sqrt{d} \, \xlin/\normtwo{\xlin}$, and for $i\in \{1, 2\}$,
$\ol{x}^*_i = \sqrt{d} x_i^*$, $x^{\spec}_i=s_i\sqrt{d} \, v_i(D)$, 
where the sign $ s_i\in\{-1,1\} $ is chosen such that $ \inprod{s_iv_i(D)}{x_i^*}\ge0$.
Then, the following holds almost surely for any $\PL(2)$ function $\Psi: \reals^3 \to \reals$. 
If $ \lambda^*(\delta_1) > \ol\lambda(\delta) $, then
\begin{align}
& \lim_{d \to \infty} \,   \frac{1}{d} \sum_{i=1}^d \Psi(\ol{x}^*_{1,i}, x^{\lin}_i,  x^{\spec}_{1,i}) = \expt{\Psi(X_1,
\, \rho_1^{\lin} X_1 + \rho_2^{\lin} X_2 + W^{\lin}
,\,  \rho_1^{\spec} X_1 + W_1^{\spec}) }.  \label{eq:psiX1joint} 
\end{align}
Similarly, if $ \lambda^*(\delta_2) > \ol\lambda(\delta) $, then
\begin{align}
& \lim_{d \to \infty} \,   \frac{1}{d} \sum_{i=1}^d \Psi(\ol{x}^*_{2,i}, x^{\lin}_i,  x^{\spec}_{2,i}) = \expt{ \Psi(X_2,
\, \rho_1^{\lin} X_1 + \rho_2^{\lin} X_2 + W^{\lin}
,\,  \rho_2^{\spec} X_2 + W_2^{\spec}) }. \label{eq:psiX2joint}    
\end{align}
Here $(X_1, X_2) \sim \cN(0,1)^{\ot2}$, the pairs $(W^{\lin}, W_1^{\spec})$ and $(W^{\lin}, W_2^{\spec})$ are independent of $(X_1, X_2)$ and each pair is 
jointly Gaussian with zero mean and covariance given by 
\begin{equation}
    \begin{split}
        & \E[(W^{\lin})^2]=1-(\rho_1^{\lin})^2-(\rho_2^{\lin})^2, \quad \E[(W_1^{\spec})^2]=1-(\rho_1^{\spec})^2, \quad \E[(W_2^{\spec})^2]=1-(\rho_2^{\spec})^2 , \\ 
        & \E[W^{\lin} W_1^{\spec}] = \frac{\alpha \rho_1^{\spec}}{n^{\lin}} \expt{  \frac{ G\cL(Y) Z}{\lambda^*(\delta_1) -Z}},  \quad  
        \E[W^{\lin} W_2^{\spec}] = \frac{(1-\alpha) \rho_2^{\spec}}{n^{\lin}} \expt{ \frac{G\cL(Y) Z}{\lambda^*(\delta_2) -Z}}.
    \end{split} \notag
\end{equation} 
\end{theorem}
The outline of the argument is presented in \Cref{sec:sum-techniques}.
The full proof is given in \Cref{sec:overlap-gamp}, and it relies on the characterization of the eigenvalues of 
$D$ carried out in \Cref{thm:main-thm-eigval}, which is stated and proved in \Cref{sec:eigval-rmt}.  

\begin{remark}[Equivalence to convergence of empirical distribution]
\Cref{eq:psiX1joint} is equivalent to the statement that the joint empirical distribution of $(\ol{x}^*_{1}, x^{\lin},  x^{\spec}_{1})$ converges in Wasserstein-2 distance to the joint law of $(X_2, \, \rho_1^{\lin} X_1 + \rho_2^{\lin} X_2 + W^{\lin},\,  \rho_1^{\spec} X_1 + W_1^{\spec})$. 
The equivalence between convergence of empirical distributions in Wasserstein distance and convergence of  empirical averages of pseudo-Lipschitz functions is proved in \cite[Corollary 7.21]{Feng22AMPTutorial}.
\end{remark}

\begin{remark}[What if either the linear or spectral estimator is ineffective]
\label{rk:one-fails}
The validity of the description of the joint law of the first 
signal and the linear/spectral estimators in \Cref{eq:psiX1joint}  
relies on two assumptions:  $\expt{G\cL(Y)}\ne0 $ for the linear estimator, and $ \lambda^*(\delta_1)>\ol\lambda(\delta) $ for the spectral one. 
They guarantee that both estimators achieve non-zero asymptotic overlaps with $x_1^*$, i.e., 
$\rho_1^{\lin}\ne0$ and $ \rho_1^{\spec}>0 $. 
If either condition is not satisfied, 
a conclusion similar to \Cref{eq:psiX1joint} 
still holds with $\Psi\colon\bbR\times\bbR\to\bbR$ only taking $x_1^*$ 
and the non-trivial estimator as inputs. 
Specifically, if only the linear estimator is effective, then we terminate GAMP in \Cref{eq:gamp-eqn-general-tmp} after one step $t=0$ and obtain the distributional characterization; 
if only the spectral estimator is effective, then the initializer in \Cref{eq:F1x1-tmp} ensures that the same proof goes through without modifications, again leading to the desired conclusion. 
An analogous argument holds for the second signal.  
\end{remark}

\begin{remark}[$ v_i(D) $ estimates $ x_i^* $]
\hl{We have been using $ v_i(D) $ to estimate $ x_i^* $, for $ i\in\{1,2\} $. 
In fact, $ v_1(D) $  is asymptotically uncorrelated with $ x_2^* $ (and $ v_2(D) $  asymptotically uncorrelated with $ x_1^* $), according to the characterization of the asymptotic distribution of the top two eigenvectors in \mbox{\Cref{sec:GAMP_as_method}}. 
Intuitively, this phenomenon arises due to the orthogonality between the two signals, and fails to hold otherwise. 
For instance, if $ \lim_{d\to\infty} \inprod{x_1^*}{x_2^*} = \rho \ne 0 $, then both the first and second eigenvectors are asymptotically correlated with both signals. 
This can be formally justified by specializing \mbox{\cite[Theorem 4.2]{Kovacevic}} to mixtures of GLMs and is numerically corroborated in \mbox{\cite[Figure 2]{Kovacevic}}. }
\end{remark}


\begin{remark}[Sign calibration of spectral estimator]
\label{rk:sign-spectral}
As the eigenvectors of a matrix are insensitive to sign flip, the spectral estimators $ x^{\spec}_1,x^{\spec}_2 $ are defined up to a change of sign. In \Cref{thm:main-thm-joint-dist}, we pick the signs so that the resulting overlaps $\rho_1^{\spec}, \rho_2^{\spec}$ are positive. In practice, there is a simple way to resolve the sign ambiguity: one can match the sign of $\expt{(\rho_1^{\lin} X_1 + \rho_2^{\lin} X_2 + W^{\lin})\, (\rho_i^{\spec} X_i + W_i^{\spec})}$ with that of the scalar product $\inprod{x^{\lin}}{x^{\spec}_i}$, as the latter can be computed empirically  (without knowing $ x_1^*,x_2^* $).
\end{remark}

\begin{remark}[Master theorem for $\alpha = 1/2$]
\label{rk:alpha-half-master}
Even though we 
assume $ \alpha\in(1/2,1) $  (see \Cref{itm:assump-alpha}), 
the conclusion of \Cref{thm:main-thm-joint-dist} still holds for $ \alpha = 1/2 $ with a slight modification in the definition of the spectral estimators.
In this case, as $n\to \infty$ the top two eigenvectors given by the spectral method correspond to the same limiting eigenvalue. These eigenvectors, $v_1(D)$ and $v_2(D)$, estimate the subspace spanned by $x_1^*, x_2^*$. To estimate each individual signal, we search for a vector in $ \spn\curbrkt{v_1(D),v_2(D)} $ whose correlation with $ x^{\lin} $ is closest to the theoretical prediction from \Cref{thm:main-thm-joint-dist}.  Indeed, let $ x_1^{\spec},x_2^{\spec} $ be defined as
\begin{align}
    x_i^{\spec} &\coloneqq \argmin_{v\,\in\,\spn\curbrkt{v_1(D),v_2(D)}\cap\sqrt{d}\,\bbS^{d-1}} \abs{\frac{\inprod{v}{x^{\lin}}}{\sqrt{d}} - \paren{\rho_i^{\lin} \rho_i^{\spec} + \expt{W^{\lin} W_i^{\spec}}}} , \quad \text{for } i\in\{1,2\} . \label{eqn:xspec-alpha-half} 
\end{align}
Then, \Cref{eq:psiX1joint,eq:psiX2joint} hold, provided $ \expt{G\cL(Y)} \ne 0 $ (which guarantees that the linear estimator attains nonzero overlaps; see \Cref{itm:assump-preproc-lin} and \Cref{eqn:def-rho-lin-12}). 
We stress that \Cref{eqn:xspec-alpha-half} is  computable  in practice since it only involves $x^{\lin}$ and theoretical predictions. 
If $x^{\lin}$ is ineffective (which is the case, for example, in mixed phase retrieval, as mentioned in \Cref{sec:results-examples}), 
a similar grid search can still be performed if the statistician is given as side information a vector with known correlation with a signal. 
The reader is referred to \Cref{rk:alpha-half-rmt,rk:alpha-half-amp} for the adaptation of our proofs to the case $\alpha=1/2$. 
\end{remark}



Equipped with \Cref{thm:main-thm-joint-dist}, we can  combine the linear and spectral estimators to improve the performance in the recovery of 
$x_1^*$ and $x_2^*$. 
Formally, consider the (rescaled) linear and spectral estimators $ x^{\lin}\in\sqrt{d}\,\bbS^{d-1} $ and $ x_1^{\spec},x_2^{\spec}\in\sqrt{d}\,\bbS^{d-1} $. 
Define 
\begin{align}
    X^{\lin}\coloneqq\rho_1^{\lin} X_1 + \rho_2^{\lin} X_2 + W^{\lin} , \quad
    X_1^{\spec}\coloneqq\rho_1^{\spec} X_1 + W_1^{\spec} , \quad 
    X_2^{\spec}\coloneqq\rho_2^{\spec} X_2 + W_2^{\spec} . 
    \label{eqn:def-rv-xlin-xspec}
\end{align}
\Cref{thm:main-thm-joint-dist} states that  the joint  empirical distribution of the estimators $(x^{\lin},x_1^{\spec},x_2^{\spec})$ converges to the  law of  $(X^{\lin}, X_1^{\spec}, X_2^{\spec} )$. 
For $i\in\{1,2\}$, define the set of functions
\begin{align}
    \cC_i &\coloneqq \curbrkt{C_i\colon\bbR\times\bbR\to\bbR \mbox{ s.t. } \expt{C_i(X^{\lin}, X_i^{\spec})^2}\in(0,\infty)} . \label{eqn:set-comb-fun-1} 
\end{align}
Then, for any $C_i\in\cC_i$, the \emph{combined estimator} $ x_i^{\comb}$ is defined as 
\begin{align}
    x^{\comb}_i \coloneqq C_i(x^{\lin}, x_i^{\spec}), 
    \label{eqn:def-combo-estimator} 
\end{align}
where $C_i$ acts on its inputs component-wise, i.e., $ x_{i,j}^{\comb} = C_i(x_j^{\lin}, x_{i,j}^{\spec}) $ for any $j\in[d]$. 
Now, \Cref{eq:psiX1joint} reduces the vector problem of estimating $x_i^*$ given $(x^{\lin}, x^{\spec}_i)$ to the scalar problem of estimating $X_i$ from $ X^{\lin} $ and $ X_i^{\spec} $.  The Bayes-optimal combined estimator that minimizes the expected squared error for this scalar problem is $\expt{X_i \mid X^{\lin}, X_i^{\spec}}$.  Recalling from \Cref{thm:main-thm-joint-dist} that $(X_i, X^{\lin}, X_i^{\spec})$ are jointly Gaussian, 
the Bayes-optimal combined estimator  is a  linear combination of $(X^{\lin}, X_i^{\spec})$. %
The performance of this combined estimator is  formalized in the following corollary, whose proof is given in \Cref{app:bayes-opt-comb}.



\begin{corollary}[Bayes-optimal linear-spectral combination]
\label{cor:opt-combo}
Consider the setting of \Cref{thm:main-thm-joint-dist}. 
For $i\in\{1,2\}$, define $ C_i^*\colon\bbR\times\bbR\to\bbR $ as follows:
\begin{align}
    C^*_i(X^{\lin}, X_i^{\spec}) &\coloneqq \expt{X_i \mid X^{\lin}, \, X_i^{\spec}} 
    = \frac{1}{1-\nu_i^2} \, \paren{\xi_i X^{\lin}  \, + \, \zeta_i X_i^{\spec}}, \label{eqn:opt-combinator-1} 
    %
\end{align} 
where
\begin{align}
    \nu_i \coloneqq \rho_i^{\lin} \rho_i^{\spec} + \expt{W^{\lin} W_i^{\spec}} , \quad  
    \xi_i \coloneqq \rho_i^{\lin} - \rho_i^{\spec}\nu_i , \quad 
    \zeta_i \coloneqq \rho_i^{\spec} - \rho_i^{\lin}\nu_i . \notag
\end{align}
For $ i\in\{1,2\} $, let $ x_i^{\comb} $ be the combined estimators defined in \Cref{eqn:def-combo-estimator} w.r.t.\ $ C_i^* $, respectively. 
Then, almost surely we have
\begin{align}
    \lim_{d\to\infty} \frac{\abs{\inprod{x_i^{\comb}}{x_i^*}}}{\normtwo{x_i^{\comb}}\normtwo{x_i^*}}
    &= \frac{1}{1 - \nu_i^2} \paren{\xi_i^2 + \zeta_i^2 + 2\xi_i\zeta_i\paren{\rho_i^{\lin}\rho_i^{\spec} + \expt{W^{\lin}W_i^{\spec}}}}^{1/2}
    \eqqcolon \mathsf{OL}_i^{\comb} . \notag 
\end{align}
Furthermore, for any $ (C_1,C_2)\in\cC_1\times\cC_2 $, the corresponding combined estimators $\wt{x}_1^{\comb},\wt{x}_2^{\comb}$ defined w.r.t.\ $ C_1,C_2 $ through \Cref{eqn:def-combo-estimator}, satisfy
\begin{align}
    \lim_{d\to\infty} \frac{\abs{\inprod{\wt{x}_i^{\comb}}{x_i^*}}}{\normtwo{\wt{x}_i^{\comb}}\normtwo{x_i^*}} &= \frac{\abs{\expt{X_i C_i(X^{\lin}, X_i^{\spec})}}}{\sqrt{\expt{C_i(X^{\lin}, X_i^{\spec})^2}}} \le \mathsf{OL}_i^{\comb} , \qquad i\in\{1,2\}. \notag 
\end{align}
\end{corollary}


\subsection{Linear estimator}
\label{sec:results-lin}

\Cref{thm:main-thm-joint-dist} allows us to derive the asymptotic overlap of  each signal with the linear estimator  in \Cref{eqn:def-lin-estimator}. 

\begin{corollary}[Overlaps, linear]
\label{lem:linear-overlap}
Consider the setting of \Cref{sec:prelim}, and let \Cref{itm:assump-signal-distr,itm:assump-alpha,itm:assump-noise-distr,itm:assump-gaussian-design,itm:assump-proportional,itm:assump-preproc-lin} hold.
Then,  almost surely, 
\begin{align}
    \lim_{d \to \infty} \, \frac{\inprod{\xlin}{x_i^*}}{\normtwo{\xlin}\normtwo{x_i^*}} & = \rho_i^{\lin} \, , \quad  i\in\{1,2\}. \label{eqn:overlap-lin-1-main} 
%
\end{align}
\end{corollary}

\begin{proof}
Choose $ \Psi(a,b,c) = ab $, and  
note that $ \Psi\in\PL(2) $. 
Then, as $\normtwo{\xlin} = \normtwo{\ol{x}^*_i}=  \sqrt{d}$,  the left side of \Cref{eq:psiX1joint,eq:psiX2joint} recovers the overlaps in \Cref{eqn:overlap-lin-1-main} for $i=1,2$, and 
the right sides of \Cref{eq:psiX1joint,eq:psiX2joint} become $ \rho_1^{\lin},\rho_2^{\lin} $ (defined in \Cref{eqn:def-rho-lin-12}).
\end{proof}

\begin{remark}[Overlap of linear estimator does not approach $1$]
\label{rk:lin-overlap-not-approach-1}
 From \Cref{eqn:overlap-lin-1-main} and the definitions of $\rho_1^{\lin},  \rho_2^{\lin} $ in \Cref{eqn:def-rho-lin-12}, we have that 
 the linear estimator 
achieves positive overlap with each signal for \emph{any} positive $\delta$, as long as $ \expt{G\cL(Y)}>0 $. As $\delta\to\infty$, the limiting overlaps  approach
    $\sqrt{\frac{\alpha^2}{\alpha^2 + (1-\alpha)^2}}$ and  
    $\sqrt{\frac{(1-\alpha)^2}{\alpha^2 + (1-\alpha)^2}}$, 
and they are strictly less than $1$ for any $ \alpha\in (1/2,1) $. In contrast,  the overlap of the spectral estimator becomes positive only when $\delta$ exceeds a certain threshold (see \Cref{rk:univ-lb-spec-thr}). However, once this threshold is exceeded, the  spectral estimator yields overlaps approaching $1$ as $\delta$ grows (see \Cref{rk:overlap-approach-1}). We also note that beyond the spectral threshold,  the  Bayes-optimal combination of the linear and spectral estimators has a larger overlap than either of the individual estimators (see Figure \ref{fig:all}). 
\end{remark}

Using the limiting overlap of a linear estimator in \Cref{lem:linear-overlap},
we can optimize the performance over the choice of $\cL$ (subject to \Cref{itm:assump-preproc-lin}). 
Let
\begin{align}
   \cI &\coloneqq \curbrkt{\cL\colon\bbR\to\bbR \, \text{ Lipschitz}\mbox { s.t. } \expt{G\cL(Y)}\ne0, \;\expt{\abs{G\cL(Y)}}<\infty} \label{eqn:def_I} 
\end{align}
be the set of functions $\cL$ satisfying \Cref{itm:assump-preproc-lin}. 
For $i\in\{1,2\}$ and $\delta\in(0,\infty)$, define the \emph{optimal overlaps} among linear estimators as
\begin{align}
    \mathsf{OL}_i^{\lin} &\coloneqq \sup_{\cL\in\cI} \rho_i^{\lin} . \notag
\end{align}
Furthermore, if $\cI= \emptyset$, we  set $\mathsf{OL}_1^{\lin}=\mathsf{OL}_2^{\lin}=0$. 
In words, $\mathsf{OL}_i^{\lin}$ ($i\in\{1,2\}$) is the largest overlap with the $i$-th signal that can be achieved by a linear estimator. 
Then, we have the following characterization of the optimal overlaps. 
The proof is contained in \Cref{sec:linear-estimator-pf}.


\begin{proposition}[Optimal linear estimator]
\label{lem:linear-optimal-overlap}
Consider the setting of \Cref{sec:prelim}, and let \Cref{itm:assump-signal-distr,itm:assump-alpha,itm:assump-noise-distr,itm:assump-gaussian-design,itm:assump-proportional} hold. Assume further that
\begin{align}
    \int_{\supp(Y)} \frac{\expt{Gp(y|G)}^2}{\expt{p(y|G)}} \diff y &\in (0,\infty) , \label{eqn:cond-lin-eff} 
\end{align}
where $ p(y|g) $ is the conditional law in \Cref{eq:condlaw} and the expectation is taken w.r.t.\ $G\sim\cN(0,1) $.
Then, for any $ \delta\in(0,\infty) $, writing $\alpha_1:=\alpha$ and $\alpha_2:=(1-\alpha)$, we have
\begin{align}
    \mathsf{OL}_i^{\lin} &= \paren{\frac{\alpha_1^2+ \alpha_2^2}{\alpha_i^2} + \frac{1}{\alpha_i^2\delta} \cdot \frac{1}{\int_{\supp(Y)} \frac{\expt{Gp(y|G)}^2}{\expt{p(y|G)}} \diff y}}^{-1/2}, \qquad i \in \{1,2\}. \label{eq:RHSinsp}
\end{align}
Moreover, define $ \cL^*\colon\bbR\to\bbR $ as
\begin{align}
    \cL^*(y) &= \frac{\expt{Gp(y|G)}}{\expt{p(y|G)}} . \notag 
\end{align}
Then, $ \cL^*\in\cI $ and for any $ \delta\in(0,\infty) $, both $ \mathsf{OL}_1^{\lin},\mathsf{OL}_2^{\lin} $ are simultaneously achieved by $ \cL^* $. 
\end{proposition}

\begin{remark}[When linear estimator is ineffective]
\label{rk:lin-ineff}
\Cref{eqn:cond-lin-eff} ensures that the linear estimator asymptotically achieves strictly positive overlap with the signals. In fact, 
if
\begin{align}
    \int_{\supp(Y)} \frac{\expt{Gp(y|G)}^2}{\expt{p(y|G)}} \diff y &= 0 , \notag 
\end{align}
then, from the RHS of \Cref{eq:RHSinsp}, we obtain that $\mathsf{OL}_1^{\lin}=\mathsf{OL}_2^{\lin}=0$ for any $ \delta\in(0,\infty) $.
This is the case for mixed phase retrieval, as mentioned in \Cref{sec:results-examples}. 
We note that the condition in \Cref{eqn:cond-lin-eff} also appears in the non-mixed setting (see Appendix C.1 of \cite{mondelli2021optimalcombination}). 
\end{remark}


\subsection{Spectral estimator}
\label{sec:results-spec}

The limiting value of the overlaps for the spectral estimator can  be obtained similarly to \Cref{lem:linear-overlap}.
%
\begin{corollary}[Overlaps, spectral]
\label{thm:overlap-spectral}
Consider the setting of \Cref{sec:prelim}, and let \Cref{itm:assump-signal-distr,itm:assump-alpha,itm:assump-noise-distr,itm:assump-gaussian-design,itm:assump-proportional,itm:assump-preproc-spec} hold.
Then, for $i\in\{1,2\}$, if $ \lambda^*(\delta_i) > \ol\lambda(\delta) $, we have that, almost surely,
\begin{align}
    \lim_{d\to\infty} \frac{\abs{\inprod{v_i(D)}{x_i^*}}}{\normtwo{v_i(D)}\normtwo{x_i^*}} 
            &= \rho_i^{\spec} . 
            \label{eqn:overlap-spectral-1-main} 
\end{align}
\end{corollary}

\begin{remark}[Condition for vanishing overlap]\label{rk:cnvan}
We focus here on the recovery of the first signal, and an analogous discussion is valid for the second one. As  $\lambda^*(\delta_1)$ approaches $\ol\lambda(\delta)$  from above, the RHS of \Cref{eqn:overlap-spectral-1-main} tends to $0$. 
Indeed, as $ \lambda^*(\delta_1)\searrow \ol\lambda(\delta) $, one can readily verify that $ \expt{\paren{\frac{Z}{\lambda^*(\delta_1) - Z}}^2}\nearrow \frac{1}{\delta} $ 
and consequently the numerator of $ \rho_1^{\spec} $ (cf.\ \Cref{eqn:def-rho-spec-12}) decreases to $0$. 
Furthermore, in the non-mixed setting ($\alpha=1$), the analysis of \cite{lu2020phase,mondelli-montanari-2018-fundamental} gives that, when 
$ \lambda^*(\delta) < \ol\lambda(\delta) $, the corresponding overlap vanishes. While we do not formally prove that the condition $ \lambda^*(\delta_1) > \ol\lambda(\delta) $ is \emph{necessary}  for the spectral method to have non-vanishing overlap, these two observations point strongly in that direction. A third piece of supporting evidence is provided in \Cref{rk:pseig}. 
\end{remark}

Equipped with \Cref{thm:overlap-spectral}, we can optimize both \emph{(i)} the spectral threshold, namely, the minimum value of $\delta$ needed to satisfy the condition $ \lambda^*(\delta_1) > \ol\lambda(\delta) $ which gives a strictly positive overlap, and \emph{(ii)} the limiting overlap  given by the right side of \Cref{eqn:overlap-spectral-1-main}. 
Formally, for $i\in\{1,2\}$ and $ \delta\in(0,\infty) $, let
\begin{align}
    \cH_i &\coloneqq \curbrkt{\cT\colon\bbR\to\bbR \text{ Lipschitz} \mbox{ s.t. } \begin{array}{c}
        \displaystyle \inf_{y\in\supp(Y)} \cT(y) > -\infty , \;
        \displaystyle 0 < \sup_{y\in\supp(Y)} \cT(y) < \infty , \\
        \prob{\cT(Y) = 0} < 1 , \;
        \lambda^*(\delta_i) > \ol\lambda(\delta)
    \end{array}
    } \label{eqn:def_Hi} 
\end{align}
be the set of functions $\cT$ satisfying \Cref{itm:assump-preproc-spec} such that $ \lambda^*(\delta_i)>\ol\lambda(\delta) $ holds. 
We recall that $ \delta_1 = \alpha\delta, \delta_2 = (1-\alpha)\delta $ and $\lambda^*(\cdot), \ol\lambda(\cdot)$ depend on the choice of the preprocessing function. 
Noting that $ \cH_i $ depends on $\delta$, we can define the \emph{spectral threshold} for the $i$-th signal  as  
\begin{align}
    \delta_i^{\spec} &\coloneqq \inf\curbrkt{\delta\in(0,\infty) : \cH_i \ne \emptyset} \, , \quad  i\in\{1,2\}. \notag 
\end{align}
In words, this is the smallest $\delta$ such that there exists a preprocessing function satisfying $ \lambda^*(\delta_i) > \ol\lambda(\delta) $ (and, hence, leading to non-vanishing limiting overlap). 
Furthermore, for $i\in\{1,2\}$ and $\delta>\delta_i^{\spec}$, define the \emph{optimal overlap} as 
\begin{align}
    \mathsf{OL}_i^{\spec} &\coloneqq \sup_{\cT\in\cH_i} \rho_i^{\spec}. \notag 
\end{align}
In words, for a given $\delta>\delta_i^{\spec}$, $\mathsf{OL}_i^{\spec}$ is the largest overlap with preprocessing functions that satisfy $ \lambda^*(\delta_i) > \ol\lambda(\delta) $. We note that the supremum is guaranteed to be over a nonempty set as $ \delta>\delta_i^{\spec} $. 
At this point, we can state the following result whose proof is given in \Cref{sec:spec-estimator-pf}. 

\begin{proposition}[Optimal spectral estimator]
\label{thm:opt-spec}
Consider the setting of \Cref{sec:prelim}, and let \Cref{itm:assump-signal-distr,itm:assump-alpha,itm:assump-noise-distr,itm:assump-gaussian-design,itm:assump-proportional} hold. Let $\alpha_1:=\alpha$ and $\alpha_2:=(1-\alpha)$.
Then, for $i \in  \{1,2\}$ we have
\begin{align}
    \delta_i^{\spec} &= \frac{1}{\alpha_i^2\int_{\supp(Y)} \frac{\expt{p(y|G)(G^2 - 1)}^2}{\expt{p(y|G)}} \diff y}  \, , 
    \label{eqn:spec-bound-1} 
\end{align}
and for $ \delta>\delta_i^{\spec} $,
\begin{align}
    \mathsf{OL}_i^{\spec} &= \frac{1}{\sqrt{\beta_i^*(\delta,\alpha) + \alpha_i}} \, , 
    \label{eqn:opt-overlap-1}
\end{align}
where  $ \beta_i^*(\delta,\alpha)\in(1-\alpha_i,\infty) $  are the unique solutions to the following pair of fixed point equations:
\begin{align}
    (\beta_i^*(\delta,\alpha) - (1-\alpha_i)) \int_{\supp(Y)} \frac{\expt{p(y|G)(G^2-1)}^2}{\alpha_i \expt{p(y|G)G^2} + \beta_i^*(\delta,\alpha) \expt{p(y|G)}} \diff y &= \frac{1}{\alpha_i^2 \delta} ,  \quad i \in\{1, 2\}.\label{eqn:beta1-fp} 
    %
\end{align}
Finally, for $i \in  \{1,2\}$, define $\cT_i^*\colon\bbR\to\bbR$ as 
\begin{align}
    \cT_i^*(y) &= 1 - \frac{1}{\alpha_i\cdot \frac{\expt{p(y|G)G^2}}{\expt{p(y|G)}} + (1-\alpha_i)} , \quad \text{ where } G\sim\cN(0,1).
    \label{eqn:opt-preprocessor}
\end{align}
Then, for  $ \delta>\delta_i^{\spec} $, we have: \emph{(i)} $ \cT_i^*\in\cH_i $, and \emph{(ii)} the value of $ \mathsf{OL}_i^{\spec} $ is achieved by $\cT_i^*$. 
\end{proposition}


\begin{remark}[Jointly optimal $ \cT,\cL $ for $ C_i^* $]
\hl{Finding the spectral and linear estimators that jointly maximize the overlap between the optimal combination of the two and the $i$-th signal (where $ i\in\{1,2\} $) amounts to solving the following constrained optimization problem over a pair of functions $ \cT, \cL $: }
\begin{align}
    & \sup_{(\cT, \cL) \in \cH_i \times \cI} \mathsf{OL}_i^{\mathrm{comb}} . \label{eqn:opt}
\end{align}
\hl{In the above display, $ \cH_i $ (defined in \mbox{\Cref{eqn:def_Hi}}) is the set of spectral preprocessing functions that satisfy \mbox{\Cref{itm:assump-preproc-spec}} and are effective for estimating $ x_i^* $ (i.e., $ \lambda^*(\delta_i) > \ol{\lambda}(\delta) $); 
$ \cI $ (defined in \mbox{\Cref{eqn:def_I}}) is the set of linear preprocessing functions that satisfy \mbox{\Cref{itm:assump-preproc-lin}} (and therefore are effective for estimating both signals); 
$ \mathsf{OL}_i^{\mathrm{comb}} $ (defined in \mbox{\Cref{cor:opt-combo}}) is the asymptotic overlap between the optimally combined estimator (with respect to fixed $ \cT, \cL $) and $ x_i^* $.
\mbox{\Cref{eqn:opt}} is an explicit yet challenging functional optimization problem that remains open. 
Note that in the special cases of $ \alpha =0$ or $\alpha=1$, \mbox{\Cref{eqn:opt}} reduces to an analogous optimization problem for (non-mixed) GLMs whose resolution was left open in \mbox{\cite[Section C.4]{mondelli2021optimalcombination}}.}
\end{remark}

\begin{remark}[Universal lower bounds on spectral thresholds]
\label{rk:univ-lb-spec-thr}
In  \Cref{sec:lower-bound-spec-thr}, we show that the spectral thresholds $\delta_1^{\spec}$ and $\delta_2^{\spec}$ are always at least $\delta_1^*\coloneqq  \frac{1}{2\alpha^2}$ and $\delta_2^*\coloneqq \frac{1}{2(1-\alpha)^2} $, for any conditional law $p(\cdot \mid g)$ in \Cref{eq:condlaw} (i.e., regardless of the model). 
These lower bounds coincide with the spectral thresholds for  both noiseless linear regression and noiseless phase retrieval, see \Cref{rk:lin-regr-phase-retrieval-coincide}. Thus, unlike the linear estimator, the spectral estimator (even the optimal one) does not achieve weak recovery for all $\delta>0$; it gives positive overlaps only when the aspect ratio $\delta$ exceeds a certain value. 
We highlight that the threshold associated to our proposed optimal spectral estimator is significantly lower than that corresponding to spectral estimators proposed earlier in the literature \cite{yi-2014-mixed-linear-regression,luo-2019-opt-preprocessing}, see \Cref{fig:threshold}. 
 
\end{remark}

\begin{figure}[tbp]
    \centering
    \begin{subfigure}{0.49\linewidth}
        \centering
        \includegraphics[width=\linewidth]{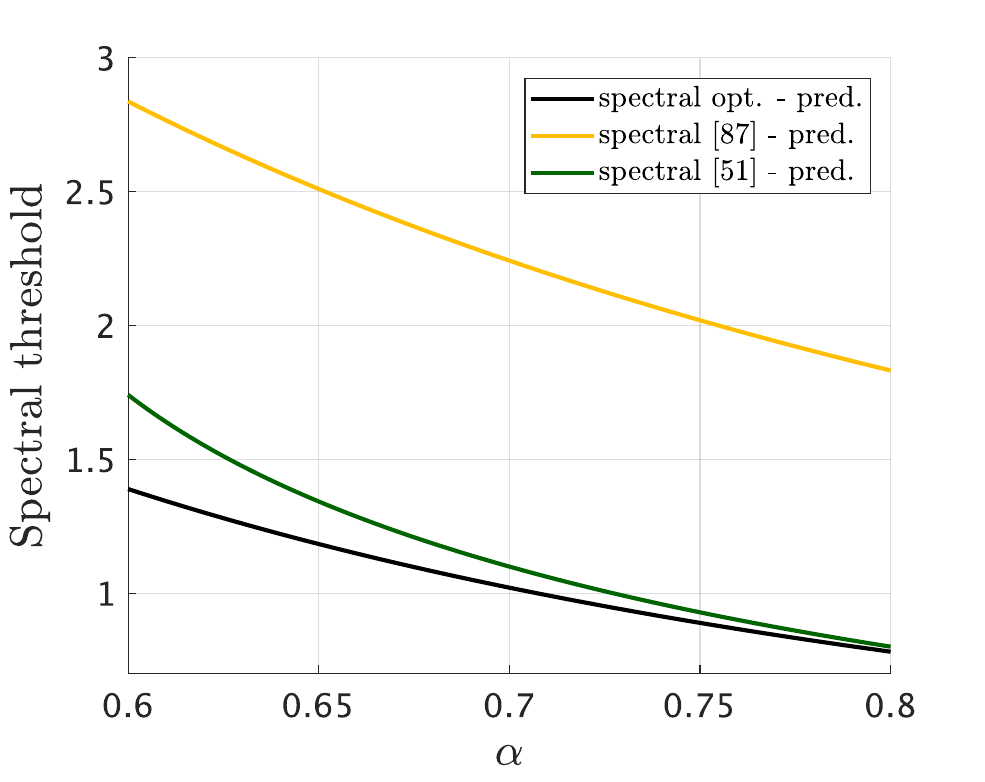}
        \caption{Recovery of $ x_1^* $}
        \label{fig:threshold1}
    \end{subfigure}
    \hfill
    \begin{subfigure}{0.49\linewidth}
        \centering
        \includegraphics[width=\linewidth]{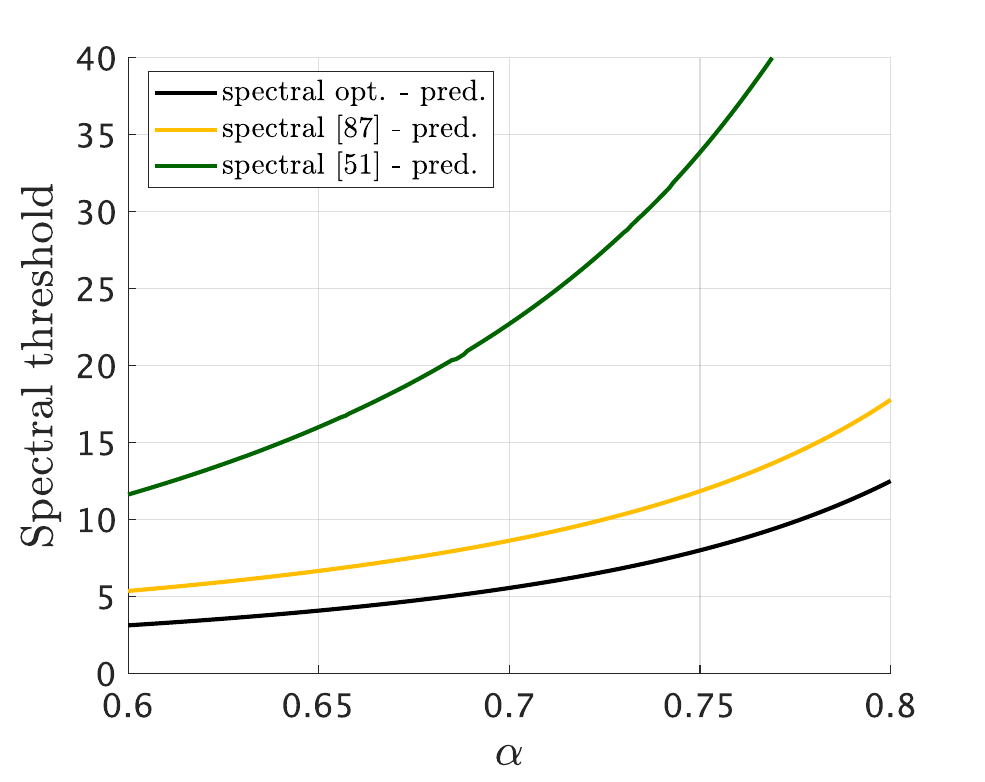}
        \caption{Recovery of $ x_2^* $}
        \label{fig:threshold2}
    \end{subfigure}
    \caption{Smallest $ \delta $ required by different spectral estimators to weakly recover signals for noiseless mixed phase retrieval. The spectral threshold is plotted as a function of a varying mixing parameter $ \alpha \in [0.6, 0.8] $. Our optimal spectral estimator always attains the lowest threshold. We note that these thresholds remain the same for noiseless mixed linear regression, due to the design of the corresponding estimators.}
    \label{fig:threshold}
\end{figure}

\begin{remark}[Overlap of spectral estimator approaches $1$]
\label{rk:overlap-approach-1}
The optimal limiting overlaps in \Cref{eqn:opt-overlap-1} approach $ 1 $ as $ \delta\to\infty $ provided
\begin{align}
    \int_{\supp(Y)} \frac{\expt{p(y|G)(G^2-1)}^2}{\alpha \expt{p(y|G)G^2} + (1-\alpha) \expt{p(y|G)}} \diff y& \, \in (0,\infty) . \label{eqn:overlap-approach-1-cond} 
\end{align}
To show this,  consider the optimal  limiting overlap between the spectral estimator and the first signal, which by \Cref{eqn:opt-overlap-1} equals $ \frac{1}{\sqrt{\beta_1^*(\delta,\alpha) + \alpha}} $. 
To show the claim, it suffices to show $ \beta_1^*(\delta,\alpha)\xrightarrow{\delta\to\infty}1-\alpha $. 
From \Cref{eqn:beta1-fp},  the fixed point equation defining $ \beta_1^*(\infty,\alpha) $ becomes 
\begin{align}
    (\beta_1^*(\infty,\alpha) - (1-\alpha)) \int_{\supp(Y)} \frac{\expt{p(y|G)(G^2-1)}^2}{\alpha \expt{p(y|G)G^2} + \beta_1^*(\infty,\alpha) \expt{p(y|G)}} \diff y &= 0 , \label{eqn:beta1-fp-delta-infty}
\end{align}
as $\delta\to\infty$. 
Since \Cref{eqn:overlap-approach-1-cond} holds, the unique solution to \Cref{eqn:beta1-fp-delta-infty} has to be $ \beta_1^*(\infty,\alpha) = 1-\alpha $. 
This proves the claim. 
We note that the condition \Cref{eqn:overlap-approach-1-cond} is satisfied by the mixed linear regression model. 
\end{remark}

\section{Numerical experiments}
\label{sec:experiments}

The experimental results in \Cref{fig:all,fig:noiseless-lin-regr,fig:lin-regr-vs-phase-retr} show that the performance of the various estimators (linear, spectral and combined)   closely match the asymptotic predictions in various settings. Furthermore, \Cref{fig:all_corr} shows that our estimators exhibit improvements over existing spectral estimators designed for non-mixed data even when the signals have mild correlation. In all plots, the signal dimension is $d=2000$, and  the vertical and horizontal axes represent the overlap and the aspect ratio $\delta$. 
The solid curves correspond to the theoretical predictions whose analytic expressions are in \Cref{sec:results-examples}. 
Discrete points (little squares, triangles, asterisks, etc.) are computed using synthetic data. Each of these points is the mean of $10$ i.i.d.\ trials together with error bars at $1$ standard deviation. Additional comments on experimental setup and results are deferred to \Cref{app:discuss-experiment}.  

\begin{figure}[tbp]
    \centering
    \includegraphics[width=0.5\linewidth]{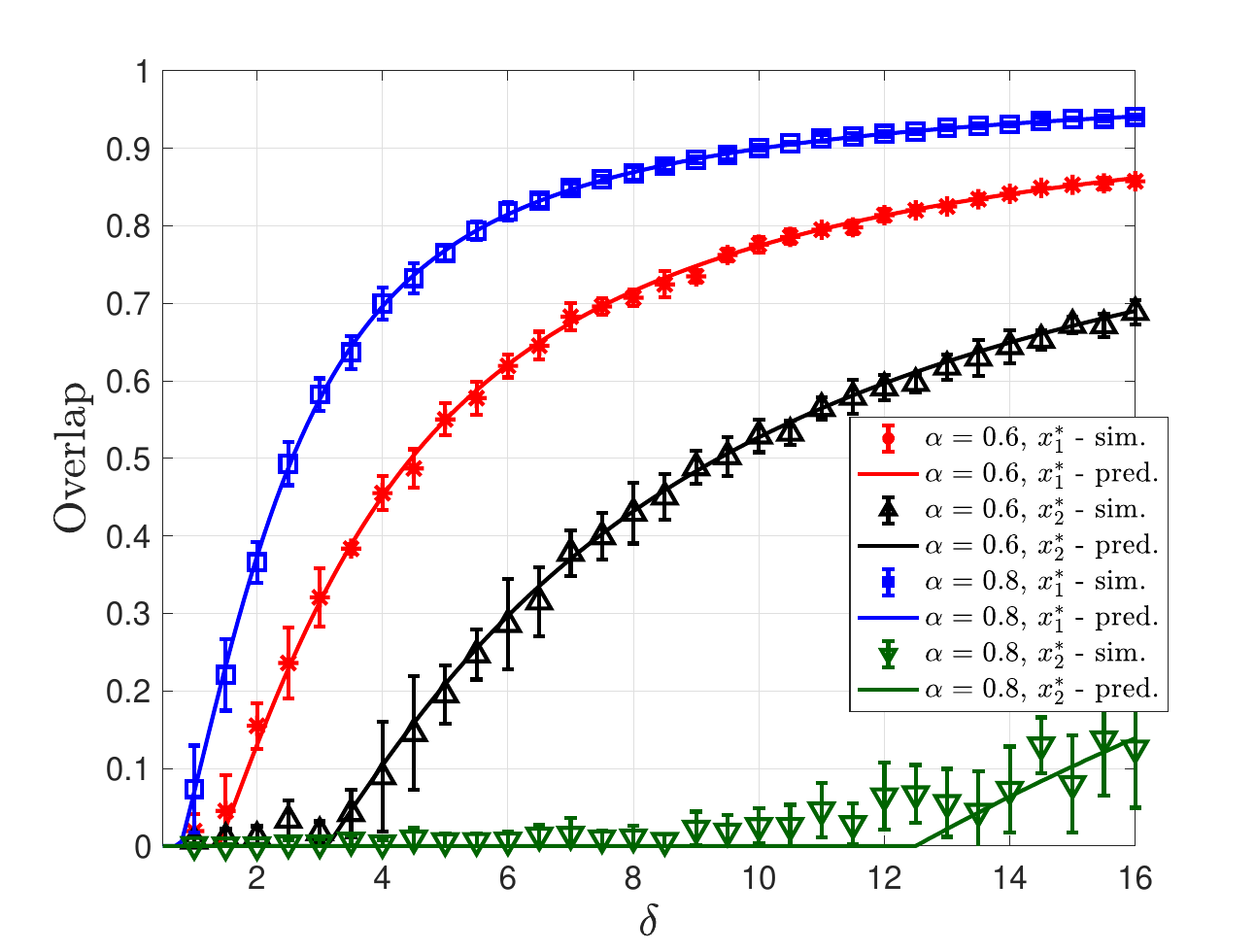}
    \caption{Spectral estimators for noiseless mixed linear regression, with mixing parameter $\alpha \in \{ 0.6, 0.8\}$. 
    Optimal spectral estimators given by \Cref{eqn:rk-spec-preproc-same} are used. 
    Overlaps with both signals $x_1^*,x_2^*$, computed from simulation (``sim.'') and prediction (``pred.''),   are plotted as a function of the aspect ratio $\delta$. 
    Same numerics apply to noiseless phase retrieval (see \Cref{rk:lin-regr-phase-retrieval-coincide}).} 
    \label{fig:noiseless-lin-regr}
\end{figure}

\begin{figure}[tbp]
    \centering
    \begin{subfigure}{0.49\linewidth}
        \centering
        \includegraphics[width=\linewidth]{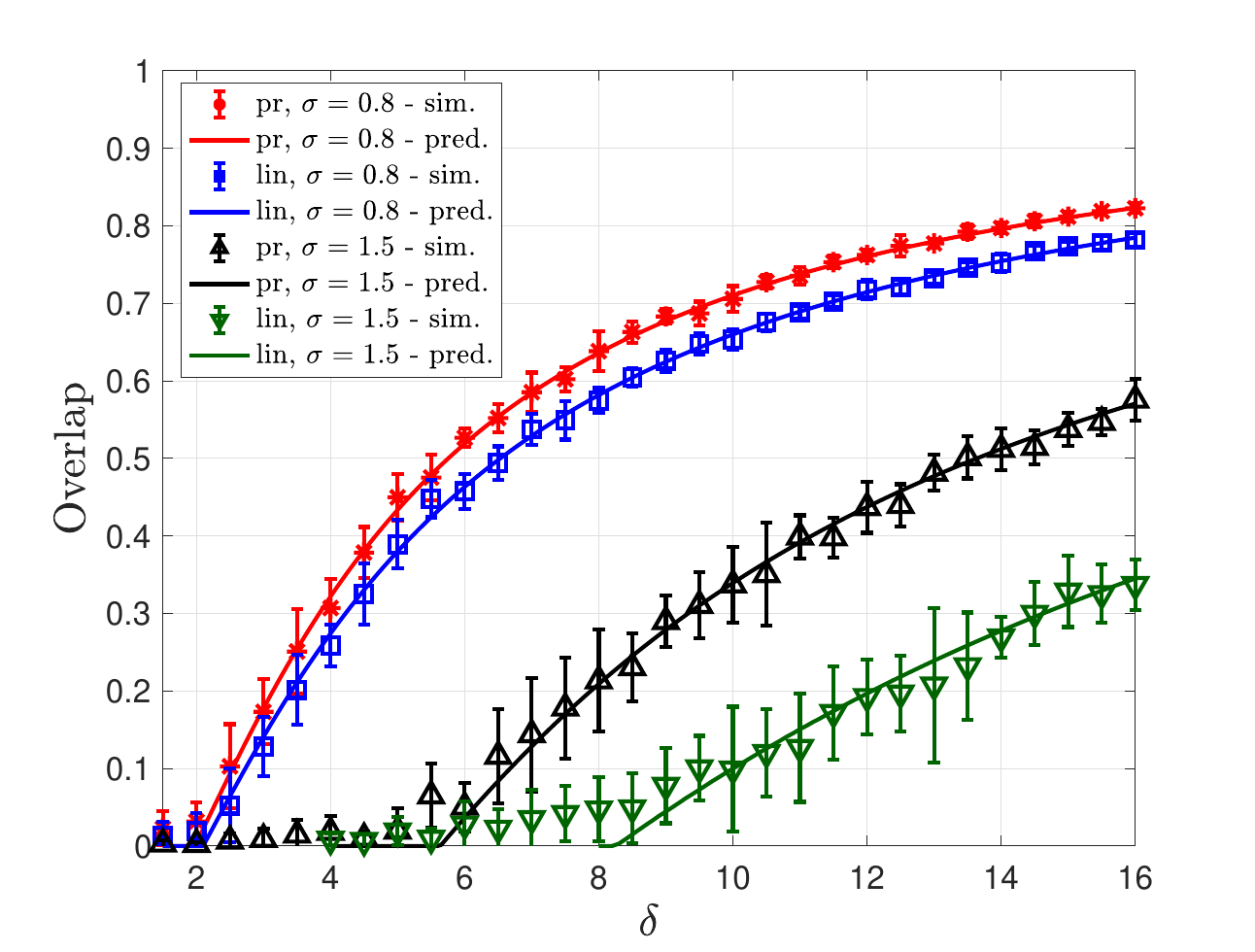}
        \caption{$ \alpha = 0.8 $ and  $ \sigma \in \{0.8,1.5\}$.}
        \label{fig:lin-regr-vs-phase-retr-diff-noise}
    \end{subfigure}
    \hfill
    \begin{subfigure}{0.49\linewidth}
        \centering
        \includegraphics[width=\linewidth]{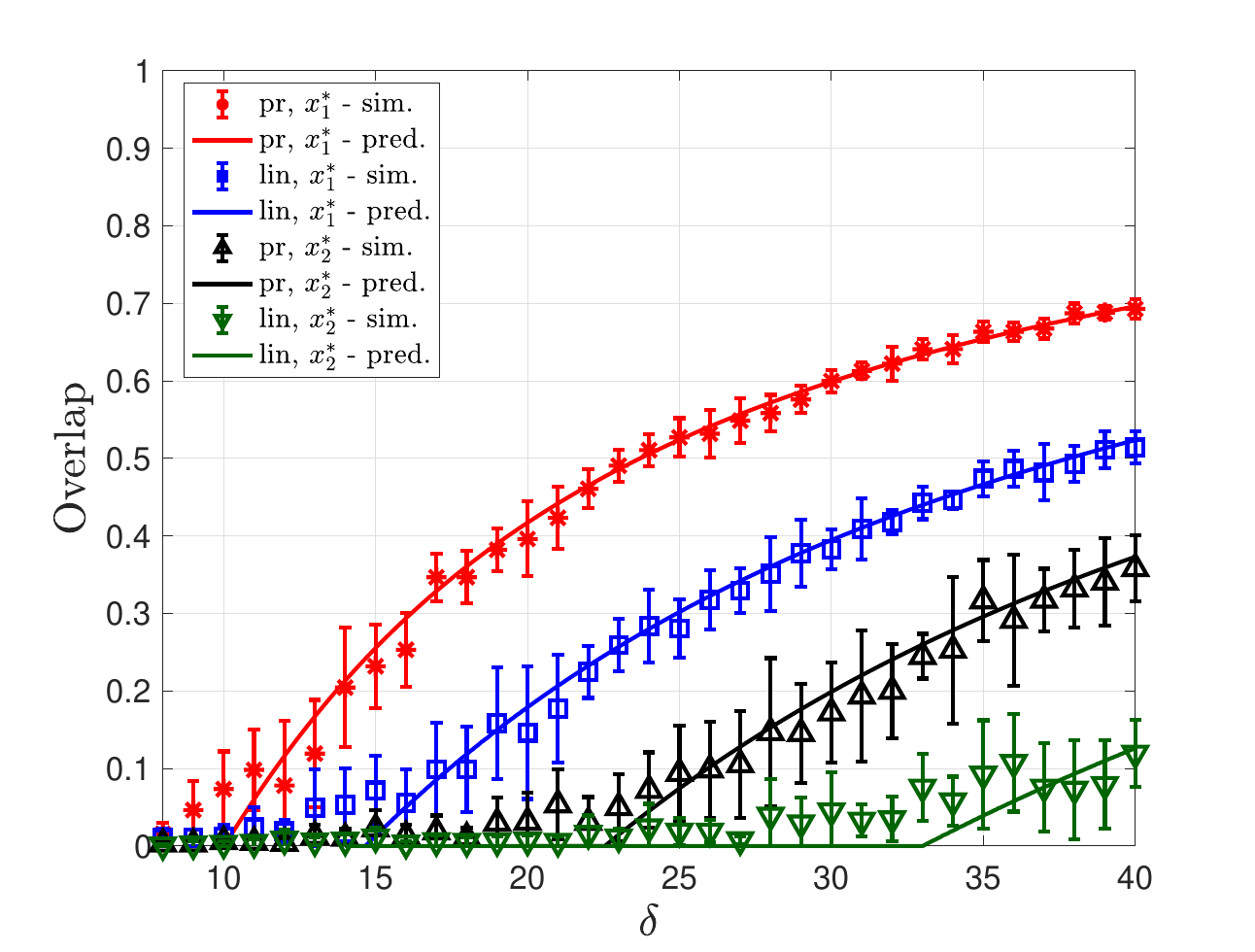}
        \caption{$ \alpha = 0.6 $ and $ \sigma = 1.5 $.}
        \label{fig:lin-regr-vs-phase-retr-both-signal}
    \end{subfigure}
    \caption{Spectral estimators for mixed linear regression  and mixed phase retrieval. 
    Optimal spectral estimators (\Cref{eqn:opt-prec-spec-lin-regr-main,eqn:opt-prec-spec-phase-retr-main}) are used.
    Overlaps with the first signal $ x_1^*$ (left plot) and with both signals $ x_1^*,x_2^* $ (right plot), computed from simulation (``sim.'') and prediction (``pred.''), are plotted as a function of the aspect ratio $\delta$.}
    \label{fig:lin-regr-vs-phase-retr}
\end{figure}

\begin{figure}[tbp]
    \centering
    \begin{subfigure}{0.49\linewidth}
        \centering
        \includegraphics[width=\linewidth]{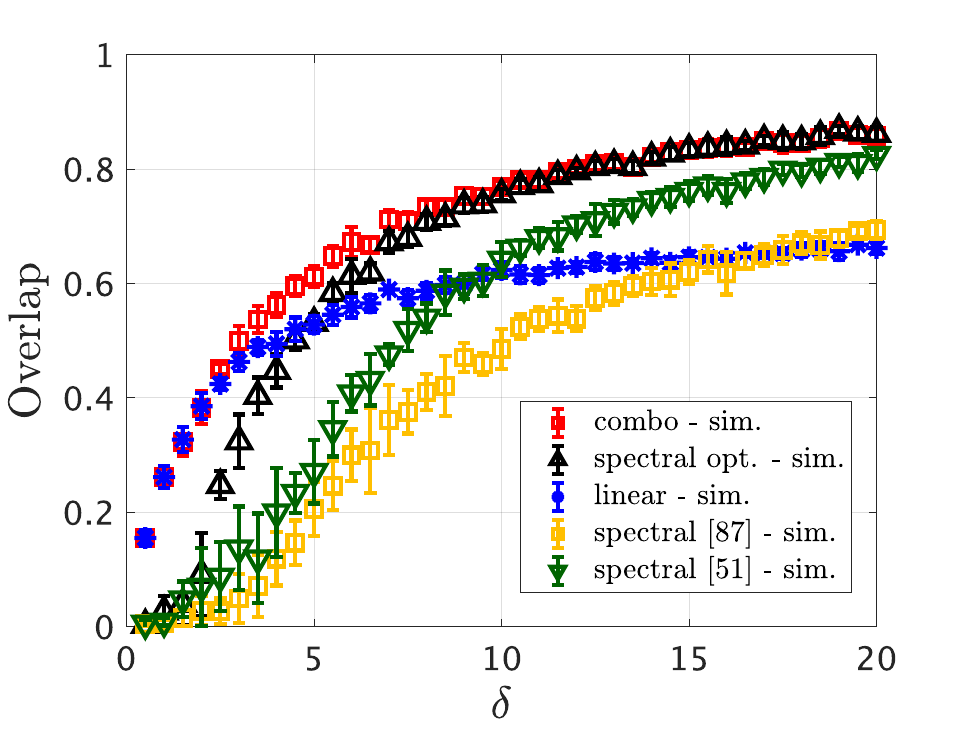}
        \caption{Recovery of $ x_1^* $}
        \label{fig:all1_corr}
    \end{subfigure}
    \hfill
    \begin{subfigure}{0.49\linewidth}
        \centering
        \includegraphics[width=\linewidth]{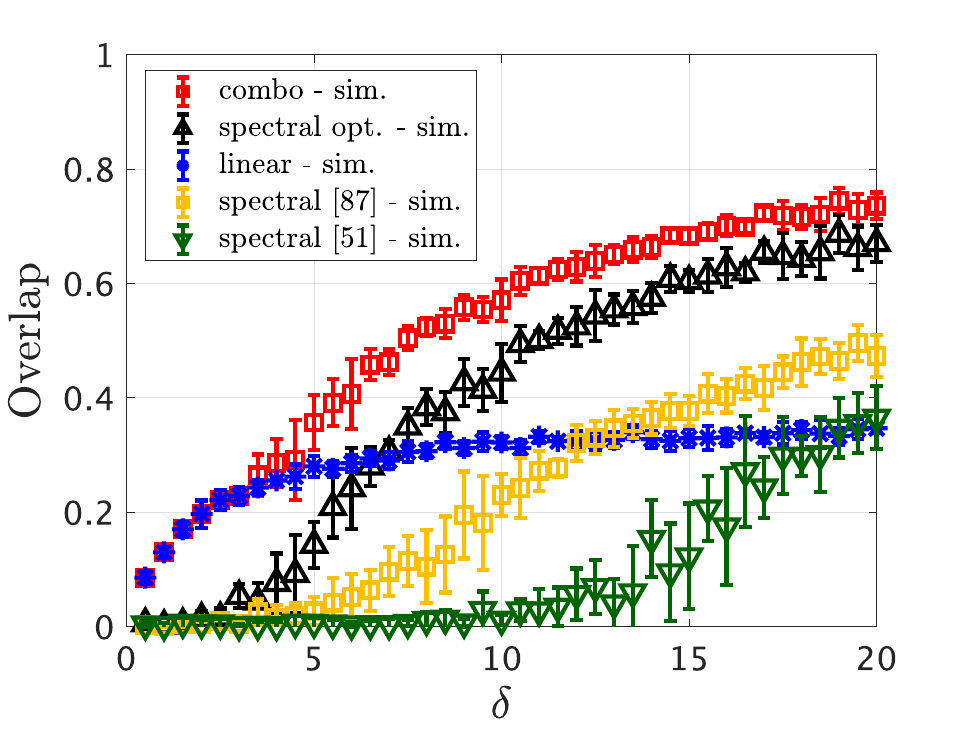}
        \caption{Recovery of $ x_2^* $}
        \label{fig:all2_corr}
    \end{subfigure}
    \caption{Performance comparison for correlated signals. The setting is noiseless mixed linear regression with mixing parameter $ \alpha = 0.6 $ and signal correlation $ \inprod{x_1^*}{x_2^*} = \rho $ where $ \rho = 0.1 $. 
    Overlaps with $ x_1^*$ (left) and $ x_2^*$ (right) are plotted as a function of the aspect ratio $ \delta $. The signal dimension is $d=2000$.
    }
    \label{fig:all_corr}
\end{figure}

\section{Proof outline}
\label{sec:sum-techniques}
The proof of 
\Cref{thm:main-thm-joint-dist} combines AMP with random matrix theory (RMT) tools. 
We now outline the high-level ideas in the analysis.

\vspace{.5em}

\paragraph{Eigenvalues via random matrix theory}
The first step is to understand the spectrum of $D$, in particular, the right edge of the bulk and the outlier(s). 
This involves the following challenges: 

\vspace{.5em}

\begin{itemize}[leftmargin=3mm]
    \item The matrix $D$ in \Cref{eqn:def-mtx-t-and-d} can be thought of as an instance of \emph{spiked matrix model}. 
    Its structure is, however, more sophisticated than the canonical ``signal plus noise'' model. 
    Indeed, the potential spikes of $D$ result from two signals through the composition of the link function $q$ and the spectral preprocessing function $ \cT $.

\vspace{.25em}
    
    \item The analysis of the limiting spectrum for non-mixed GLMs  is provided in \cite{lu2020phase,mondelli-montanari-2018-fundamental}. 
    In our mixed setting, applying the strategy of \cite{lu2020phase,mondelli-montanari-2018-fundamental} to analyze the spectrum of $D$ results in additional matrix terms which are hard to bound. 
\end{itemize}

\vspace{.5em}

The key idea is to decompose $D$ into the sum of two asymptotically free random matrices, consisting of the observations corresponding to the first and second signal. 
To be more specific, let us condition on $ \eta_1, \cdots, \eta_n $ and assume for notational convenience that $ \eta_i = 1 $ for $ 1\le i\le n_1 $ and $ \eta_i = 0 $ for $ n_1+1\le i\le n $, for some $ 0\le n_1\le n $. 
Let $ n_2 = n - n_1 $. Note that, almost surely, $ n_1/d\to\delta_1, n_2/d\to\delta_2 $. 
Now, we can write the matrices of interest in block form:
\begin{align}
A &= \begin{bmatrix}
A_1 \\ A_2
\end{bmatrix} , \quad 
T = \begin{bmatrix}
T_1 & 0_{n_1\times n_2} \\
0_{ n_2\times n_1} & T_2
\end{bmatrix} , 
\label{eqn:decomp-a-t} 
\end{align}
where $ A_1\in\bbR^{n_1\times d}, A_2\in\bbR^{ n_2\times d} $ and $ T_1\in\bbR^{n_1\times n_1}, T_2\in\bbR^{ n_2\times n_2} $. 
We also let $ \veps_1 = (\eps_1,\cdots,\eps_{n_1})$ and $\veps_2 = (\eps_{n_1+1},\cdots,\eps_n) $. 
Then,
\begin{align}
A^\top TA &= \begin{bmatrix}
A_1^\top & A_2^\top
\end{bmatrix} \begin{bmatrix}
T_1 & 0_{n_1\times n_2} \\
0_{ n_2\times n_1} & T_2
\end{bmatrix} \begin{bmatrix}
A_1 \\ A_2
\end{bmatrix}
= A_1^\top T_1 A_1 + A_2^\top T_2 A_2 . 
\label{eqn:decomp-d}  
\end{align}
Note that, for $ i\in \{1,2\} $, 
    $A_i^\top T_i A_i = A_i^\top \diag(\cT(q(A_i x_i^*, \veps_i))) A_i $. 
Since $ A_1,x_1^*,\veps_1 $ and $ A_2,x_2^*,\veps_2 $ are mutually independent, $ A_1^\top T_1 A_1 $ is independent of $ A_2^\top T_2 A_2 $. 
However, $ A_1 $ and $ T_1 $ are \emph{not} independent, neither are $ A_2 $ and $ T_2 $. 
When considered in isolation, $A_1^\top T_1 A_1$ and $A_2^\top T_2 A_2$ are obtained from a non-mixed GLM with aspect ratio discounted by $\alpha$ and $1-\alpha$, respectively. 
Thanks to \cite{lu2020phase,mondelli-montanari-2018-fundamental}, their limiting spectra are well understood.
Now, the crucial observation is that
$ A_1^\top T_1 A_1 $ and $ A_2^\top T_2 A_2 $ are \emph{asymptotically free}. 
Indeed, let $ O\sim \haar(\bbO(d)) $ be a matrix sampled uniformly from the orthogonal group $\bbO(d)$ and independent of everything else. 
Then, 
\begin{align}
A_1^\top T_1 A_1 &+ A_2^\top T_2 A_2 = 
A_1^\top \diag(\cT(q(A_1 x_1^*, \veps_1))) A_1 + A_2^\top \diag(\cT(q(A_2 x_2^*, \veps_2))) A_2 \notag \\
&\eqqlaw A_1^\top \diag(\cT(q(A_1 x_1^*, \veps_1))) A_1 + (A_2O)^\top \diag(\cT(q((A_2O)x_2^*, \veps_2))) (A_2O) \label{eqn:rotate-tmp} \\
&= A_1^\top \diag(\cT(q(A_1 x_1^*, \veps_1))) A_1 + O^\top A_2^\top \diag(\cT(q(A_2 (Ox_2^*), \veps_2))) A_2O \notag \\
&\eqqlaw A_1^\top \diag(\cT(q(A_1 x_1^*, \veps_1))) A_1 + O^\top A_2^\top \diag(\cT(q(A_2x_2^*, \veps_2))) A_2O \label{eqn:use-indep-signal-tmp} \\
&= A_1^\top T_1 A_1 + O^\top A_2^\top T_2 A_2O . \label{eqn:asymp-free}
\end{align}
\Cref{eqn:rotate-tmp} follows from the independence of $ A_1, A_2 $, and from the rotational invariance of isotropic Gaussians. 
\Cref{eqn:use-indep-signal-tmp} follows since $ O $ and $ Ox_2^* $ are independent if $ O\sim \haar(\bbO(d)) $ and $ x_2^*\sim\unif(\bbS^{d-1}) $. 
In this step, we crucially use the assumption that $ x_1^* $ and $ x_2^* $ are independent and each uniformly distributed over $ \bbS^{d-1} $. 

\label{page:free-add-conv}

The asymptotic freeness shown in \Cref{eqn:asymp-free} allows us to study the (free) sum of $ A_1^\top T_1 A_1 $ and $ A_2^\top T_2 A_2 $ using the tools developed in \cite{belinschi-outliers}.  
Indeed,  the analysis carried out in \Cref{sec:bulk,sec:outliers} implies the following characterization of the top three limiting eigenvalues of $D$ (see \Cref{thm:main-thm-eigval}): 
\begin{align}
    \lim_{d\to\infty} \lambda_1(D) &= \zeta(\lambda^*(\delta_1); \delta) , \; 
    \lim_{d\to\infty} \lambda_2(D) = \zeta(\lambda^*(\delta_2); \delta) , \; 
    \lim_{d\to\infty} \lambda_3(D) = \zeta(\ol\lambda(\delta); \delta) . \label{eqn:eigval-tmp}
\end{align}
Here, it is helpful to recall the definitions of $ \zeta(\cdot;\cdot) $ (see \Cref{eqn:def-zeta-i}), $ \lambda^*(\cdot) $ (see page \pageref{def:lambdastard}) and $ \ol{\lambda}(\cdot) $ (see \Cref{eqn:def-lam-bar}). 
Moreover, using the convexity of the function $ \zeta(\cdot; \delta) $, it can be shown that for $i\in\{1,2\}$, $ \lambda_i(D) $ is strictly larger than $ \lambda_3(D) $ in the high-dimensional limit if $ \lambda^*(\delta_i) > \ol{\lambda}(\delta) $, meaning that $\lambda_i(D)$ is detached from the bulk spectrum of $D$ and becomes an outlier eigenvalue. 
Therefore, $D$ exhibits a spectral gap between the $i$-th eigenvalue and the right edge of the bulk. 
In that case, the limiting eigenvalues admit the more explicit expressions reported in \Cref{rk:explicit-formula-eigval}. 
The existence of a spectral gap will be crucially used in proving the convergence of GAMP iterates to spectral estimators, as discussed below.

\vspace{.5em}

\paragraph{Joint distribution via GAMP} 

The convergence results in \Cref{eq:psiX1joint,eq:psiX2joint} are obtained using a generalized approximate message passing (GAMP) algorithm \cite{RanganGAMP}. 
In a mixed GLM, since the observations $(y_i)_{i \in [n]}$ are unlabeled (i.e., it is unknown to the estimator whether each $y_i$ is generated from the first or the second signal), estimating \emph{both} signals is more challenging than estimating each one from an individual non-mixed GLM.
However, the existing state evolution result for GAMP \cite{RanganGAMP}, \cite[Sec. 4]{Feng22AMPTutorial} is derived for a non-mixed model, and only keeps track of the effect of a single signal.
We generalize the GAMP state evolution result to mixed GLMs (see \Cref{prop:GAMP_SE}), so that the state evolution recursion tracks the effect of both signals. 
For convenience, 
let us work with the following rescalings:
\begin{align}
    \Abar  &\coloneqq \frac{1}{\sqrt{d}}\, A , \quad 
    \xone \coloneqq \sqrt{d}\, x_1^* , \quad 
    \xtwo \coloneqq \sqrt{d}\, x_2^* , \quad 
    \Dbar \coloneqq \Abar^\top \, T\, \Abar  = \frac{n}{d} A^\top TA . 
    \label{eqn:rescaled-tmp} 
\end{align}
Given two sequences of \emph{denoising functions} $ f_{t+1}\colon\bbR^3\to\bbR, g_t\colon\bbR^2\to\bbR $ (for each iteration $ t\ge0 $), GAMP maintains a pair of iterates $ u^t\in\bbR^n, v^{t+1}\in\bbR^d $ according to 
\begin{align}
\begin{split}
    u^t &= \frac{1}{\sqrt{\delta}} \Abar  \tv^t - \sfb_t \tu^{t-1}, \quad  \tu^t = g_{t}(u^{t};y) ,  \\
    v^{t+1} &= \frac{1}{\sqrt{\delta}} \Abar^\top \tu^t - \sfc_t  \tv^t, \quad  \tv^{t+1}=f_{t+1}(v^{t+1}; \, \xone, \xtwo) ,  
\end{split}
\label{eq:gamp-eqn-general-tmp}
\end{align}
where $ f_{t+1},g_t $ are applied component-wise, i.e., $f_{t+1}(v^{t+1}; \, \xone, \xtwo)=(f_{t+1}(v^{t+1}_1; \, \ol{x}_{1,1}^*, \ol{x}_{2,1}^*)$, $\ldots, f_{t+1}(v^{t+1}_d; \, \ol{x}_{1,d}^*, \ol{x}_{2,d}^*)), g_t(u^t; y)=(g_t(u^t_1; y_1), \ldots, g_t(u^t_n; y_n))$.
The scalars $\sfb_t, \sfc_t$ are defined as
\begin{equation}
\sfb_t =\frac{1}{n}\sum_{i=1}^d f_t'(v_i^t; \, \ol{x}_{1,i}^*, \ol{x}_{2,i}^*), \qquad
\sfc_t = \frac{1}{n}\sum_{i=1}^n g_t'(u_i^t; y_i), \notag 
\end{equation}
where $f_t'$ and $g_t'$ each denote the derivative with respect to the first argument. 
The iteration is initialized with a given $\tv^0 \in \bbR^d$ and $\tu^{-1}=0_n$.
Under the assumption that the design matrix is Gaussian (indeed $ \Abar _{i,j} \iid \cN(0,1/d) $ according to \Cref{itm:assump-gaussian-design}), the joint empirical distribution of $ u^t,v^{t+1} $ converges (as $n,d \to \infty$ with $n/d \to \delta$) to the law of a pair of jointly Gaussian random variables $ U_t,V_{t+1} $: 
\begin{equation}
    U_t \coloneqq \mu_{1,t} G_1 + \mu_{2,t} G_2 + W_{U,t} , \quad 
    V_{t+1} \coloneqq \chi_{1,t+1} X_1 + \chi_{2,t+1} X_2 +  W_{V,t+1} ,
    \notag 
\end{equation}
where $(G_1, G_2, W_{U,t}) \sim \normal(0,1) \otimes \normal(0,1) \otimes \normal(0, \sigma_{U,t}^2) $, and 
$(X_1, X_2, W_{V,t+1})  \sim \normal(0,1) \otimes \normal(0,1) \otimes \normal(0,\sigma_{V,t+1}^2)$.
The covariance structure of these jointly Gaussian random variables is described by a set of recursions called \emph{state evolution}:
\begin{gather}
    \mu_{1,t} = \frac{1}{\sqrt{\delta}} \E [ X_1 f_t(V_t; \, X_1, X_2) ], \quad
    \mu_{2,t} = \frac{1}{\sqrt{\delta}} \E[ X_2 f_t(V_t; \, X_1, X_2) ],  \nonumber \\
    \sigma_{U,t}^2 = \frac{1}{\delta}\E[ f_t(V_t; \, X_1, X_2)^2 ] - \mu_{1,t}^2 - \mu_{2,t}^2 \, , \nonumber \\
    \chi_{1,t+1} = \sqrt{\delta} \left( \E[G_1 g_t(U_t; \tY) ] - \E[ g_t'(U_t; \tY)] \mu_{1,t} \right), \nonumber \\
    \chi_{2,t+1} = \sqrt{\delta} \left( \E [ G_2 g_t(U_t; \tY) ] - \E [g_t'(U_t; \tY)] \mu_{2,t} \right), \nonumber \\
    \qquad \sigma_{V,t+1}^2 = \E[g_t(U_t; \tY)^2 ] , \nonumber 
\end{gather} 
where the random variable $\tY$ is given by 
\begin{equation}
    \tY = q(\eta G_1 + (1-\eta)G_2, \, \eps), \text{ with } (G_1, G_2, \eta, \eps) \sim \normal(0,1) \ot \normal(0,1) \ot \text{Bern}(\alpha) \ot P_{\eps}.
    \notag 
\end{equation}
The recursion is initialized as
\begin{align}
    &\mu_{1,0} =\frac{1}{\sqrt{\delta}} \lim_{d \to \infty} \frac{\langle \xone \, , \, \tv^0   \rangle}{d} , 
    \ \ 
    \mu_{2,0} =\frac{1}{\sqrt{\delta}} \lim_{d \to \infty} \frac{\langle \xtwo \, , \, \tv^0   \rangle}{d} , 
    \ \ 
    \sigma_{U,0}^2= \frac{1}{\delta} \lim_{d \to \infty} \frac{\normtwo{\tv^0}^2}{d} - \mu_{1,0}^2 - \mu_{2,0}^2 . 
    \notag 
\end{align}
The proof of convergence of the empirical distributions of $ u^t,v^{t+1} $ to the laws of $ U_t,V_{t+1} $, given in \Cref{sec:se-gamp-mixed-glm}, uses a reduction to an abstract AMP recursion with matrix-valued iterates for which a state evolution result was established in \cite{javanmard2013state,Feng22AMPTutorial}.
For details, see the formal statements in \Cref{prop:GAMP_SE} which track the joint empirical distribution of \emph{all} iterates. 



At this point, the linear estimator is readily obtained via the iterate of GAMP run for one step ($t=0$).
For $t\ge1 $, we tailor the denoisers $ (f_{t+1}, g_t)_{t\ge1} $ so that the iterates of GAMP implement a power method, which for large enough $t$, gives the first and second eigenvector of the spectral matrix $D$ (defined in \Cref{eqn:def-mtx-t-and-d}). 
Specifically, consider the GAMP iteration in \Cref{eq:gamp-eqn-general-tmp} with the initializer $ \tv^0 = 0_d $, and the following choice of denoisers:
\begin{equation}
\begin{split}
& g_0(u^0; y) = \sqrt{\delta} \cL(y),  \qquad f_1(v; \, \xone, \xtwo)= f(\xone, \xtwo), \\
& g_t(u; \, y) = \sqrt{\delta} \, u \, \cF(y), \quad f_{t+1}(v; \, \xone, \xtwo) = \frac{v}{\beta_{t+1}}, \quad t \ge 1,
\end{split}
\label{eq:ft_gt_choice-tmp}
\end{equation}
where $\cF: \reals \to \reals$ is bounded and Lipschitz,  $f: \reals^2 \to \reals$ is Lipschitz, and $\beta_{t+1} \coloneqq \sqrt{\chi_{1,t+1}^2 + \chi_{2,t+1}^2 + \sigma_{V,t+1}^2}$. 
To prove \Cref{thm:main-thm-joint-dist}, we select two pairs of functions $(f, \cF)$, in terms of the spectral preprocessing function $\cT$ (see \Cref{eq:F1x1-tmp,eq:F2x2-tmp}).
With the above choice of $ (f_{t+1})_{t\ge1}, (g_t)_{t\ge1} $, the GAMP iteration becomes
\begin{equation}
    \begin{split}
    & u^0 =0_n, \quad v^1= \Abar^\top \cL(y), \\
    %
    & u^{t} = \frac{1}{\sqrt{\delta} \, \beta_{t}} \paren{\Abar  v^{t} \, -  \, F u^{t-1}},  \quad  
    v^{t+1} =  \Abar^\top F u^t - \frac{\sqrt{\delta}}{\beta_t} \, \E[ \cF(\tY) ] \,  v^t, \qquad 
    t \ge 2,
    \end{split}
    \label{eq:newGAMP-tmp}
\end{equation}
where $F = \diag(\cF(y_1), \ldots, \cF(y_n))$.
First, note that the iterate $v^1$ coincides with the linear estimator $\xlin$ in \Cref{eqn:def-lin-estimator}.
Furthermore, we show that in the high-dimensional limit, as $t\to\infty$, the iterate $v^t$ is aligned with an eigenvector of the matrix 
\begin{equation}\label{eq:newmat}
 \Abar^\top F(\sqrt{\delta} \beta_{\infty}I_n + F)^{-1} \Abar,
\end{equation}
where $ \beta_{\infty} = \lim\limits_{t \to \infty} \beta_t $.
To justify the claim,
assume the iterates $u^t, v^{t+1}$ converge to the limits $u^\infty, v^\infty$ in the sense that 
$\lim\limits_{t \to \infty} \lim\limits_{d \to \infty} \frac{1}{d} \normtwo{u^t - u^{\infty}}^2 =0$ and  $\lim\limits_{t \to \infty} \lim\limits_{d \to \infty} \frac{1}{d} \normtwo{v^t - v^{\infty}}^2 =0$. 
Then, from \Cref{eq:newGAMP-tmp} we can derive
\begin{equation}
    v^{\infty}\left( 1 + \frac{\sqrt{\delta}}{\beta_\infty} \E[ \cF(\tY)] \right) = \Abar^{\top}F (\sqrt{\delta}\beta_{\infty}I_n + F)^{-1}\Abar  v^{\infty}.
    \label{eq:evector_eqn-tmp}
\end{equation}
Therefore, $v^{\infty}$ is an eigenvector of the matrix in \Cref{eq:newmat}, and the GAMP iteration of \Cref{eq:newGAMP-tmp} is effectively a power method.

Recall that our goal is to obtain via GAMP the two leading eigenvectors of 
$\Abar^{\top}T\Abar $.  Hence, 
we 
pick $\cF$ so that $F (\sqrt{\delta}\beta_{\infty}I_n + F)^{-1} = c\, T$, for some constant $c$. To this end, we analyze the iteration in \Cref{eq:newGAMP-tmp} with two choices for the function $\cF(y)$ and initialization $\tv^0$. 
\begin{align}
    &\text{Choice 1}: \quad \cF_1(y) \coloneqq \frac{\cT(y)}{\lambda^*(\delta_1) - \cT(y)}, \quad f(\xone, \xtwo)=\xone, 
    \label{eq:F1x1-tmp} 
    \\
    &\text{Choice 2}: \quad  \cF_2(y) \coloneqq \frac{\cT(y)}{\lambda^*(\delta_2) - \cT(y)}, \quad f(\xone, \xtwo)=\xtwo, 
    \label{eq:F2x2-tmp}
\end{align}
where 
for $i\in\{1, 2\}$, $ \lambda^*(\delta_i) $ is the unique solution of $ \zeta(\lambda;\delta_i) = \phi(\lambda) $ (see page \pageref{def:lambdastard}).
The above two choices are motivated by the characterization of the limiting eigenvalues in \Cref{eqn:eigval-tmp} and \Cref{rk:explicit-formula-eigval}. 
As outlined below, choice 1 (resp.\ choice 2) ensures that \Cref{eq:evector_eqn-tmp} becomes an eigen-equation for the first (resp.\ second) eigenvalue of $\Dbar$. 

\Cref{lem:SE_fixed_pts} shows that the state evolution parameters $ (\chi_{1,t},\chi_{2,t},\sigma_{V,t}^2) $ for choice 1 satisfy
\begin{align}
\lim_{t\to\infty} \chi_{1,t} &= \frac{\rho_1^{\spec}}{\sqrt{\delta}} , \quad 
\lim_{t\to\infty} \sigma_{V,t}^2 = \frac{1- (\rho_1^{\spec})^2}{\delta} ,  \quad \text{ and }  \ 
  \chi_{2,t} = 0 ,  \ \forall t\ge2,    \notag 
\end{align}
where the quantity $ \rho_1^{\spec} $  was defined in \Cref{eqn:def-rho-spec-12}. 
Hence,
\begin{align}
\beta_\infty &= \lim_{t\to\infty} \sqrt{\chi_{1,t}^2 + \chi_{2,t}^2 + \sigma_{V,t}^2} = \frac{1}{\sqrt{\delta}} , \notag 
\end{align}
and  \Cref{eq:evector_eqn-tmp} becomes:
\begin{equation}
    v^{\infty}\left( 1 +  \delta \E\left[ \frac{\cT(Y)}{\lambda^*(\delta_1) - \cT(Y)} \right] \right) 
    = \frac{1}{\lambda^*(\delta_1)} \Abar^{\top} T \Abar v^\infty .
    \label{eq:evector_eqn_1-tmp}
\end{equation}
With choice 1, \Cref{eq:evector_eqn_1-tmp} gives that the GAMP iterate converges to an eigenvector of $\Dbar = \Abar^\top T \Abar $ corresponding to the eigenvalue  $\lambda^*(\delta_1)\left( 1 +  \delta \E\left[ \frac{\cT(Y)}{\lambda^*(\delta_1) - \cT(Y)} \right] \right)$. 
Similarly, with choice 2, the GAMP iterate converges to an eigenvector of $\Dbar$ corresponding to the eigenvalue  $\lambda^*(\delta_2)\left( 1 +  \delta \E\left[ \frac{\cT(Y)}{\lambda^*(\delta_2) - \cT(Y)} \right] \right)$.
These 
claims match the rigorous eigenvalue characterization in \Cref{eqn:eigval-tmp} and \Cref{rk:explicit-formula-eigval}. 
At this point, note that power methods (and therefore our GAMP iterations in \Cref{eq:newGAMP-tmp}) crucially require a spectral gap to converge to the desired eigenvector. This spectral gap is guaranteed precisely by \Cref{eqn:eigval-tmp} 
provided $ \lambda^*(\delta_1) > \ol{\lambda}(\delta) $ (resp.\ $ \lambda^*(\delta_2) > \ol{\lambda}(\delta) $), which gives that $ \lambda_1(\Dbar) $ (resp.\ $ \lambda_2(\Dbar) $) is asymptotically an outlier in the spectrum of $\Dbar$. As a consequence, we can rigorously prove the convergence of the GAMP iterates under choice 1 (resp.\ choice 2) to $ v_1(\Dbar) $ (resp.\ $ v_2(\Dbar) $). 

To conclude, the iterate $ v^1 $ in the GAMP iteration in \Cref{eq:newGAMP-tmp} equals the linear estimator, and $ v^{t+1} $ asymptotically aligns with the spectral estimator. 
Since the state evolution tracks the limiting joint distribution of all iterates, the characterization 
in \Cref{eq:psiX1joint,eq:psiX2joint} follows. 
We stress that GAMP in our argument is used only as a tool for analysis and is not part of the estimators. 
The actual estimators  (spectral and linear) can be computed by a combination of the following simple operations: \emph{(i)} applying a component-wise nonlinearity, \emph{(ii)} matrix-vector/-matrix multiplication, \emph{(iii)} computation of eigenvectors.

\vspace{.5em}

\paragraph{Optimal linear and spectral estimators}
The master theorem (\Cref{thm:main-thm-joint-dist}) holds for arbitrary linear and spectral preprocessing functions $ \cL,\cT $ satisfying the stated assumptions. 
Specializing \Cref{thm:main-thm-joint-dist} to linear and spectral estimators alone and using the explicit formulas for their limiting overlaps (given in \Cref{lem:linear-overlap,thm:overlap-spectral}), we find the optimal preprocessing functions $ \cL^*,\cT_1^*,\cT_2^* $ that maximize the limiting overlaps. This is done in \Cref{lem:linear-optimal-overlap,thm:opt-spec} by casting the optimization problem as a  variational problem and solving it explicitly.





\section{Discussion}
\label{sec:discussion}
\paragraph{Universality beyond the Gaussian design matrix} A natural question is whether the predictions obtained under an i.i.d.\ Gaussian design are valid more generally. 
This topic has been investigated in the random matrix theory literature \cite{tao-vu-univ,erdos-yau-yin-univ,tulino-univ,farrell-univ,ANDERSON-univ}, and a recent line of research has focused on AMP \cite{bayati-universality,chen-lam-universality,dudeja-universality-linearized,Wang22AMPUniversality,dudeja2022universality,dudeja2022spectral-universality}. We note that none of these results is directly applicable to our setting, and the problem also remains open in the non-mixed (i.e., $\alpha=1$) setup. However, the aforementioned body of work suggests that the Gaussian predictions may hold for much more general -- even ``almost deterministic'' -- design matrices.

\vspace{.5em}

\paragraph{Mixed GLM with multiple components}
We focus on the mixed GLM with two components, but 
our approach is 
well suited to handle mixed GLMs with \emph{multiple} components. 
We now briefly sketch how to generalize our main \Cref{thm:main-thm-joint-dist}. 
The other results 
(overlaps for linear, spectral and combined estimators, and their optimization) are generalized in a similar fashion. 

Let $ x_1^*,\cdots,x_\ell^*\in\bbR^d $ be $\ell$ signal vectors, and let the observation $ y = (y_1,\cdots,y_n)\in\bbR^n $ be generated as
 $   y_i = q\paren{\inprod{a_i}{x_{\upsilon_i}^*}, \eps_i}$. 
The latent vector $\underline{\upsilon} = (\upsilon_1,\cdots,\upsilon_n) $ is a sequence of i.i.d.\ mixing random variables s.t.\ 
 $   \prob{\upsilon_i = j} = \alpha_j$ for all $i\in [n]$ and $j\in [\ell]$.
We assume that $ x_1^*,\cdots,x_\ell^*$ are i.i.d.\ and uniform on the unit sphere, and $1>\alpha_1>\alpha_2>\cdots>\alpha_\ell>0 $ (corresponding to \Cref{itm:assump-signal-distr,itm:assump-alpha}). We also impose our previous  \Cref{itm:assump-noise-distr,itm:assump-gaussian-design,itm:assump-proportional}, and assume that $\ell$ is a constant (independent of $n,d$). 
For $i\in[\ell]$, let $ \delta_i \coloneqq \alpha_i\delta $ and  
\begin{align}
    n^{\lin} &\coloneqq \paren{\paren{\sum_{k = 1}^\ell \alpha_k^2} \expt{G\cL(Y)}^2 + \frac{\expt{\cL(Y)^2}}{\delta}}^{1/2} , \notag \\
    \rho_i^{\lin} &\coloneqq \frac{\alpha_i \expt{G\cL(Y)}}{n^{\lin}} , \qquad 
    \rho_i^{\spec} \coloneqq \paren{\frac{\frac{1}{\delta} - \expt{\paren{\frac{Z}{\lambda^*(\delta_i) - Z}}^2}}{\frac{1}{\delta} + \alpha_i \expt{\paren{\frac{Z}{\lambda^*(\delta_i) - Z}}^2 (G^2 - 1)}}}^{1/2} . \notag
\end{align}
Here $Z = \cT(Y)$, $Y= q(G, \eps)$, and $G \sim \cN(0,1)$, as before. Then, under the same setting of \Cref{thm:main-thm-joint-dist} with $ \ol{x}_i^* $ and $ x_i^{\spec} $ defined similarly for $i\in [\ell]$, we have that, if $ \lambda^*(\delta_i) > \ol\lambda(\delta) $, 
\begin{align}
    \lim_{d\to\infty} \frac{1}{d} \sum_{j = 1}^d \Psi(\ol{x}_{i,j}^*, x_j^{\lin}, x_{i,j}^{\spec}) 
    &= \expt{\Psi\paren{X_i, \, \sum_{k = 1}^\ell \rho_k^{\lin} X_k + W^{\lin}, \,  \rho_i^{\spec} X_i + W_i^{\spec}}} , \notag 
\end{align}
where $ (X_1,\cdots,X_\ell) \sim \cN(0,1)^{\ot\ell} $, $ (W^{\lin}, W_i^{\spec}) $ is independent of $ (X_1,\cdots,X_\ell) $ and is jointly Gaussian with zero mean and covariance given by
\begin{align}
    \expt{(W^{\lin})^2} &= 1 - \sum_{k = 1}^\ell (\rho_k^{\lin})^2 , \; 
    \expt{(W_i^{\spec})^2} = 1 - (\rho_i^{\spec})^2 , \notag \\
    \expt{W^{\lin} W_i^{\spec}} &= \frac{\alpha_i \rho_i^{\spec}}{n^{\lin}} \expt{\frac{G\cL(Y)Z}{\lambda^*(\delta_i) - Z}} . \notag 
\end{align}

The result on the eigenvalues of the spectral matrix $D$ can be obtained by following the strategy detailed in \Cref{sec:eigval-rmt} (and sketched in \Cref{sec:sum-techniques}). Indeed, $D$ can be decomposed into the (asymptotically) free sum of the $\ell$ components associated to each of the signals, and \cite[Theorem 2.1]{belinschi-outliers} is well equipped  to characterize its top $\ell+1$ eigenvalues. 
To derive the limiting joint empirical law of the $i$-th signal and the linear and spectral estimators, we can then run a GAMP algorithm similar to \Cref{eq:ft_gt_choice-tmp} with denoisers tailored for the $i$-th signal. 
The condition $ \lambda^*(\delta_i) > \ol\lambda(\delta) $ guarantees the existence of a spectral gap between the $i$-th largest eigenvalue of $D$ and the rest of its spectrum, which in turn is leveraged to argue the convergence of GAMP to the desired eigenvector. 
This yields results analogous to \Cref{thm:main-thm-joint-dist} with $\alpha$ underlying \Cref{eq:psiX1joint} therein replaced with $\alpha_i$, for $1\le i \le \ell$.

\vspace{.5em}

\paragraph{Lower bounds for inference in mixed GLMs}

In the non-mixed setting, 
\cite{mondelli-montanari-2018-fundamental} 
 derives an information-theoretic threshold $\delta^{\infthr}$ such that, for $\delta < \delta^{\infthr}$, no estimation method gives a non-trivial estimate of the signal.\footnote{More formally, it is proved that the minimum mean squared error achieved by the Bayes-optimal estimator coincides with the error of a trivial estimator which always outputs the all-0 vector.} Furthermore, for 
  noiseless phase retrieval, 
  $\delta^{\infthr}=1/2$, which matches the threshold achieved by a spectral method. 
In the mixed setting, 
denoting by $\delta^{\infthr}_1,\delta^{\infthr}_2 $ the information-theoretic thresholds corresponding to the two signals, we have 
\begin{align}
    \delta^{\infthr}_1 \ge \frac{1}{\alpha} \delta^{\infthr}, \qquad \delta^{\infthr}_2 \ge \frac{1}{1-\alpha} \delta^{\infthr} . 
    \label{eqn:it-thr-trivial}
\end{align}
To see this, 
note that, if a genie reveals the values of the  mixing variables $(\eta_1,\cdots,\eta_n) $, 
then 
the estimation problem given mixed data with aspect ratio $\delta$ can be decoupled into two non-mixed ones with aspect ratios $ \alpha\delta$ and $(1-\alpha)\delta $. 
We also remark that adapting the second moment method of 
\cite{mondelli-montanari-2018-fundamental} to our mixed setting does not improve the bound in \Cref{eqn:it-thr-trivial} (hence, this derivation is omitted). 
Following the strategy of \cite{barbier-mmse-glm-pnas} -- which establishes the exact asymptotics of the minimum mean squared error and, thus, gives a tight bound in the non-mixed case -- requires additional ideas beyond the scope of this paper, so it is left for future research. 

As a final remark, let us contrast 
\Cref{eqn:it-thr-trivial} with the spectral bounds mentioned in \Cref{rk:univ-lb-spec-thr}, which are universal in the sense that they hold for \emph{any} mixed GLM. In particular, we note that the former scales as $(1/\alpha, 1/(1-\alpha))$, while the latter as $(1/\alpha^2, 1/(1-\alpha)^2)$, which suggests a gap between what is achievable information-theoretically and algorithmically. The possibility of a statistical-computational trade-off is also suggested by the fact that, for $\alpha=1/2$ and antipodal signals ($x_1^*=-x_2^*$), mixed linear regression reduces to phase retrieval, which is widely believed to have such a gap, see e.g.\ \cite{maillard2020phase,brennan-bresler-reducibility,celentano2020estimation,arpino2023statistical}. Closing the gap or understanding its fundamental nature remains an intriguing open question for future investigation. 



\bibliographystyle{siamplain}
\bibliography{ref}

\newpage

\appendix

\input{supplement}

\end{document}

%% file: shared.tex
\usepackage{lipsum}
\usepackage{amsfonts}
\usepackage{graphicx}
\usepackage{epstopdf}
\usepackage{algorithmic}
\ifpdf
  \DeclareGraphicsExtensions{.eps,.pdf,.png,.jpg}
\else
  \DeclareGraphicsExtensions{.eps}
\fi

\usepackage{macro}

\usepackage{enumitem}
\setlist[enumerate]{leftmargin=.5in}
\setlist[itemize]{leftmargin=.5in}

\newsiamremark{remark}{Remark}
\newsiamremark{hypothesis}{Hypothesis}
\crefname{hypothesis}{Hypothesis}{Hypotheses}
\newsiamthm{claim}{Claim}

\headers{Spectral Methods in Mixed GLMs}{Y.~Zhang, M.~Mondelli, and R.~Venkataramanan}

\title{Precise Asymptotics for Spectral Methods in Mixed Generalized Linear Models%
\thanks{Submitted to the editors \today.
\funding{Y.~Zhang and M.~Mondelli were partially supported by the 2019 Lopez-Loreta prize.}}}

\author{Yihan Zhang\thanks{School of Mathematics, University of Bristol (\email{yihan.zhang@bristol.ac.uk}).}
\and Marco Mondelli\thanks{Institute of Science and Technology Austria (\email{marco.mondelli@ist.ac.at}).}
\and Ramji Venkataramanan\thanks{Department of Engineering, University of Cambridge (\email{rv285@cam.ac.uk}).}}

\usepackage{amsopn}
\DeclareMathOperator{\diag}{diag}

\newcounter{assctr} 


%% file: supplement.tex






\paragraph{Organization of the supplementary material}
The supplementary material is organized as follows. 
Two illustrative examples of our main results in \Cref{sec:results} are given in \Cref{sec:results-examples}. 
Discussions on numerical simulations in \Cref{sec:experiments} are provided in \Cref{app:discuss-experiment}. 
The proof of the master theorem (\Cref{thm:main-thm-joint-dist}) is divided across two sections. 
\Cref{sec:eigval-rmt} contains a characterization of the top three eigenvalues of the matrix $D$ which is used in the analysis of GAMP in the following section. 
The limiting joint law of the signal,  the linear and the spectral estimators in \Cref{thm:main-thm-joint-dist} is then proved in \Cref{sec:overlap-gamp} using a GAMP algorithm and its characterization via state evolution.  The proof of the state evolution characterization is deferred to \Cref{sec:se-gamp-mixed-glm}. 
\Cref{app:bayes-opt-comb,sec:linear-estimator-pf,sec:spec-estimator-pf,sec:lower-bound-spec-thr} contain the proofs of various consequences of the master theorem. 
Several auxiliary lemmas are in \Cref{sec:aux}. 

\section{Two illustrative examples}
\label{sec:results-examples}

We specialize the results in \Cref{sec:results-spec,sec:results-lin} to two prototypical examples of mixed GLMs: 
the \emph{mixed linear regression} model where
\begin{align}
    q(g,\eps) &= g + \eps , \quad 
    \veps\sim \cN(0,\sigma^2 I_n) , \label{eqn:noisy-lin-regr-model}
\end{align}
and the \emph{mixed phase retrieval} model where
\begin{align}
    q(g,\eps) &= |g| + \eps , \quad 
    \veps \sim\cN(0,\sigma^2 I_n) . \label{eqn:noisy-phase-retrieval-model}
\end{align}
The explicit formulas for the optimal preprocessing functions, the optimal overlaps, and the thresholds (for spectral estimators) are collected in the following corollaries.  Throughout this section, for brevity we write $\alpha_1=\alpha$ and $\alpha_2=(1-\alpha)$. 
Let us first consider linear estimators. 
\begin{corollary}[Mixed linear regression, linear estimator]
\label{cor:noisy-mixed-linear-regr-lin}
Consider the mixed linear regression model  in \Cref{eqn:noisy-lin-regr-model}, 
and let \Cref{itm:assump-signal-distr,itm:assump-alpha,itm:assump-gaussian-design,itm:assump-proportional} hold. Then, the optimal preprocessing function 
$\cL^*$ defined in \Cref{lem:linear-optimal-overlap} is given by
\begin{align}
    \cL^*(y) &= \frac{y}{1+\sigma^2}. \label{eqn:opt-lin-for-lin-regr-main} 
\end{align}
Recalling that $ \xlin \coloneqq \frac{1}{n} A^\top \cL^*(y)$,  we almost surely have:
\begin{align}
    \lim_{d \to \infty} \, \frac{\inprod{\xlin}{x_i^*}}{\normtwo{\xlin}\normtwo{x_1^*}} &= \paren{\frac{\alpha_1^2 + \alpha_2^2}{\alpha_i^2} + \frac{1+\sigma^2}{\alpha_i^2\delta}}^{-1} , \qquad  i \in 
    \{1,2\}. \notag 
\end{align}
\end{corollary}

For mixed phase retrieval, one can readily 
check 
that  the overlap of  the linear estimator with each signal is always vanishing regardless of the choice of the preprocessing function. 
Next, we consider spectral estimators. 

\begin{corollary}[Mixed linear regression, spectral estimator]
\label{cor:noisy-mixed-linear-regr-spec}
Consider the mixed linear regression model 
and let \Cref{itm:assump-signal-distr,itm:assump-alpha,itm:assump-gaussian-design,itm:assump-proportional} hold. 
Then, for $i \in \{1,2 \}$, the optimal preprocessing function $\cT_i^*$ defined in \Cref{thm:opt-spec} is:
\begin{align}
    \cT_i^*(y) &= 1 - \frac{1}{\alpha_i \cdot \frac{y^2 + \sigma^2 + \sigma^4}{(1+\sigma^2)^2} + (1-\alpha_i)}.
    \label{eqn:opt-prec-spec-lin-regr-main}
\end{align}
Let $ T_i^* = \diag(\cT_i^*(y)) $ and $ D_i^* = \frac{1}{n}A^\top T_i^* A $, for $i \in \{1,2\}$. 
Denote by $ v_1(D_i^*),v_2(D_i^*) $ the eigenvectors of $D_i^*$ corresponding to the  two largest eigenvalues. Then for $ \delta > \frac{(1+\sigma^2)^2}{2\alpha_i^2} $, we almost surely have
\begin{align}
    \lim_{d\to\infty} \frac{\abs{\inprod{v_i(D_i^*)}{x_i^*}}}{\normtwo{v_i(D_i^*)}\normtwo{x_i^*}}
    &= \frac{1}{\sqrt{\beta_i^*(\delta,\alpha,\sigma) + \alpha_i}} , 
    \notag
\end{align}
where $ \beta_i^*(\delta,\alpha,\sigma) $  is the unique solution  in $ (1-\alpha_i,\infty) $ to the following fixed point equation:
\begin{multline}
    (\beta_i^*(\delta,\alpha,\sigma) - (1-\alpha_i)) \left[ -\frac{\alpha_i + \beta_i^*(\delta,\alpha,\sigma)}{\alpha_i^2} 
    + \paren{\frac{\alpha_i+\beta_i^*(\delta,\alpha,\sigma)}{\alpha_i}}^2 \right.\\
    \left. \times \sqrt{\frac{ \pi(1+\sigma^2)^2}{2\alpha_i(\sigma^2\alpha_i +(1+\sigma^2)\beta_i^*(\delta,\alpha,\sigma))}} \right. 
    \left. \exp\paren{\frac{\sigma^2\alpha_i + (1+\sigma^2)\beta_i^*(\delta,\alpha,\sigma)}{2\alpha_i}} \right. \\
    \left. \times \erfc\paren{\sqrt{\frac{\sigma^2\alpha_i + (1+\sigma^2)\beta_i^*(\delta,\alpha,\sigma)}{2\alpha_i}}} \right] = \frac{1}{\alpha_i^2\delta} . \notag 
\end{multline}
\end{corollary}

\begin{corollary}[Mixed phase retrieval, spectral]
\label{cor:noisy-mixed-phase-retrieval-spec}
Consider the mixed phase retrieval model 
in \Cref{eqn:noisy-phase-retrieval-model},  and let \Cref{itm:assump-signal-distr,itm:assump-alpha,itm:assump-gaussian-design,itm:assump-proportional} hold. Then, for $i \in \{1,2 \}$, the optimal preprocessing function $\cT_i^*$ defined in \Cref{thm:opt-spec} is:
\begin{align}
    \cT_i^*(y) &= 1 - \frac{1}{\alpha_i \Delta(y) + (1-\alpha_i)} , 
    \label{eqn:opt-prec-spec-phase-retr-main} 
\end{align}
where the auxiliary function $ \Delta\colon\bbR\to\bbR $ is defined as
\begin{align}
    \Delta(y) &\coloneqq \frac{y^2 + \sigma^2 + \sigma^4}{(1+\sigma^2)^2} + \sqrt{\frac{2}{\pi}} \cdot \frac{\sigma y \exp\paren{-\frac{y^2}{2\sigma^2(1+\sigma^2)}}}{(1+\sigma^2)^{3/2}}\sqrbrkt{1 + \erf\paren{\frac{y}{\sqrt{2\sigma^2(1+\sigma^2)}}}}^{-1} . \notag 
\end{align}
Let $ T_i^* = \diag(\cT_i^*(y))\in\bbR^{n\times n} $ and $ D_i^* = \frac{1}{n}A^\top T_i^* A\in\bbR^{d\times d} $, for $i\in\{1,2\}$. 
Denote by $ v_1(D_i^*),v_2(D_i^*) $ the eigenvectors of $D_i^*$ corresponding to the two largest eigenvalues
Let 
\begin{align}
    \delta^*_i &= \frac{1}{\alpha_i^2} \paren{ \frac{2}{(1+\sigma^2)^2} + \frac{4\sigma^5 h(\sigma^2)}{\pi^{3/2}(1+\sigma^2)^2} }^{-1} ,  \quad \text{ where } h(\sigma^2) \coloneqq \int_\bbR \frac{\exp\paren{-(2+\sigma^2)z^2} z^2}{1+\erf(z)} \diff z.
\end{align}
Define the functions $ m_0,m_1\colon\bbR\to\bbR $ and $ I\colon[1/2,1]\times(0,\infty)\to\bbR $ as
\begin{align}
    m_0(y) &\coloneqq \frac{1}{\sqrt{2\pi(1+\sigma^2)}} \exp\paren{-\frac{y^2}{2(1+\sigma^2)}} \sqrbrkt{1 + \erf\paren{\frac{y}{\sqrt{2\sigma^2(1+\sigma^2)}}}} , \notag \\
    m_1(y) &\coloneqq m_0(y) \frac{y^2 + \sigma^2 + \sigma^4}{(1+\sigma^2)^2} 
    + \frac{\sigma y}{\pi(1+\sigma^2)^2} \exp\paren{-\frac{y^2}{2\sigma^2}} , \notag \\
    I(\alpha,\beta) &\coloneqq \int_{\supp(Y)} \frac{(m_2(y) - m_0(y))^2}{\alpha m_2(y) + \beta m_0(y)} \diff y . \notag 
\end{align}
Then for $i\in\{1,2\}$ if $ \delta > \delta^*_i $, we almost surely have:
\begin{align}
    \lim_{d\to\infty} \frac{\abs{\inprod{v_i(D_i^*)}{x_i^*}}}{\normtwo{v_i(D_i^*)}\normtwo{x_i^*}}
    &= \frac{1}{\sqrt{\beta_i^*(\delta,\alpha,\sigma) + \alpha_i}} \,  , 
    \notag
\end{align}
where $ \beta_i^*(\delta,\alpha,\sigma) $ is the unique solution in $ (1-\alpha_i,\infty) $ to the fixed point equation:
\begin{align}
    (\beta_i^*(\delta,\alpha,\sigma) - (1-\alpha_i)) I(\alpha_i, \beta_i^*(\delta,\alpha,\sigma)) &= \frac{1}{\alpha_i^2\delta} . \notag 
\end{align}
\end{corollary}

\begin{remark}[Linear regression vs.\ phase retrieval]
\label{rk:lin-regr-phase-retrieval-coincide}
We note that the performance of the optimal spectral estimators given in \Cref{cor:noisy-mixed-linear-regr-spec,cor:noisy-mixed-phase-retrieval-spec} coincides for mixed \emph{noiseless} linear regression and mixed \emph{noiseless} phase retrieval. 
Specifically, for both models, when $\sigma^2=0$, the spectral thresholds and the optimal preprocessing functions are:
\begin{align}
    \delta_i^* &= \frac{1}{2\alpha_i^2} , \quad 
    \cT_i^*(y) = 1-\frac{1}{\alpha_i y^2 + (1-\alpha_i)} \, ,   
    \label{eqn:rk-spec-preproc-same}
\end{align}
and the corresponding overlaps are
$    \frac{1}{\sqrt{\beta_i^*(\delta,\alpha,0) + \alpha_i}}$,  for $i \in \{1,2\}$. Here $\beta_i^*(\delta,\alpha,0)$ is the solution to the fixed point equation in \Cref{cor:noisy-mixed-linear-regr-spec} with $\sigma=0$.
%
In fact, one can verify that even the first-order dependence of the spectral thresholds on the noise variance $\sigma$ coincides for the noisy versions of these two problems: $ \delta^*_i = \frac{1 + 2\sigma^2}{2\alpha_i^2} + \cO(\sigma^4) $ for $i \in \{1,2\}$. 
This phenomenon is because the optimal preprocessing functions $\cT_i^*(y)$  in both models depend only on $y^2$, and are therefore invariant to the signs of the observations $(y_1, \ldots, y_n)$. 
\end{remark}


\section{Discussion on numerical experiments}
\label{app:discuss-experiment}

We make a few remarks on the numerical results in \Cref{sec:experiments}. 

\begin{itemize}
    \item \Cref{fig:all} shows numerical results for the recovery of the first and second signal, respectively, from a noiseless linear regression model (i.e., the model in \Cref{eqn:noisy-lin-regr-model} with $ \sigma = 0$) with mixing parameter $\alpha = 0.6$.  We plot overlaps obtained via 
    \emph{(i)} the optimal spectral estimator in \Cref{eqn:rk-spec-preproc-same},
    \emph{(ii)} the optimal linear estimator in \Cref{eqn:opt-lin-for-lin-regr-main},
    \emph{(iii)} the Bayes-optimal linear combination of the estimators in \emph{(i)} and \emph{(ii)} (as per \Cref{cor:opt-combo}), 
    \emph{(iv)} the spectral estimator proposed in \cite{yi-2014-mixed-linear-regression} whose preprocessing function  $ \cT^{\mathrm{YCS}} $ is:
    \begin{align}
        \cT^{\mathrm{YCS}}(y) = \min\curbrkt{y^2, 10} , \label{eqn:ycs-truncated}
    \end{align}
    and \emph{(v)} the spectral estimator proposed in \cite{luo-2019-opt-preprocessing} whose preprocessing function $ \cT^{\mathrm{LAL}}$ is:
    \begin{align}
        \cT^{\mathrm{LAL}}(y) = \max\curbrkt{1 - \frac{1}{y^2}, -10} . \label{eqn:lal-truncated}
    \end{align}
    The estimators in  \Cref{eqn:ycs-truncated,eqn:lal-truncated} are truncated at $+10$ and $-10$, respectively, in order  to compute our theoretical predictions. Choosing a larger value in magnitude for the truncation does not  lead to improved empirical performance. This is because these choices are not optimal for estimation from mixed models. 
Our combined estimator  and our optimal design of the spectral method  yield substantially larger  overlaps compared to existing heuristic choices, such as those in \Cref{eqn:ycs-truncated,eqn:lal-truncated}. 

    \item In \Cref{fig:noiseless-lin-regr}, we consider the recovery of both signals for noiseless mixed linear regression (link function given by \Cref{eqn:noisy-lin-regr-model} with $\sigma=0$), using the spectral estimator with optimal preprocessing functions given by \Cref{eqn:rk-spec-preproc-same}. 
    Overlaps are plotted for two values of the mixing parameter $\alpha\in\{0.6, 0.8\}$. 
    The results for noiseless phase retrieval are identical (for both simulations with $d=2000$ and the asymptotic prediction),
    as noted in \Cref{rk:lin-regr-phase-retrieval-coincide}. 
    
    \item In \Cref{fig:lin-regr-vs-phase-retr-diff-noise}, we compare mixed linear regression  and mixed phase retrieval (\Cref{eqn:noisy-lin-regr-model,eqn:noisy-phase-retrieval-model}), under their respective optimal spectral estimators (\Cref{eqn:opt-prec-spec-lin-regr-main,eqn:opt-prec-spec-phase-retr-main}). For each model, we plot the overlap with the first signal for two different values of the noise standard deviation  $ \sigma \in \{0.8, 1.5\}$. 
    In all cases, the mixing parameter is fixed to be $\alpha = 0.8$. 
    Though for $\sigma = 0$ the curves for both models coincide, 
    the gap between phase retrieval and linear regression grows with $\sigma$, with increasingly better performance for phase retrieval. 
    
    \item In \Cref{fig:lin-regr-vs-phase-retr-both-signal}, we consider the recovery of both signals for mixed linear regression and mixed phase retrieval with mixing parameter $\alpha = 0.6$ and noise standard deviation  $\sigma = 1.5$. 
   The overlaps for linear regression are noticeably lower compared to phase retrieval, showing how the model noise makes the latter problem easier  for spectral estimation. 

   \item In \Cref{fig:all_corr}, we test our estimators against those in \Cref{eqn:ycs-truncated,eqn:lal-truncated} under the same setting of \Cref{fig:all} with a change in the prior: each signal is uniform on $ \bbS^{d-1} $ with their correlation being $ \inprod{x_1^*}{x_2^*} = \rho = 0.1 $. 
   The results show that our estimators retain their superiority. 
   Similar improvements are observed for $\rho = 0.3$, verifying the robustness of our estimators to mild signal correlation. 

\end{itemize}

\section{Eigenvalues via random matrix theory}
\label{sec:eigval-rmt}

The characterization of the limiting joint law of spectral and linear estimators in \Cref{thm:main-thm-joint-dist} is based on the analysis of a Generalized Approximate Message Passing (GAMP) algorithm. 
The proof of convergence of the GAMP iteration to the desired high-dimensional limit, whenever the conditions $ \lambda^*(\delta_1) > \ol\lambda(\delta) $ and/or $ \lambda^*(\delta_2) > \ol\lambda(\delta) $ are satisfied, crucially relies on the existence of an eigengap in  the matrix $D$ (defined in \Cref{eqn:def-mtx-t-and-d}).  In this section, we derive the limits of the top three eigenvalues of $D$. This result, stated as \Cref{thm:main-thm-eigval} below,  is then used in \Cref{sec:overlap-gamp} to prove \Cref{thm:main-thm-joint-dist}. 
\begin{theorem}[Eigenvalues]
\label{thm:main-thm-eigval}
Consider the setting of \Cref{sec:prelim} and let \Cref{itm:assump-signal-distr,itm:assump-alpha,itm:assump-noise-distr,itm:assump-gaussian-design,itm:assump-proportional,itm:assump-preproc-spec} hold. 
Then we have 
\begin{align}
    \lim_{d\to\infty} \lambda_1(D) &= \zeta(\lambda^*(\delta_1); \delta) , \quad 
    \lim_{d\to\infty} \lambda_2(D) = \zeta(\lambda^*(\delta_2); \delta) , \quad 
    \lim_{d\to\infty} \lambda_3(D) = \zeta(\ol\lambda(\delta); \delta) , \notag
\end{align}
almost surely. 
Furthermore, 
\begin{enumerate}
    \item \label{itm:mono-2} If $ \lambda^*(\delta_1) > \lambda^*(\delta_2) > \ol\lambda(\delta) $, then 
    \begin{align}
        \zeta(\lambda^*(\delta_1); \delta) > \zeta(\lambda^*(\delta_2); \delta) > \zeta(\ol\lambda(\delta); \delta) ; \notag 
    \end{align}
    
    \item \label{itm:mono-1} If $ \lambda^*(\delta_1) > \ol\lambda(\delta) \ge \lambda^*(\delta_2) $, then 
    \begin{align}
        \zeta(\lambda^*(\delta_1); \delta) > \zeta(\lambda^*(\delta_2); \delta) = \zeta(\ol\lambda(\delta); \delta) ; \notag 
    \end{align}
    
    \item \label{itm:mono-0} If $ \ol\lambda(\delta) \ge \lambda^*(\delta_1) > \lambda^*(\delta_2) $, then 
    \begin{align}
        \zeta(\lambda^*(\delta_1); \delta) = \zeta(\lambda^*(\delta_2); \delta) = \zeta(\ol\lambda(\delta); \delta) . \notag 
    \end{align}
\end{enumerate}
\end{theorem}

\begin{remark}[Phase transition for eigenvalues]
\label{rk:pseig}
\Cref{thm:main-thm-eigval} shows a \emph{phase transition} phenomenon for the top three eigenvalues of $D$: \emph{(i)} the top two eigenvalues escape from the bulk of $D$ if $ \lambda^*(\delta_1) > \lambda^*(\delta_2) > \ol\lambda(\delta) $; \emph{(ii)}  
only the largest eigenvalue escapes from the bulk if $ \lambda^*(\delta_1) > \ol\lambda(\delta) \ge \lambda^*(\delta_2) $; \emph{(iii)} 
no outlier eigenvalue exists if $ \ol\lambda(\delta)\ge\lambda^*(\delta_1) > \lambda^*(\delta_2) $. 
See \Cref{fig:psi-phi-zeta} on page \pageref{fig:psi-phi-zeta}. 
In words, the condition $\lambda^*(\delta_i)> \ol\lambda(\delta)$ is \emph{necessary and sufficient} for the $i$-th eigenvalue to escape the bulk of the spectrum. This provides an additional piece of evidence (see also \Cref{rk:cnvan}) suggesting that such condition is also necessary and sufficient for the corresponding eigenvector to have non-vanishing overlap with the signal. In fact, phase transitions in the behavior of eigenvalues typically correspond to phase transitions in the behavior of the related eigenvectors, see e.g.\ \cite{benaych2011-pca-thr-square,benaych2012-pca-thr-rect, mondelli-montanari-2018-fundamental,lu2020phase}. 
%
\end{remark}

\begin{remark}[Eigenvalues for $\alpha=1/2$]
\label{rk:alpha-half-rmt}
Similar results to \Cref{thm:main-thm-eigval} hold for $\alpha = 1/2$. 
In this case, the limits of the first and second eigenvalues of $D$ coincide and equal $ \zeta(\lambda^*(\delta/2); \delta) $. 
The limit of the right edge of the bulk of $D$ does not depend on $\alpha$ and remains the same ($\ol\lambda(\delta)$) as in \Cref{thm:main-thm-eigval}. 
Therefore, we get two cases: 
\emph{(i)} if $ \lambda^*(\delta/2) > \ol\lambda(\delta) $, the top two eigenvalues of $D$ are repeated outliers; 
otherwise, \emph{(ii)} $ \ol\lambda(\delta) \ge \lambda^*(\delta/2) $ and there is no outlier eigenvalue in the limiting spectrum of $D$. 
\end{remark}

\begin{remark}[Explicit formulas]
\label{rk:explicit-formula-eigval}
By the definition of $\zeta(\lambda;\delta)$ (cf.\ \Cref{eqn:def-zeta-i}), we can write the limits of the eigenvalues in the following more explicit form, which will be convenient in \Cref{sec:overlap-gamp}: 
\begin{align}
    \zeta(\lambda^*(\delta_1); \delta) &= 
    \begin{cases}
    \lambda^*(\delta_1)\paren{\frac{1}{\delta} + \expt{\frac{Z}{\lambda^*(\delta_1) - Z}}} , & \lambda^*(\delta_1) > \ol\lambda(\delta) \\
    \ol\lambda(\delta)\paren{\frac{1}{\delta} + \expt{\frac{Z}{\ol\lambda(\delta) - Z}}} , & \lambda^*(\delta_1) \le \ol\lambda(\delta) 
    \end{cases}, \notag \\
    \zeta(\lambda^*(\delta_2); \delta) &= 
    \begin{cases}
    \lambda^*(\delta_2)\paren{\frac{1}{\delta} + \expt{\frac{Z}{\lambda^*(\delta_2) - Z}}} , & \lambda^*(\delta_2) > \ol\lambda(\delta) \\
    \ol\lambda(\delta)\paren{\frac{1}{\delta} + \expt{\frac{Z}{\ol\lambda(\delta) - Z}}} , & \lambda^*(\delta_2) \le \ol\lambda(\delta) 
    \end{cases}, \notag \\
    \zeta(\ol\lambda(\delta); \delta) &= \ol\lambda(\delta)\paren{\frac{1}{\delta} + \expt{\frac{Z}{\ol\lambda(\delta) - Z}}} . \notag 
\end{align}
\end{remark}

\begin{proof}[Proof of \Cref{thm:main-thm-eigval}]
The proof is divided into three steps. 
Specifically, we first condition on $ \eta_1, \cdots, \eta_n $ and write $D$ as the sum of two asymptotically free spiked random matrices as on page \pageref{page:free-add-conv} of the main text. 
Then, the limit of $ \lambda_3(D) $ is determined in \Cref{lem:bulk-of-sum}. 
Finally, the limits of $ \lambda_1(D),\lambda_2(D) $ and the monotonicity properties of the limiting eigenvalues in \Cref{itm:mono-2,itm:mono-1,itm:mono-0} of the theorem are given by \Cref{lem:outlier}. 
\end{proof}

\subsection{Right edge of the bulk of $ D $}
\label{sec:bulk}

Before proceeding to the analysis, let us introduce some more notation. 
Let 
\begin{align}
D_1 = \frac{1}{n} A_1^\top T_1A_1, \quad 
D_2=\frac{1}{n}A_2^\top T_2A_2 . 
\label{eqn:d1-d2}
\end{align}
Therefore $ D=D_1+D_2 $ according to \Cref{eqn:decomp-d}. 
We first calculate the limiting value of the right edge of the bulk of the spectrum of $D$. 
\begin{lemma}
\label{lem:bulk-of-sum}
Consider the setting of \Cref{sec:prelim}. Let \Cref{itm:assump-signal-distr,itm:assump-alpha,itm:assump-noise-distr,itm:assump-gaussian-design,itm:assump-proportional,itm:assump-preproc-spec} hold. 
Denote by $\mu_D$ the empirical spectral distribution of $D$. 
Then 
\begin{align}
\lim_{d\to\infty} \sup\supp(\mu_{D}) &= \frac{1}{\delta} \cdot s^{-1}_{\mu_1\boxplus\mu_2}(-1/\ol\lambda(\delta)) , \label{eqn:lim-bulk}
\end{align}
almost surely, 
where $ \ol\lambda(\delta) $ is the solution to
\begin{align}
\expt{\paren{\frac{Z}{\ol\lambda(\delta) - Z}}^2} &= \frac{1}{\delta} \label{eqn:lam-bar-rmt} 
\end{align}
and the function $ s_{\mu_1\boxplus\mu_2}^{-1} $ is defined as
\begin{align}
s_{\mu_1\boxplus\mu_2}^{-1}(z) &= -\frac{1}{z} + \delta\, \expt{\frac{Z}{1+zZ}} . \label{eqn:s-sum-inv-def} 
\end{align}
\end{lemma}

\begin{remark}
\label{rk:stieltjes-equals-psi}
The function $ s_{\mu_1\boxplus\mu_2}^{-1} $ is the inverse \emph{Stieltjes transform} of the free additive convolution of the limiting spectral distributions $ \mu_1 $ of $\frac{n}{d}D_1$ and $ \mu_2 $ of $\frac{n}{d}D_2$. 
Furthermore,  $ s_{\mu_1\boxplus\mu_2}^{-1}(\lambda) $ 
is precisely $ \delta \psi(\lambda;\delta)$, where $\psi(\lambda;\delta) $ defined in \Cref{eqn:def-psi-fn}. 
We note also that the parameter $ \ol\lambda(\delta) $ defined in \Cref{eqn:lam-bar-rmt} is the same as that defined through \Cref{eqn:def-lam-bar}. 
(See \Cref{rk:explicit-formulas}.)
The connection shall become more transparent in the proof below. 
\end{remark}

\begin{proof}[Proof of \Cref{lem:bulk-of-sum}]
First note that the scaling factor $ \frac{1}{d} $ in \cite{mondelli-montanari-2018-fundamental} is different from our scaling $ \frac{1}{n} $ in the definition of $D$ (cf.\ \Cref{eqn:def-mtx-t-and-d}). 
We therefore consider $ \wt{D} = \frac{1}{d}A^\top TA $ for the convenience of applying Lemma 3 in \cite{mondelli-montanari-2018-fundamental}. 
All results regarding the matrix $ \wt{D} $ can be translated to $ D $ by inserting a factor $ \frac{d}{n}\to\frac{1}{\delta} $ at proper places, since $ D = \frac{d}{n}\wt{D} $. 

Let 
\begin{align}
    \wt{D}_1 = \frac{1}{d} A_1^\top T_1  A_1 , \quad 
    \wt{D}_2 = \frac{1}{d} A_2^\top T_2  A_2 . \label{eqn:def-di-tilde}
\end{align}
By \Cref{eqn:decomp-d}, $ \wt{D}=\wt{D}_1+\wt{D}_2 $. 
Let $ \mu_1 $ and $ \mu_2 $ be the limiting spectral distributions of $ \wt{D}_1 $ and $ \wt{D}_2 $, respectively, as $n_1,n_2,d\to\infty$ with $n_1/d\to\delta_1 = \alpha\delta$ and $ n_2/d\to\delta_2 = (1-\alpha)\delta $. 
As argued on page \pageref{page:free-add-conv},
$ \wt{D}_1 $ and $ \wt{D}_2 $ are asymptotically free. Hence, 
the limiting spectral distribution of $ \wt{D} $ is given by the free additive convolution of $ \mu_1,\mu_2 $,
denoted by $ \mu_1\boxplus\mu_2 $ \cite{voiculescu-free,speicher-93}. 
It remains to 
compute $ \sup\supp(\mu_1\boxplus\mu_2) $.

A careful inspection of the proof of \cite[Lemma 2]{mondelli-montanari-2018-fundamental} shows that the bulk of the spectrum of $ \wt{D}_i $, i.e., $ \lambda_2(\wt{D}_i)\ge\cdots\ge\lambda_d(\wt{D}_i) $, interlaces the spectrum of 
    $E_i \coloneqq \frac{1}{d} \wt{A}_i^\top T_i \wt{A}_i\in\bbR^{(d-1)\times(d-1)}$ 
for $i\in\{1,2\}$, respectively. 
Specifically, 
\begin{align}
    \lambda_1(E_i) \ge \lambda_2(\wt{D}_i) \ge \lambda_2(E_i) \ge \lambda_3(\wt{D}_i)
    \ge \cdots \ge \lambda_{d-2}(E_i) \ge \lambda_{d-1}(\wt{D}_i) \ge \lambda_{d-1}(E_i) \ge \lambda_d(\wt{D}_i) . \label{eqn:interlace} 
\end{align}
Here, $ T_i = \diag(\cT(q(A_i x_i^*,\veps_i))) $ (recall \Cref{eqn:decomp-a-t}) and $ \wt{A}_i\in\bbR^{n_i\times(d-1)} $ is an \emph{independent} matrix with i.i.d.\ $ \cN(0,1) $ entries. 
In particular, $ T_i $ and $ \wt{A}_i $ are independent. 
We also define, for $i\in \{1,2\}$,
    $ \wt{E}_i \coloneqq \frac{1}{d-1} \wt{A}_i^\top T_i \wt{A}_i $. 
Note that $ E_i = \frac{d-1}{d} \wt{E}_i $, $ n_1/(d-1)\to\delta_1 $ and $ n_2/(d-1)\to\delta_2 $. 

Since each $y_i$ (for $1\le i\le n$) is i.i.d., the limiting spectral distributions of $ T_1 $ and $ T_2 $ are in fact the same and both equal the law of $Z$. 
Thus, Lemma 3 in \cite{mondelli-montanari-2018-fundamental} provides us with a characterization of the limiting spectral distribution of $ \wt{E}_i $:
\begin{align}
    \mu_{\wt{E}_1} &\to \wt{\mu}_1 , \quad 
    \mu_{\wt{E}_2} \to \wt{\mu}_2 , \notag
\end{align}
weakly as $n_1,n_2,d\to\infty$ with $ n_1/(d-1)\to\delta_1,n_2/(d-1)\to\delta_2 $. 
Furthermore, the limiting spectral distributions admit the following explicit description through the inverse Stieltjes transform: 
\begin{align}
s_{\wt{\mu}_1}^{-1}(z) &= -\frac{1}{z} + \delta_1 \expt{\frac{Z}{1+zZ}} , \quad 
s_{\wt{\mu}_2}^{-1}(z) = -\frac{1}{z} + \delta_2 \expt{\frac{Z}{1+zZ}} . \label{eqn:s-inv-mui}  
\end{align}
In view of the scaling factor $ \frac{d-1}{d}\to1 $, the limiting spectral distributions of $ E_1,E_2 $ are also given by $\wt{\mu}_1,\wt{\mu}_2$, respectively. 
Recall that the bulks of the spectra of $ \wt{D}_1,\wt{D}_2 $ interlace the spectra of $E_1,E_2$, respectively (cf.\ \Cref{eqn:interlace}).
Since $ \wt{D}_1 $ and $ \wt{D}_2 $ can each have at most one outlier eigenvalue by Lemma 2 in \cite{mondelli-montanari-2018-fundamental}, the limiting spectral distributions $\mu_1,\mu_2$ of $ \wt{D}_1,\wt{D}_2 $, respectively, are the same as $ \wt{\mu}_1,\wt{\mu}_2 $ whose inverse Stieltjes transforms are shown in \Cref{eqn:s-inv-mui}. 

Let $R_\mu(z) \coloneqq s_\mu^{-1}(-z) - \frac{1}{z}$ denote the \emph{$R$-transform} \cite{Voiculescu-addition} of $\mu$. Then, a well-known fact in free probability theory is that $R_{\mu_1\boxplus\mu_2}(z) = R_{\mu_1}(z) + R_{\mu_2}(z)$. Thus,
\begin{align}
s_{\mu_1\boxplus\mu_2}^{-1}(z) 
&= s_{\mu_1}^{-1}(z) + s_{\mu_2}^{-1}(z) + \frac{1}{z} 
= s_{\wt{\mu}_1}^{-1}(z) + s_{\wt{\mu}_2}^{-1}(z) + \frac{1}{z} 
= -\frac{1}{z} + \delta \expt{\frac{Z}{1+zZ}} . \label{eqn:s-inv-sum} 
\end{align}
Given $ s_{\mu_1\boxplus\mu_2}^{-1} $, one can calculate $ \sup\supp(\mu_1\boxplus\mu_2) $ which is in turn the limiting value of $ \sup\supp(\mu_{\wt{D}}) $, where $ \mu_{\wt{D}} $ denotes the empirical spectral distribution of $ \wt{D} $. 
This can be accomplished thanks to the results in \cite[Lemma 3.1]{bai-yao-2012} (see also \cite[Sec.\ 4]{silverstein-choi-95}):
\begin{align}
\lim_{d\to\infty} \sup\supp(\mu_{\wt{D}})
&= \sup\supp(\mu_1\boxplus\mu_2) \label{eqn:convg-to-add-convl} \\
&= \min_{\lambda>\sup\supp(Z)} s_{\mu_1\boxplus\mu_2}^{-1}(-1/\lambda) \notag \\
&= \min_{\lambda>\sup\supp(Z)} \lambda + \delta\expt{\frac{Z}{1 - Z/\lambda}} . \label{eqn:min-s-sum-inv}
\end{align}
The convergence in \Cref{eqn:convg-to-add-convl} holds almost surely since 
\begin{align}
    \mu_{\wt{D}} &= \mu_{\wt{D}_1 + \wt{D}_2}
    \to \mu_1 \boxplus \mu_2 \notag 
\end{align}
weakly 
\cite{voiculescu-free,speicher-93}. 
To solve the minimization problem in \Cref{eqn:min-s-sum-inv}, we observe that the function $s_{\mu_1\boxplus\mu_2}^{-1}(-1/\lambda)$ can be written in terms of $ \psi(\lambda;\delta) $ defined in \Cref{eqn:def-psi-fn}: 
\begin{align}
    s_{\mu_1\boxplus\mu_2}^{-1}(-1/\lambda)
    &= \lambda + \delta\expt{\frac{Z}{1 - Z/\lambda}}
    = \delta \lambda \paren{\frac{1}{\delta} + \expt{\frac{Z}{\lambda - Z}}}
    = \delta \cdot \psi(\lambda;\delta) . \notag 
\end{align}
Since $ \psi(\lambda;\delta) $ is convex in the first argument (cf.\ \Cref{lem:properties-phi-psi}), so is $ s_{\mu_1\boxplus\mu_2}^{-1}(-1/\lambda) $ as a function of $\lambda$. 
As a result, the minimizer $ \ol\lambda(\delta) $ in  \Cref{eqn:min-s-sum-inv} is the critical point of $ s_{\mu_1\boxplus\mu_2}^{-1}(-1/\lambda) $. 
That is, 
\begin{align}
\left. \frac{\diff}{\diff\lambda}  s_{\mu_1\boxplus\mu_2}^{-1}(-1/\lambda)\right|_{\lambda = \ol\lambda(\delta)} &= 1 - \delta\expt{\paren{\frac{Z}{\ol\lambda(\delta) - Z}}^2} = 0 , \notag 
\end{align}
i.e., $ \ol\lambda(\delta) $ is the solution to the following equation
\begin{align}
\expt{\paren{\frac{Z}{\ol\lambda(\delta) - Z}}^2} &= \frac{1}{\delta} . \label{eqn:fp-taustar} 
\end{align}
The minimum value in \Cref{eqn:min-s-sum-inv} is therefore
\begin{align}
s_{\mu_1\boxplus\mu_2}^{-1}(-1/\ol\lambda(\delta)) &= \ol\lambda(\delta)\paren{1 + \delta\expt{\frac{Z}{\ol\lambda(\delta) - Z}}} . \label{eqn:bulk-of-free-sum-formula} 
\end{align}

At this point, we have successfully computed the limiting value of $\sup\supp(\mu_{\wt{D}})$. 
However, recall that the original matrix we are interested in is $D = \frac{d}{n} (\wt{D}_1 + \wt{D}_2)$. 
Therefore, 
\begin{align}
\lim_{d\to\infty} \sup\supp(\mu_D)
&= \lim_{d\to\infty} \frac{d}{n} \sup\supp(\mu_{\wt{D}})
= \ol\lambda(\delta)\paren{\frac{1}{\delta} + \expt{\frac{Z}{\ol\lambda(\delta) - Z}}} , \notag 
\end{align}
where $ \ol\lambda(\delta) $ satisfies \Cref{eqn:fp-taustar}. 
This concludes the proof. 
\end{proof}

\subsection{Outlier eigenvalues of $D$}
\label{sec:outliers}

Finally, we need to understand the outliers in the spectrum of $ D $. 
\begin{lemma}
\label{lem:outlier}
Consider the setting of \Cref{sec:prelim}. Let \Cref{itm:assump-signal-distr,itm:assump-alpha,itm:assump-noise-distr,itm:assump-gaussian-design,itm:assump-proportional,itm:assump-preproc-spec} hold. 
Let the function $ s_{\mu_1\boxplus\mu_2}^{-1}(-1/\lambda) $ be given by \Cref{eqn:s-sum-inv-def}. 
Then, the following statements hold. 
\begin{enumerate}
    \item \label{itm:outlier-concl-1} $ \lambda^*(\delta_1) > \lambda^*(\delta_2) $; 
    
    \item \label{itm:outlier-concl-2} If $ \lambda^*(\delta_1) > \lambda^*(\delta_2) > \ol\lambda(\delta) $, then $ s_{\mu_1\boxplus\mu_2}^{-1}(-1/\lambda^*(\delta_1)) > s_{\mu_1\boxplus\mu_2}^{-1}(-1/\lambda^*(\delta_2)) $; 
    
    \item \label{itm:outlier-concl-3} For $ i\in \{1,2\} $, if $ \lambda^*(\delta_i) > \ol\lambda(\delta) $, then $ s_{\mu_1\boxplus\mu_2}^{-1}(-1/\lambda^*(\delta_i)) > s_{\mu_1\boxplus\mu_2}^{-1}(-1/\ol\lambda(\delta)) $; 
    
    \item \label{itm:outlier-concl-4}
    We have that, almost surely,
    \begin{align}
        \lim_{d\to\infty} \lambda_1(D) &= \frac{1}{\delta} \cdot s_{\mu_1\boxplus\mu_2}^{-1}(-1/\max\{\lambda^*(\delta_1),\ol\lambda(\delta)\}) , \label{eqn:lim-lam1-d-stieltjes} \\
        \lim_{d\to\infty} \lambda_2(D) &= \frac{1}{\delta} \cdot s_{\mu_1\boxplus\mu_2}^{-1}(-1/\max\{\lambda^*(\delta_2),\ol\lambda(\delta)\}) , \label{eqn:lim-lam2-d-stieltjes} \\
        \lim_{d\to\infty} \lambda_3(D) &= \frac{1}{\delta} \cdot \sup\supp(\mu_1\boxplus\mu_2) = \frac{1}{\delta} \cdot s_{\mu_1\boxplus\mu_2}^{-1}(-1/\ol\lambda(\delta)) . \label{eqn:lim-lam3-d-stieltjes} 
    \end{align}
\end{enumerate}
\end{lemma}

\begin{remark}
\label{rk:eigval-order}
Recalling the definition $ \zeta(\lambda;\delta) = \psi(\max\{\lambda,\ol\lambda(\delta)\};\delta) $ (cf.\ \Cref{eqn:def-zeta-i}) and the relation $ \frac{1}{\delta}\cdot s_{\mu_1\boxplus\mu_2}^{-1}(-1/\lambda) = \psi(\lambda;\delta) $ (cf.\ \Cref{rk:stieltjes-equals-psi}), we can write the limiting values of $ \lambda_1(D),\lambda_2(D),\lambda_3(D) $ in \Cref{eqn:lim-lam1-d-stieltjes,eqn:lim-lam2-d-stieltjes,eqn:lim-lam3-d-stieltjes} as 
\begin{align}
    \zeta(\lambda^*(\delta_1);\delta) \ge 
    \zeta(\lambda^*(\delta_2);\delta) \ge 
    \zeta(\ol\lambda(\delta);\delta), \label{eqn:order}
\end{align}
respectively. 
To see why the above chain of inequalities holds, note that by \Cref{itm:outlier-concl-3} of \Cref{lem:outlier}, $ \zeta(\lambda^*(\delta_i); \delta)>\zeta(\ol\lambda(\delta);\delta) $ if $ \lambda^*(\delta_i)>\ol\lambda(\delta) $ and $ \zeta(\lambda^*(\delta_i); \delta)=\zeta(\ol\lambda(\delta);\delta) $ otherwise. 
So 
\begin{align}
    \zeta(\lambda^*(\delta_i); \delta)\ge\zeta(\ol\lambda(\delta);\delta) \label{eqn:order-1}
\end{align}
is always true for $i\in \{1,2\}$. 
Also, by \Cref{itm:outlier-concl-2} of \Cref{lem:outlier}, $ \zeta(\lambda^*(\delta_1);\delta)\ge\zeta(\lambda^*(\delta_2);\delta)>\zeta(\ol\lambda(\delta);\delta) $ if $ \lambda^*(\delta_1)\ge\lambda^*(\delta_2)>\ol\lambda(\delta) $. 
If $ \lambda^*(\delta_1)\ge\ol\lambda(\delta)\ge\lambda^*(\delta_2) $, $ \zeta(\lambda^*(\delta_2);\delta) = \zeta(\ol\lambda(\delta); \delta) \le \zeta(\lambda^*(\delta_1);\delta) $ by \Cref{eqn:order-1}. 
If $ \ol\lambda(\delta)\ge\lambda^*(\delta_1)\ge\lambda^*(\delta_2) $, $ \zeta(\lambda^*(\delta_1);\delta) = \zeta(\lambda^*(\delta_2);\delta) = \zeta(\ol\lambda(\delta); \delta) $. 
All cases have been exhausted in light of \Cref{itm:outlier-concl-1} of \Cref{lem:outlier}. 
In any case, 
\begin{align}
    \zeta(\lambda^*(\delta_1);\delta)\ge\zeta(\lambda^*(\delta_2);\delta) \label{eqn:order-2}
\end{align}
holds. 
\Cref{eqn:order} then follows from \Cref{eqn:order-1,eqn:order-2}. 
\end{remark}

\begin{proof}[Proof of \Cref{lem:outlier}]
The proof is divided into three parts. 
We first explicitly evaluate the theoretical predictions of the limiting values of the top three eigenvalues of $D$. 
The convergence of the outlier eigenvalues and the right edge of the bulk to the respective predictions is then formally justified in the second part. 
Finally, several properties concerning the spectral threshold and the limiting eigenvalues are proved in the third part. 

\paragraph{Limiting eigenvalues}
To understand the outlier eigenvalues of $D = D_1 + D_2$, we need to first understand the outlier eigenvalues of $D_1$ and $D_2$ individually. 
To calibrate the scaling, let us define
\begin{align}
D_1' \coloneqq \frac{1}{n_1}A_1^\top T_1 A_1 ,\quad 
D_2' \coloneqq \frac{1}{n_2}A_2^\top T_2 A_2 .\notag
\end{align}
Lemma 2 in \cite{mondelli-montanari-2018-fundamental} applies to the above matrices $D_1',D_2'$ and implies that each of $ D_1' $ and $ D_2' $ has a potential outlier eigenvalue $ \lambda_1(D_1') $ and $ \lambda_1(D_2') $, respectively. 
As $n_1,n_2,d\to\infty$ with $ n_1/d\to\delta_1 $ and $ n_2/d\to\delta_2 $, they converge almost surely to the following limiting values:
\begin{align}
\lim_{d\to\infty} \lambda_1(D_1') &= \zeta(\lambda^*(\delta_1); \delta_1) , \quad 
\lim_{d\to\infty} \lambda_1(D_2') = \zeta(\lambda^*(\delta_2); \delta_2) , \notag 
\end{align}
where $ \lambda^*(\delta_1) $ and $ \lambda^*(\delta_2) $ are the solutions to
\begin{align}
\zeta(\lambda^*(\delta_1); \delta_1) &= \phi(\lambda^*(\delta_1)) , \quad
\zeta(\lambda^*(\delta_2); \delta_2) = \phi(\lambda^*(\delta_2)) , \notag
\end{align}
respectively. 
For $ i\in\{1,2\} $, let us assume that $ \lambda_1(D_i') $ is indeed an outlier eigenvalue of $ D_i' $, that is, its limiting value $ \zeta(\lambda^*(\delta_i);\delta_i) $ lies outside the bulk of the limiting spectrum of $ D_i' $. 
According to Lemma 2 in \cite{mondelli-montanari-2018-fundamental}, this happens if and only if $ \lambda^*(\delta_i) > \ol\lambda(\delta_i) $. 
In this case, the limiting value of the outlier eigenvalue can be written more explicitly as 
\begin{align}
\zeta(\lambda^*(\delta_i);\delta_i) = \psi(\lambda^*(\delta_i);\delta_i) = \lambda^*(\delta_i) \paren{\frac{1}{\delta_i} + \expt{\frac{Z}{\lambda^*(\delta_i) - Z}}} , \label{eqn:outlier-of-summand} 
\end{align}
where $ \lambda^*(\delta_i) $ is the solution to 
\begin{align}
\expt{\frac{Z(G^2 - 1)}{\lambda^*(\delta_i) - Z}} &= \frac{1}{\delta_i} . \label{eqn:def-lam-star-explicit} 
\end{align}

Let us first translate the above result (i.e., \Cref{eqn:outlier-of-summand,eqn:def-lam-star-explicit}) regarding $ D_1',D_2' $ to $ \wt{D}_1,\wt{D}_2 $ defined in \Cref{eqn:def-di-tilde}. 
Since $ n_1/d\to\delta_1,n_2/d\to\delta_2 $ and $ \wt{D}_1 = \frac{n_1}{d}D_1',\wt{D}_2 = \frac{n_2}{d}D_2' $, we have that, almost surely,
\begin{align}
\lim_{d\to\infty} \lambda_1(\wt D_i) &= \delta_i\lambda^*(\delta_i)\paren{\frac{1}{\delta_i} + \expt{\frac{Z}{\lambda^*(\delta_i) - Z}}} 
= \lambda^*(\delta_i)\paren{1 + \delta_i\expt{\frac{Z}{\lambda^*(\delta_i) - Z}}}
\eqqcolon\theta_i , \notag 
\end{align}
where we have denoted the limiting value of $ \lambda_1(\wt D_i) $ by $ \theta_i $. 
In view of the definition of $s_{\mu_i}^{-1}$ in \Cref{eqn:s-inv-mui}, we recognize that 
\begin{align}
\theta_i = s_{\mu_i}^{-1}(-1/\lambda^*(\delta_i)) . \label{eqn:thetai-equals-si-inv}
\end{align}

Provided with the individual outlier of $\wt D_i$ (cf.\ \Cref{eqn:thetai-equals-si-inv}), 
we now invoke \cite[Theorem 2.1]{belinschi-outliers} 
to determine how an outlier of $\wt D_i$ is mapped to the spectrum of $\wt D$ by the free additive convolution. 
Specifically, 
the limiting value, denoted by $ \rho_i $, of the potential outlier of $ \wt{D} = \wt{D}_1+\wt{D}_2 $ resulting from $ \theta_i $ is given by 
\begin{align}
\rho_i \coloneqq w_i^{-1}(\theta_i) , \label{eqn:outliers-sum}
\end{align}
where $ w_1,w_2 $ are the pair of \emph{subordination functions} associated with the pair of distributions $ \mu_1,\mu_2 $. 

As the name suggests, $ w_1,w_2 $ enjoy the following subordination property (cf.\ \cite[Sec.\ 3.4.1]{belinschi-outliers}): 
\begin{align}
s_{\mu_1\boxplus\mu_2}(z) &= s_{\mu_1}(w_1(z))
= s_{\mu_2}(w_2(z)) = \frac{1}{z - (w_1(z) + w_2(z))} . \label{eqn:subord-prop} 
\end{align}
To understand the value of $ \rho_i = w_i^{-1}(\theta_i) $ (cf.\ \Cref{eqn:outliers-sum}), let us compute
\begin{align}
s_{\mu_1\boxplus\mu_2}(w_i^{-1}(\theta_i)) 
&= s_{\mu_i}(\theta_i) 
= -1/\lambda^*(\delta_i). \label{eqn:implies-wi-inv}
\end{align}
The first equality is by the subordination property (\Cref{eqn:subord-prop}) and the second one by the observation in \Cref{eqn:thetai-equals-si-inv}. 
\Cref{eqn:implies-wi-inv} then gives
\begin{align}
\rho_i &= w_i^{-1}(\theta_i) = s_{\mu_1\boxplus\mu_2}^{-1}(-1/\lambda^*(\delta_i)) . \label{eqn:rhoi} 
\end{align}

To translate the result in \Cref{eqn:rhoi} regarding $ \wt{D} $ to $ D $ in \Cref{eqn:d1-d2}, we simply note that $ D = \frac{d}{n}\wt{D} $ and $ d/n\to1/\delta $. 
Therefore, the limiting eigenvalue of $ D $ resulting from the outlier eigenvalue of $D_i$ is given by
\begin{align}
\frac{1}{\delta} \cdot \rho_i = \frac{1}{\delta} \cdot s_{\mu_1\boxplus\mu_2}^{-1}(-1/\lambda^*(\delta_i)) , \label{eqn:outlier-d-ref} 
\end{align}
almost surely.

\paragraph{Convergence of eigenvalues}
We then formally justify that the right edge of the bulk and the outlier eigenvalues of $D$ indeed converge to the theoretical predictions in \Cref{eqn:lim-bulk,eqn:outlier-d-ref}, respectively, as $d\to\infty$, therefore confirming the validity of the latter formulas. 
Let $ \cK_0 \coloneqq\supp(\mu_1\boxplus\mu_2) $. 
For $i\in\{1,2\}$, let $ \cK_i $ be the singleton set $ \{\rho_i\} $ if $ \theta_i\notin\supp(\mu_i) $ and $ \emptyset $ otherwise. 
Let $ \cK\coloneqq\cK_0\cup\cK_1\cup\cK_2 $. 
Then the first statement of \cite[Theorem 2.1]{belinschi-outliers} 
guarantees that for any $\eps>0$, 
\begin{align}
    \prob{\exists d_0,\,\forall d>d_0,\,\{\lambda_i(\wt D)\}_{i=1}^d\subset \cK_\eps} &= 1 , \label{eqn:use-belinschi-1}
\end{align}
where $ \cK_\eps $ denotes the \emph{$\eps$-enlargement} of $\cK$, i.e., 
\begin{align}
    \cK_\eps \coloneqq \curbrkt{\rho\in\bbR : \inf_{\rho'\in\cK} |\rho - \rho'| \le \eps} . \notag 
\end{align}
In words, \Cref{eqn:use-belinschi-1} says that almost surely for every sufficiently large dimension $d$, the spectrum of $\wt D$ is contained in an arbitrarily small neighbourhood of $\cK$. 
Furthermore, suppose $ \rho\in\cK_1\cup\cK_2 $ and $ \rho\notin\cK_0 $, that is, $\rho$ is an outlier in the limiting spectrum of $\wt D$.
Assume also that $ \eps>0 $ is sufficiently small so that $ (\rho - 2\eps,\rho+2\eps) \cap \cK = \{\rho\} $. 
Then
\begin{align}
    \prob{\exists d_0,\,\forall d>d_0,\,\abs{\{\lambda_i(\wt D)\}_{i=1}^d \cap (\rho-2\eps,\rho+2\eps)} = \indicator{w_1(\rho) = \theta_1} + \indicator{w_2(\rho) = \theta_2}} &= 1 . 
    \label{eqn:use-belinschi-2}
\end{align}
In words, \Cref{eqn:use-belinschi-2} says that almost surely for every sufficiently large dimension $d$, the outlier $ \theta_1 $ (resp.\ $\theta_2$) in the limiting spectrum of $\wt D_1$ (resp.\ $\wt D_2$) is mapped to $ w_1^{-1}(\theta_1) $ (resp.\ $w_2^{-1}(\theta_2)$) in the limiting spectrum of $\wt D$. 
Since $ D,D_1,D_2 $ and $ \wt D,\wt D_1,\wt D_2 $ only differ by a $\delta$ factor, similar statements hold true for $ D,D_1,D_2 $ as well. 

Combining \Cref{eqn:outlier-d-ref,eqn:use-belinschi-1,eqn:use-belinschi-2} yields \Cref{eqn:lim-lam1-d-stieltjes,eqn:lim-lam2-d-stieltjes,eqn:lim-lam3-d-stieltjes} in \Cref{itm:outlier-concl-4} of \Cref{lem:outlier}. 

\paragraph{Properties of spectral threshold and limiting eigenvalues}
We identify under what condition $ \rho_i = w_i^{-1}(\theta_i) $ is an outlier in the limiting spectrum of $\wt D$. 
For this to be the case, $\theta_i$ is necessarily an outlier in the limiting spectrum of $\wt D_i$, which is assumed in the preceding derivations. 
As Lemma 2 in \cite{mondelli-montanari-2018-fundamental} guaranteed, a sufficient and necessary condition for this event is $ \lambda^*(\delta_i) > \ol\lambda(\delta_i) $. 
Under the free additive convolution, the outlier $\theta_i$ of $\wt D_i$ is then mapped to $w_i^{-1}(\theta_i)$. 
Let us compare $w_i^{-1}(\theta_i)$ with $ \sup\supp(\mu_1\boxplus\mu_2) $, i.e., the right edge of the bulk of the limiting spectral distribution of $ \wt D = \wt D_1 + \wt D_2 $. 
The former quantity equals $ s_{\mu_1\boxplus\mu_2}^{-1}(-1/\lambda^*(\delta_i)) $ (as derived in \Cref{eqn:rhoi}) and the latter one equals $ s_{\mu_1\boxplus\mu_2}^{-1}(-1/\ol\lambda(\delta)) $ (see \Cref{eqn:bulk-of-free-sum-formula} in the proof of \Cref{lem:bulk-of-sum}). 
Recall the following two facts: 
\begin{enumerate}
    \item $ s_{\mu_1\boxplus\mu_2}^{-1}(-1/\lambda) = \delta \cdot \psi(\lambda;\delta) $ (as observed in \Cref{rk:stieltjes-equals-psi});
    \item $ \psi(\lambda;\delta) $ is convex in $\lambda$ and increasing for $ \lambda\in[\ol\lambda(\delta),\infty) $ (proved in \Cref{lem:properties-phi-psi}). 
\end{enumerate}
We therefore conclude that $s_{\mu_1\boxplus\mu_2}^{-1}(-1/\lambda^*(\delta_i)) > s_{\mu_1\boxplus\mu_2}^{-1}(-1/\ol\lambda(\delta))$ if $ \lambda^*(\delta_i) > \ol\lambda(\delta) $. 
This establishes \Cref{itm:outlier-concl-3} of \Cref{lem:outlier}. 
This condition is more stringent than the previous one $ \lambda^*(\delta_i) > \ol\lambda(\delta_i) $. 
This can be seen by inspecting the definitions (see, e.g., \Cref{eqn:def-lam-bar-explicit} in \Cref{rk:explicit-formulas}) of $\ol\lambda(\delta_i)$ and $ \ol\lambda(\delta) $:
\begin{align}
    \expt{\paren{\frac{Z}{\ol\lambda(\delta) - Z}}^2} &= \frac{1}{\delta} , \quad 
    \expt{\paren{\frac{Z}{\ol\lambda(\delta_i) - Z}}^2} = \frac{1}{\delta_i} , \label{eqn:lambda-bar-order}
\end{align}
respectively, and realizing that $ \ol\lambda(\delta) > \ol\lambda(\delta_i) $ since $ \delta > \delta_i $. 

We pause and make the following remark regarding the effect of the free additive convolution on the outliers in the spectra of the addends. 
Comparing \Cref{eqn:rhoi} with the limiting value of the right edge of the bulk (cf.\ \Cref{eqn:bulk-of-free-sum-formula}), we note the following: $ \lambda_1(\wt{D}_i) $ being an outlier eigenvalue of $ \wt{D}_i $ does \emph{not} imply that its image $\rho_i$ under the free additive convolution is also an outlier eigenvalue of $ \wt{D} = \wt{D}_1+\wt{D}_2 $. 
In fact, it can be buried strictly inside the bulk, which happens if $ \ol\lambda(\delta_i) < \lambda^*(\delta_i) < \ol\lambda(\delta) $. 

We then show $ \lambda^*(\delta_1) > \lambda^*(\delta_2) $ in \Cref{itm:outlier-concl-1} of \Cref{lem:outlier}. 
Recall that $ \lambda^*(\delta_1) $ and $ \lambda^*(\delta_2) $ are the unique solutions to $ \zeta(\lambda^*(\delta_1); \delta_1) = \phi(\lambda^*(\delta_1)) $ and $ \zeta(\lambda^*(\delta_2); \delta_2) = \phi(\lambda^*(\delta_2)) $, respectively. 
Since $ \zeta(\cdot;\delta_1),\zeta(\cdot;\delta_2) $ are non-decreasing and $ \phi(\cdot) $ is strictly decreasing, it suffices to show 
\begin{align}
   \zeta(\lambda;\delta_1) < \zeta(\lambda;\delta_2) \label{eqn:zeta-toshow}
\end{align}
for any $ \lambda>\sup\supp(Z) $. 
We do so in four steps. 
(The following arguments are best understood with \Cref{fig:psi-phi-zeta} in mind.)
\begin{enumerate}
    \item \label{itm:step1} First we claim that $ \ol\lambda(\delta_1) > \ol\lambda(\delta_2) $. 
    This follows from a similar observation as in \Cref{eqn:lambda-bar-order} and the assumption $ \alpha>1/2 $ (cf.\ \Cref{itm:assump-alpha}) which implies $ \delta_1>\delta_2 $. 
    
    \item \label{itm:step2} Second we claim that $ \psi(\ol\lambda(\delta_1); \delta_1) < \psi(\ol\lambda(\delta_2); \delta_2) $. 
    Indeed,
    \begin{align}
        \psi(\ol\lambda(\delta_1); \delta_1) &< \psi(\ol\lambda(\delta_2); \delta_1)
        < \psi(\ol\lambda(\delta_2); \delta_2) . \notag 
    \end{align}
    The first inequality follows since $ \psi(\lambda;\delta_1) $ is strictly decreasing for $ \lambda\le\ol\lambda(\delta_1) $ (see \Cref{itm:property-2} of \Cref{lem:properties-phi-psi}) and $\ol\lambda(\delta_1) > \ol\lambda(\delta_2)$ as shown in \Cref{itm:step1} above. 
    The second inequality follows since 
    \begin{align}
        \psi(\cdot;\delta_1)<\psi(\cdot;\delta_2) \label{eqn:psi-order}
    \end{align}
    for any $ \lambda>\sup\supp(Z) $ (see the definition of $\psi$ in \Cref{eqn:def-psi-fn} and also \Cref{itm:property-3} of \Cref{lem:properties-phi-psi}). 
    Note that in this step we use $ \sup\supp(Z)>0 $ in \Cref{itm:assump-preproc-spec}. 
    This shows that \Cref{eqn:zeta-toshow} holds for any $ \lambda\le\ol\lambda(\delta_2) $. 
    
    \item \label{itm:step3} We then claim that \Cref{eqn:zeta-toshow} holds for any $ \lambda\ge\ol\lambda(\delta_1) $. 
    This is because, in this regime, we have
    \begin{align}
        \zeta(\lambda; \delta_1) = \psi(\lambda; \delta_1)
        < \psi(\lambda; \delta_2)
        = \zeta(\lambda; \delta_2) \notag 
    \end{align}
    using the definition of $\zeta(\cdot;\delta_i)$ (cf.\ \Cref{eqn:def-zeta-i}) and \Cref{eqn:psi-order}. 
    
    \item \label{itm:step4} Finally, it remains to verify that \Cref{eqn:zeta-toshow} holds for $ \ol\lambda(\delta_2)\le\lambda\le\ol\lambda(\delta_1) $. 
    Indeed, we have
    \begin{align}
    \zeta(\lambda; \delta_1) &= \psi(\ol\lambda(\delta_1); \delta_1)
    <\psi(\ol\lambda(\delta_2); \delta_2)
    <\psi(\lambda; \delta_2) . \notag 
    \end{align}
    The equality is by definition of $\zeta(\cdot;\delta_1)$. 
    The first inequality is by \Cref{itm:step2} above. 
    The second inequality follows since $\psi(\cdot;\delta_2)$ is strictly increasing for $ \lambda\ge\ol\lambda(\delta_2) $ (see \Cref{itm:property-2} of \Cref{lem:properties-phi-psi}). 
\end{enumerate}
Combining \Cref{itm:step1,itm:step2,itm:step3,itm:step4} above then proves \Cref{eqn:zeta-toshow} which implies \Cref{itm:outlier-concl-1} of \Cref{lem:outlier}. 

Since $\ol\lambda(\delta)$ is the (unique) critical point of $s_{\mu_1\boxplus\mu_2}^{-1}(-1/\lambda)$ which is increasing for $ \lambda\ge\ol\lambda(\delta) $, \Cref{itm:outlier-concl-2} of \Cref{lem:outlier} then follows. This concludes the argument.
\end{proof}

\section{Joint distribution via Approximate Message Passing}
\label{sec:overlap-gamp}

The limiting joint distribution in \Cref{thm:main-thm-joint-dist} is obtained via a \emph{generalized approximate message passing} (GAMP) algorithm whose iterates converge to the top two  eigenvectors of $D = A^\top T A $.
Within this section, we adopt the following rescaling for the convenience of applying the GAMP machinery:
\begin{align}
    \Abar  &\coloneqq \frac{1}{\sqrt{d}}\, A , \quad 
    \xone \coloneqq \sqrt{d}\, x_1^* , \quad 
    \xtwo \coloneqq \sqrt{d}\, x_2^* , \quad 
    \Dbar \coloneqq \Abar^\top \, T\, \Abar  = \frac{n}{d} A^\top TA . \label{eqn:rescaled} 
\end{align}
Due to \Cref{itm:assump-signal-distr,itm:assump-gaussian-design}, we have $ \Abar _{i,j} \iid \cN(0,1/d) $ and $\xone,\xtwo \iid \unif(\sqrt{d}\,\bbS^{d-1})$. 
Let $ \ol{a}_i^\top\in\bbR^d $ denote the $i$-th row of $\Abar$. 
Then, we have
\begin{align}
    y_i &= q\paren{\inprod{a_i}{\eta_i x_1^* + (1-\eta_i) x_2^*}, \eps_i}
    = q\paren{\inprod{\ol{a}_i}{\eta_i \xone + (1-\eta_i) \xtwo}, \eps_i} . \notag 
\end{align}
Therefore, $y\in\bbR^n$ and related quantities such as $T\in\bbR^{n\times n}$ (defined in \Cref{eqn:def-mtx-t-and-d}) do not have to be rescaled. 
The overlaps are invariant under rescaling of $D$. 
Furthermore, since $n/d\to\delta$, the limiting eigenvalues of $\Dbar$ are equal to those of $D$ multiplied by $\delta$ in view of \Cref{eqn:rescaled}. 



We first extend the GAMP algorithm for the non-mixed GLM \cite{RanganGAMP} and its associated state evolution analysis to the mixed GLM model.  The GAMP algorithm is defined in terms of a sequence of Lipschitz functions $g_t:\mathbb{R}^2 \to \mathbb R$ and  $f_{t+1}:\mathbb{R}^3 \to \mathbb R$, for $t \geq 0$.  For $t\ge 0$, the algorithm iteratively computes $u^t, \tu^t \in \bbR^n$ and $v^{t+1}, \tv^{t+1} \in \bbR^d$ as follows:
\begin{align}
\begin{split}
    u^t &= \frac{1}{\sqrt{\delta}} \Abar  \tv^t - \sfb_t \tu^{t-1}, \quad  \tu^t = g_{t}(u^{t};y) ,  \\
    v^{t+1} &= \frac{1}{\sqrt{\delta}} \Abar^\top \tu^t - \sfc_t  \tv^t, \quad  \tv^{t+1}=f_{t+1}(v^{t+1}; \, \xone, \xtwo). 
\end{split}
\label{eq:gamp-eqn-general}
\end{align}
The iteration is initialized with a given $\tv^0 \in \bbR^d$ and $\tu^{-1}=0_n$. The functions $f_t$ and $g_t$ are  applied component-wise, i.e., $f_t(v^t; \, \xone, \xtwo)=(f_t(v^t_1; \, \ol{x}_{1,1}^*, \ol{x}_{2,1}^*)$, $\ldots, f_t(v^t_d; \, \ol{x}_{1,d}^*, \ol{x}_{2,d}^*))$ and $g_t(u^t; y)=(g_t(u^t_1; y_1), \ldots, g_t(u^t_n; y_n))$. The scalars $\sfb_t, \sfc_t$ are defined as
\begin{equation}
\sfb_t =\frac{1}{n}\sum_{i=1}^d f_t'(v_i^t; \, \ol{x}_{1,i}^*, \ol{x}_{2,i}^*), \qquad
\sfc_t = \frac{1}{n}\sum_{i=1}^n g_t'(u_i^t; y_i),
\label{eq:GAMP_onsager}
\end{equation}
where $f_t'$ and $g_t'$ each denote the derivative with respect to the first argument. 

An important feature of the GAMP algorithm is that as $d \to \infty$, the empirical distributions of the iterates $u^t$ and $v^{t+1}$ converge to the laws of well-defined scalar random variables $U_t$ and $V_{t+1}$, respectively. Specifically, for $t \ge 0$, let 
\begin{equation}
    U_t \coloneqq \mu_{1,t} G_1 + \mu_{2,t} G_2 + W_{U,t} , \quad 
    V_{t+1} \coloneqq \chi_{1,t+1} X_1 + \chi_{2,t+1} X_2 +  W_{V,t+1},
    \label{eq:UtVt_def}
\end{equation}
where $(G_1, G_2, W_{U,t}) \sim \normal(0,1) \otimes \normal(0,1) \otimes \normal(0, \sigma_{U,t}^2) $, and 
$(X_1, X_2, W_{V,t+1})  \sim \normal(0,1) \otimes \normal(0,1) \otimes \normal(0,\sigma_{V,t+1}^2)$. The random variables $X_1, X_2$ are distributed according to limiting laws of the signals $\xone, \xtwo$, and $G_1, G_2$ according to the limiting laws of $\Abar \xone, \Abar \xtwo$. Since $\xone, \xtwo$ are independent and uniformly distributed on the sphere, we have 
$X_1, X_2\iid\normal(0, 1)$.
The deterministic coefficients 
$(\mu_{1,t}, \mu_{2,t}, \sigma_{U,t}, \chi_{1,t+1}, \chi_{2, t+1}, \sigma_{V, t+1})$ are computed using the following state evolution recursion:
\begin{align}
   &  \mu_{1,t} = \frac{1}{\sqrt{\delta}} \E [ X_1 f_t(V_t; \, X_1, X_2) ], \quad
    \mu_{2,t} = \frac{1}{\sqrt{\delta}} \E[ X_2 f_t(V_t; \, X_1, X_2) ],  \label{eq:muU_update} \\
   &  \sigma_{U,t}^2 = \frac{1}{\delta}\E[ f_t(V_t; \, X_1, X_2)^2 ] - \mu_{1,t}^2 - \mu_{2,t}^2 \, ,  \nonumber \\
   &  \chi_{1,t+1} =  \sqrt{\delta} \left( \E[G_1 g_t(U_t; \tY) ] - \E[ g_t'(U_t; \tY)]  \mu_{1,t} \right),  \notag \\
   & \chi_{2,t+1} =  \sqrt{\delta} \left( \E [ G_2 g_t(U_t; \tY) ] - \E [g_t'(U_t; \tY)]  \mu_{2,t} \right), \nonumber \\
    & \sigma_{V,t+1}^2 = \E[g_t(U_t; \tY)^2 ]. \nonumber 
\end{align}
Here the random variable $\tY$ is given by 
\begin{equation}
    \tY= q(\eta G_1 + (1-\eta)G_2, \, \eps), \text{ where } (G_1, G_2, \eta, \eps) \sim \normal(0,1) \ot \normal(0,1) \ot \text{Bern}(\alpha) \ot P_{\eps}.
    \label{eq:G1G2Y_joint}
\end{equation}
The state evolution recursion is initialized in terms of the limiting correlation of the initializer $\tv^0$ with each of the signals $\xone$ and $\xtwo$. The existence of these limiting correlations is guaranteed by imposing the following condition on  $\tv^0$:

\begin{enumerate}[label=(A\arabic*)]
\setcounter{enumi}{\value{assctr}}
    \item\label[ass]{ass:gamp-init} The initializer $\tv^0 \in \bbR^d$ is independent of $\Abar $. Furthermore, there exists a Lipschitz $F_0: \bbR^2 \to \bbR$ such that  
    \begin{align}
        \lim_{d \to \infty}  \frac{\langle  \tv^0, \,  \Phi(\xone, \xtwo) \rangle}{d} 
        = \E[ F_0(X_1, X_2) \Phi(X_1, X_2) ] \ \text{ almost surely},
    \end{align}
    for any Lipschitz $\Phi:\reals^2 \to \reals$. 
    Here $X_1, X_2 \iid \normal(0,1)$.
\setcounter{assctr}{\value{enumi}}
\end{enumerate}

This assumption is typical in AMP algorithms \cite{Feng22AMPTutorial}, and our initializer for proving \Cref{thm:main-thm-joint-dist}  will be $\tv^0=0_d$, which trivially satisfies \Cref{ass:gamp-init}.
\Cref{ass:gamp-init} allows us to initialize the state evolution recursion as:
\begin{align}
\begin{split}
    &\mu_{1,0} =\frac{1}{\sqrt{\delta}} \lim_{d \to \infty} \frac{\langle \xone \, , \, \tv^0   \rangle}{d} = \frac{1}{\sqrt{\delta}} \E[ F_0(X_1, X_2) X_1 ], \\
    &\mu_{2,0} =\frac{1}{\sqrt{\delta}} \lim_{d \to \infty} \frac{\langle \xtwo \, , \, \tv^0   \rangle}{d} = \frac{1}{\sqrt{\delta}} \E[ F_0(X_1, X_2) X_2 ], \\
    &\sigma_{U,0}^2= \frac{1}{\delta} \lim_{d \to \infty} \frac{\normtwo{\tv^0}^2}{d} - \mu_{1,0}^2 - \mu_{2,0}^2 =  \frac{1}{\delta} \E[F_0(X_1, X_2)^2]  - \mu_{1,0}^2 - \mu_{2,0}^2.
\end{split}
    \label{eq:SE_gen_init}
\end{align}

 The sequences of random variables $(W_{U,t})_{t \geq 0}$ and 
$(W_{V,t+1})_{t \geq 0}$ in \Cref{eq:UtVt_def} are each jointly Gaussian with zero mean and the following covariance structure:
\begin{equation}
\E[W_{U,0} W_{U,t}]  = 
\frac{1}{\delta} \E[F_0(X_1, X_2) \, f_t(V_t; X_1, X_2) ] - \mu_{1,0} \mu_{1,t} - \mu_{1,0} \mu_{2,t}, \qquad t \geq 1, 
\label{eq:WV_corr_init}
\end{equation}
and for $r, t \geq 1$: 
\begin{align}
\label{eq:WV_corr}
& \E [ W_{V,r} W_{V,t} ]  = 
\E\left[g_{r-1}(U_{r-1}; \tY) \, g_{t-1}(U_{t-1}; \tY) \right],  \\
& \E[W_{U,r} W_{U,t}]   = 
 \frac{1}{\delta}\E [  f_r(V_r; \, X_1, X_2) f_t(V_t; \, X_1, X_2) ] - \mu_{1,r} \mu_{1,t} - \mu_{1,r} \mu_{2,t}.
\label{eq:WU_corr} 
\end{align}
Note that for $r=t$ we have $\E [ W_{U,t}^2 ] = \sigma_{U,t}^2$ and  $\E [W_{V,t+1}^2 ]=\sigma_{V,t}^2$.

The state evolution result for the GAMP  is stated in terms of pseudo-Lipschitz test functions (see \Cref{eq:PL2_prop}). 

\begin{proposition}[State evolution]
Consider the setup of \Cref{thm:main-thm-joint-dist}  and the GAMP iteration in \Cref{eq:gamp-eqn-general}, 
with initialization $\tv^0$ that satisfies \Cref{ass:gamp-init}. Assume that for $t \ge 0$, the functions $g_t: \reals^2 \to \reals$ and $f_{t+1}: \reals^3 \to \reals$ are Lipschitz.  Let $ g_1 \coloneqq \Abar \xone, g_2\coloneqq \Abar \xtwo $. 
Then, the following holds almost surely for any $\PL(2)$ function $\Psi: \reals^{t+3}  \to \reals$, for $t \geq 0$:
\begin{align}
& \lim_{n \to \infty}\frac{1}{n} \sum_{i=1}^n \Psi( g_{1, i}, g_{2, i}, u^t_i, u^{t-1}_i, \ldots, u^0_{i} ) = \E[ \Psi( G_1, G_2, \, U_t, \, U_{t-1}, \, \ldots, U_0) ], \label{eq:psiG} \\
& \lim_{d \to \infty} \frac{1}{d} \sum_{i=1}^d \Psi(\ol{x}_{1,i}^*, \ol{x}_{2,i}^*, v^{t+1}_i, v^{t}_i, \ldots, v^1_i) = \E[ \Psi(X_1, \, X_2, \, V_{t+1}, \, V_{t}, \, \ldots, V_1) ], \label{eq:psiX}
\end{align}
where the  distributions of the random vectors $( G_1, G_2, U_t, \ldots, U_{0})$ and $(X_1, X_2, V_{t+1}, \ldots, V_1)$ are given by the state evolution recursion in \Cref{eq:UtVt_def} to \eqref{eq:WU_corr}.
\label{prop:GAMP_SE}
\end{proposition}

The proof of the proposition, given in \Cref{sec:se-gamp-mixed-glm}, uses a reduction to an abstract AMP recursion with matrix-valued iterates for which a state evolution result was established in \cite{javanmard2013state}.

The result in \Cref{eq:psiX} is equivalent to the statement that the joint empirical  distribution of the rows of $(\xone, \xtwo, v^t, \ldots, v^1)$ converges in Wasserstein-$2$ distance to the joint law of 
$(X_1, X_2, V_t, \ldots, V_1)$ (see \cite[Corollary 7.21]{Feng22AMPTutorial}). A similar equivalence holds for the result in \Cref{eq:psiG}.

\begin{remark}
\label{rem:det_Onsager}
 The result in \Cref{prop:GAMP_SE} also applies to the  GAMP algorithm in which the memory coefficients $(\sfb_t, \sfc_t)$ in \Cref{eq:GAMP_onsager} are replaced with their deterministic limits $\bar{\sfb}_t, \bar{\sfc}_t$ computed via state evolution: 
 \begin{equation}
     \bar{\sfb}_t = \frac{1}{\delta} \E[ f_t'(V_t; \, X_1, X_2) ],  \qquad  \bar{\sfc}_t = \E[ g_t'(U_t; \tY) ].
     \label{eq:detOnsager}
 \end{equation}
This equivalence follows from an argument similar to  \cite[Remark 4.3]{Feng22AMPTutorial}.
\end{remark}

\subsection{GAMP as a method to compute the linear and spectral estimators}
\label{sec:GAMP_as_method}

Consider the GAMP iteration in \Cref{eq:gamp-eqn-general} with the initializer $\tv^0=0$, and the following choice of functions: 
\begin{equation}
\begin{split}
& g_0(u^0; y) = \sqrt{\delta} \cL(y),  \qquad f_1(v; \, \xone, \xtwo)= f(\xone, \xtwo), \\
 &  g_t(u; \, y) = \sqrt{\delta} \, u \, \cF(y), \quad f_{t+1}(v; \, \xone, \xtwo) = \frac{v}{\beta_{t+1}}, \quad t \ge 1,
   \label{eq:ft_gt_choice}
   \end{split}
\end{equation}
where $\cF: \reals \to \reals$ is bounded and Lipschitz,  $f: \reals \to \reals$ is Lipschitz,  and $\beta_{t+1}$ is a constant, defined iteratively for $t \ge 0$ via the state evolution equations below (\Cref{eq:beta_t_def}). To prove \Cref{thm:main-thm-joint-dist}, we will consider two different choices for the pair of functions $(f, \cF)$, in terms of the spectral preprocessing function $\cT$ (see \Cref{eq:F1x1,eq:F2x2}). 

With the above 
choice of $f_t, g_t$, the memory coefficients in \Cref{eq:GAMP_onsager} are given by
\begin{equation}\label{eq:coeffbct}
      \sfc_0= \sfb_1=0, \qquad
\sfc_t = \sqrt{\delta} \cdot \frac{1}{n}\sum_{i=1}^n \cF(y_i), \qquad \sfb_{t+1} =\frac{1}{\delta \beta_{t+1}}.
\end{equation}
Replacing the parameter $\sfc_t$  with its almost sure limit $\bar{\sfc}_t = \sqrt{\delta} \, \E[\cF(\tY)]$, the GAMP iteration becomes %
\begin{equation}\label{eq:newGAMP}
    \begin{split}
    & u^0 =0, \quad v^1= \Abar^\top \cL(y), \\
    & u^1= \frac{1}{\sqrt{\delta}} \Abar f(\xone, \xtwo), \quad v^2 = \frac{1}{\sqrt{\delta}} \Abar^\top F u^1 - \sqrt{\delta} \E[ \cF(\tY) ] f(\xone,\xtwo), \\
    & 
    u^{t} = \frac{1}{\sqrt{\delta} \, \beta_{t}} \paren{\Abar  v^{t} \, -  \, F u^{t-1}},  \quad  
    v^{t+1} =  \Abar^\top F u^t - \frac{\sqrt{\delta}}{\beta_t} \, \E[ \cF(\tY) ] \,  v^t, \qquad 
    t \ge 2,
    \end{split}
\end{equation}
where $F = \diag(\cF(y_1), \ldots, \cF(y_n))$.  With $f_t, g_t$ given by \Cref{eq:ft_gt_choice}, the initialization for the state evolution in \Cref{eq:muU_update} to \eqref{eq:SE_gen_init} is:
\begin{align}
     & \mu_{1,0} =\mu_{2,0}= \sigma_{U,0}^2=0, \nonumber \\
     &\chi_{1,1} = \delta \E[ G_1 \cL(\tY)], \quad 
     \chi_{2,1} = \delta \E[ G_2 \cL(\tY)], \quad 
     \sigma_{V,1}^2 = \delta \E[\cL(\tY)^2 ],      \nonumber \\
    & \mu_{1,1} = \frac{1}{\sqrt{\delta}}\expt{X_1 f(X_1, X_2)}, \quad \mu_{2,1} = \frac{1}{\sqrt{\delta}}\expt{X_2 f(X_1, X_2)},  \nonumber \\
     & \sigma_{U,1}^2 = \frac{1}{\delta} \expt{f(X_1, X_2)^2} - \mu_{1,1}^2 - \mu_{2,1}^2,
     \label{eq:SE_lin_init}
\end{align}
where the joint distribution of $(G_1, G_2, \tY)$  is given by \Cref{eq:G1G2Y_joint}. Furthermore, for $t \ge 1$:
\begin{align}
    & \chi_{1,t+1} = \delta \mu_{1,t} \, \E [ \cF(\tY)(G_1^2-1) ], \qquad 
     \chi_{2,t+1} = \delta \mu_{2,t} \,  \E [ \cF(\tY)(G_2^2-1) ], \nonumber \\
    & \sigma_{V,t+1}^2 =  \delta \paren{ \mu_{1,t}^2 \E [\cF(\tY)^2 G_1^2] + \mu_{2,t}^2 \E[\cF(\tY)^2 G_2^2]  +  \sigma_{U,t}^2 \E[ \cF(\tY)^2] } ,
    \label{eq:chi_sigmaV_lin} \\
    & \beta_{t+1} \coloneqq \sqrt{\chi_{1,t+1}^2 + \chi_{2,t+1}^2 + \sigma_{V,t+1}^2} \, , \label{eq:beta_t_def} \\ %
    & \mu_{1,t+1} = \frac{\chi_{1,t+1}}{\sqrt{\delta} \beta_{t+1}}, \qquad \mu_{2,{t+1}} = \frac{\chi_{2,t+1}}{\sqrt{\delta} \beta_{t+1}}, \qquad \sigma_{U,t+1}^2= \frac{\sigma_{V,t+1}^2}{\delta \beta_{t+1}^2}. \label{eq:mu_sigmaU_lin} 
\end{align}

First note that the iterate $v^1$ coincides with the linear estimator $\xlin$ in \Cref{eqn:def-lin-estimator}. We will show that in the high-dimensional limit the iterate $v^t$ is  aligned with an eigenvector of the matrix $M \coloneqq \Abar^\top F(\sqrt{\delta} \beta_{\infty}I_n + F)^{-1} \Abar $, as $t \to \infty$. (\Cref{lem:SE_fixed_pts} shows that  $\beta_{\infty} = \lim\limits_{t \to \infty} \beta_t$ is well-defined for our choices of $\cF$ and initializations.) For a heuristic justification of this claim, assume the iterates   $u^t, v^t$ converge to the limits $u^\infty, v^\infty$ in the sense that 
$\lim\limits_{t \to \infty} \lim\limits_{d \to \infty} \frac{1}{d} \normtwo{u^t - u^{\infty}}^2 =0$ and  $\lim\limits_{t \to \infty} \lim\limits_{d \to \infty} \frac{1}{d} \normtwo{v^t - v^{\infty}}^2 =0$. Then, from \Cref{eq:newGAMP} these limits satisfy
\begin{equation}
    u^\infty = \frac{1}{\sqrt{\delta} \, \beta_\infty} \paren{ \Abar  v^\infty \, -  \, F u^{\infty} },  \qquad 
    v^{\infty} =  \Abar^\top F u^\infty - \frac{\sqrt{\delta}}{\beta_\infty} \, \E[ \cF(\tY) ] \,  v^\infty,
\end{equation}
which after simplification, can be written as:
\begin{equation}
    v^{\infty}\left( 1 + \frac{\sqrt{\delta}}{\beta_\infty} \E[ \cF(\tY)] \right) = \Abar^{\top}F (\sqrt{\delta}\beta_{\infty}I_n + F)^{-1}\Abar  v^{\infty}.
    \label{eq:evector_eqn}
\end{equation}
Therefore, $v^{\infty}$ is an eigenvector of the matrix
$\Abar^{\top}F (\sqrt{\delta}\beta_{\infty}I_n + F)^{-1}\Abar $, and the GAMP iteration \Cref{eq:newGAMP} is effectively a power method.

We wish to obtain via GAMP the two leading eigenvectors of the matrix $\Abar^{\top}T\Abar $, so the heuristic above indicates that we should  choose 
\begin{align}
    \cF(y) &= \frac{c \sqrt{\delta} \beta_\infty \cT(y)}{1 - c \cT(y)} , \notag
\end{align}
so that $F (\sqrt{\delta}\beta_{\infty}I_n + F)^{-1} = c\, T$, for some constant $c$. 
\hl{For estimating the $i$-th signal, we fix the values of $ \beta_\infty $ and $ c $ by enforcing the following two constraints: }
\begin{align}
    1 &= \lim_{d\to\infty} \frac{\normtwo{\wt{v}^\infty}}{\sqrt{d}} = \lim_{d\to\infty} \frac{\normtwo{v^\infty}}{\beta_\infty \sqrt{d}} , \notag \\
    \frac{1}{\sqrt{\delta}} &= \lim_{d\to\infty} \frac{\normtwo{v^\infty}}{\sqrt{d}} = \sqrt{\delta} \expt{ \frac{c \sqrt{\delta} \beta_\infty \cT(\wt{Y})}{1 - c \cT(\wt{Y})} (G_i^2 - 1) } , \notag 
\end{align}
\hl{where the last equality in the second line is by state evolution (formally shown in \mbox{\Cref{lem:SE_fixed_pts}}). 
Upon simplifications, the above two conditions are equivalent to }
\begin{align}
&&
    \beta_\infty &= \frac{1}{\sqrt{\delta}} , &
    c &= \frac{1}{\lambda^*(\delta_i)} , & 
& \notag 
\end{align}
\hl{which in turn motivates the choice of $ \cF $: }
\begin{align}
    \cF(y) &= \frac{\cT(y)}{\lambda^*(\delta_i) - \cT(y)} . \notag
\end{align}

Formally, we analyze the iteration in \Cref{eq:newGAMP} with two choices for the function $\cF(y)$ and  initialization $\tv^0$: 
\begin{align}
    &\text{ Choice 1}: \quad \cF_1(y) \coloneqq \frac{\cT(y)}{\lambda^*(\delta_1) - \cT(y)}, \quad f(\xone, \xtwo)=\xone, \label{eq:F1x1} \\
     &\text{ Choice 2}: \quad  \cF_2(y) \coloneqq \frac{\cT(y)}{\lambda^*(\delta_2) - \cT(y)}, \quad f(\xone, \xtwo)=\xtwo. \label{eq:F2x2}
\end{align}
Here, we recall that, for $i\in\{1, 2\}$, $ \lambda^*(\delta_i) $ is the unique solution of $ \zeta(\lambda;\delta_i) = \phi(\lambda) $ (see page \pageref{def:lambdastard}).
 The initializations in \Cref{eq:F1x1,eq:F2x2} are not feasible in practice since they depend on the unknown signals $\xone$ and $\xtwo$, but this is not an issue as we use the GAMP in \Cref{eq:ft_gt_choice} only as a proof technique. 
 
 We now examine the state evolution recursion in \Cref{eq:chi_sigmaV_lin,eq:beta_t_def,eq:mu_sigmaU_lin} under each of these choices.
\paragraph{Choice 1} From \Cref{eq:SE_lin_init}, this corresponds to the initialization
 \begin{align}
     &\chi_{1,1} = \delta \E[ G_1 \cL(\tY)], \ 
     \chi_{2,1} = \delta \E[ G_2 \cL(\tY)], \ 
     \sigma_{V,1}^2 = \delta \E[\cL(\tY)^2 ], \ 
     \mu_{1,1} = \frac{1}{\sqrt{\delta}}, \ 
     \mu_{2,1} = \sigma_{U,1}^2 = 0.
     \label{eq:SElin_init_x1}
\end{align}
 For $t \ge 1$, the state evolution equations in \Cref{eq:chi_sigmaV_lin,eq:beta_t_def,eq:mu_sigmaU_lin} reduce to:
\begin{equation}
\begin{split}
     &  \chi_{1,t+1} = \delta \mu_{1,t} \, \E [ \cF_1(\tY)(G_1^2-1) ], \quad \sigma_{V,t+1}^2 = \delta \left( \mu_{1,t}^2 \E[ \cF_1(\tY)^2 G_1^2 ]  \, + \, \sigma_{U,t}^2 \E[ \cF_1(\tY)^2 ] \right), \\
     & \beta_{t+1} = \sqrt{\chi_{1,t+1}^2 + \sigma_{V,t+1}^2} \, , \qquad
      \mu_{1,t+1} = \frac{\chi_{1,t+1}}{\sqrt{\delta} \beta_{t+1}} \, , \qquad  \sigma_{U,t+1}^2= \frac{\sigma_{V,t+1}^2}{\delta \beta_{t+1}^2} \, ,  \\
\end{split}
\label{eq:mu1_chi1_lin}
\end{equation}
and $\mu_{2,t+1} = \chi_{2,t+1} =0$ for $t \ge 1$.
Using this in \Cref{prop:GAMP_SE}, we obtain that:
\begin{equation}
  \lim_{d \to \infty} \frac{\la \xone, v^{1} \ra}{d} = \chi_{1,1}, \; 
  \lim_{d \to \infty} \frac{\la \xtwo, v^{1} \ra}{d} = \chi_{2,1}, \; 
  \lim_{d \to \infty} \frac{\la \xone, v^{t+1} \ra}{d} = \chi_{1,t+1}, \; 
  \lim_{d \to \infty} \frac{\la \xtwo, v^{t+1} \ra}{d} = 0,
\end{equation}
for $t \ge 1$. Thus, when initialized with $f(\xone, \xtwo) =\xone$, the GAMP iterates $\{ v^{t+1} \}_{t \ge 1}$ are asymptotically uncorrelated with the signal $\xtwo$.

\paragraph{Choice 2} This corresponds to the initialization \begin{align}
     &\chi_{1,1} = \delta \E[ G_1 \cL(\tY)], \
     \chi_{2,1} = \delta \E[ G_2 \cL(\tY)], \ 
     \sigma_{V,1}^2 = \delta \E[\cL(\tY)^2 ], \ 
    \mu_{2,1} = \frac{1}{\sqrt{\delta}}, \ 
     \mu_{1,1}=\sigma_{U,1}^2 = 0.
     \label{eq:SElin_init_x2}
\end{align}
The state evolution  equations are:  $\mu_{1,t+1} = \chi_{1,t+1} =0$ for $t \ge 1$, and
\begin{equation}
\begin{split}
     &  \chi_{2,t+1} = \delta \mu_{2,t} \, \E [ \cF_2(\tY)(G_2^2-1) ], \quad \sigma_{V,t+1}^2 = \delta \left( \mu_{2,t}^2 \E[ \cF_2(\tY)^2 G_2^2 ]  \, + \, \sigma_{U,t}^2 \E[ \cF_2(\tY)^2 ] \right), \\
    &   \beta_{t+1} = \sqrt{\chi_{2,t+1}^2 + \sigma_{V,t+1}^2} \, , \qquad  
       \mu_{2,t+1} = \frac{\chi_{2,t+1}}{\sqrt{\delta} \beta_{t+1}}, \qquad  \sigma_{U,t+1}^2= \frac{\sigma_{V,t+1}^2}{\delta \beta_{t+1}^2} \, .
\end{split}
\label{eq:mu2_chi2_lin}
\end{equation}
Using this in \Cref{prop:GAMP_SE}, we obtain that for $t \ge 1$,
\begin{equation}
    \lim_{d \to \infty} \frac{\la \xone, v^{1} \ra}{d} = \chi_{1,1}, \;
    \lim_{d \to \infty} \frac{\la \xtwo, v^{1} \ra}{d} = \chi_{2,1}, \; 
    \lim_{d \to \infty} \frac{\la \xone, v^{t+1} \ra}{d} = 0, \;
    \lim_{d \to \infty} \frac{\la \xtwo, v^{t+1} \ra}{d} = \chi_{2,t+1}.
\end{equation}
The following lemma gives the fixed point of state evolution under choices 1 and 2.
\begin{lemma}[Limiting values of state evolution parameters] 
\leavevmode
Consider the state evolution recursion under choice $i \in \{1,2\}$. Assume assume that $\E [ \cF_i(\tY)(G_i^2-1) ]>0$ and $\delta >  \frac{\E [ \cF_i(\tY)^2 ]}{ (\E[ \cF_i(\tY)(G_i^2-1)])^2} $. Then, as $t\to \infty$ the state evolution parameters $(\chi_{i,t}, \sigma_{V,t}^2)$ converge to the fixed point $(\tilde{\chi}_i, \tilde{\sigma}_i^2 )$, where
\begin{equation}
\tilde{\chi}_i = \sqrt{\frac{\tilde{\beta}_i^2(\tilde{\beta}_i^2 -  \E[\cF_i(\tY)^2])}{\tilde{\beta}_i^2 + \E[\cF_i(\tY)^2G_i^2] - \E[\cF_i(\tY)^2] }}, \qquad  
\tilde{\sigma}_i^2 = \frac{\tilde{\beta}_i^2 \E[\cF_i(\tY)^2 G_i^2]}{ \tilde{\beta}_i^2 + \E[\cF_i(\tY)^2G_i^2] - \E[\cF_i(\tY)^2]}, 
\label{eq:FP1}
\end{equation}
and 
\begin{equation}
\tilde{\beta}_i^2 = \tilde{\chi}_i^2 + \tilde{\sigma}_i^2 =  \delta\, 
( \E[\cF_i(\tY)(G_i^2-1)])^2.
\label{eq:beta1_def}
\end{equation}
\label{lem:SE_fixed_pts}
\end{lemma}
\begin{proof}

The proof  is identical to that of Lemma 5.2 in \cite{mondelli2021optimalcombination}, which analyzes GAMP for a \emph{non-mixed} GLM with $f_t, g_t$ given by \Cref{eq:ft_gt_choice}.  The state evolution recursion under choice 1 in \Cref{eq:SElin_init_x1,eq:mu1_chi1_lin} has the same form for all values of $\alpha \in [1/2,1)$. 
The value of $\alpha$  affects the recursion only through  the joint distribution of
$(\wt{Y}, G_1) = (q(\eta G_1 + (1-\eta)G_2,\eps), G_1)$, where  $\eta \sim \bern(\alpha)$. The proof of Lemma 5.2 in \cite{mondelli2021optimalcombination} does not depend on this joint distribution and applies for any $\alpha$ such that the lower bound on $\delta$ in the statement of the first part  is satisfied. The argument for  choice 2, where the joint distribution determining the state evolution in \Cref{eq:SElin_init_x2,eq:mu2_chi2_lin}  is $(\wt{Y}, G_2) = (q(\eta G_1 + (1-\eta)G_2,\eps), G_2)$, is identical. 
\end{proof}

It is convenient to express the state evolution fixed points in \Cref{lem:SE_fixed_pts} in terms of the joint law of $(G, Y)$, where $Y = q(G, \eps)$, with $G \sim \normal(0,1)$ and  $\eps \sim P_{\eps}$ are independent. Recalling the joint law of $(\tY, G_1, G_2)$ given in \Cref{eq:G1G2Y_joint} and the definitions of $\cF_1,\cF_2$ in \Cref{eq:F1x1,eq:F2x2}, we have 
\begin{align}
    \E[\cF_1(\tY)] &= \E[\cF_1(Y)] = \E\left[ \frac{\cT(Y)}{\lambda^*(\delta_1) - \cT(Y)} \right], \notag \\ 
    \E[\cF_1(\tY)^2] &= \E[\cF_1(Y)^2] = \E\left[ \frac{\cT(Y)^2}{ (\lambda^*(\delta_1) - \cT(Y))^2} \right ], \notag \\
  \E[\cF_1(\tY)G_1^2] &= \alpha \E[\cF_1(q(G_1, \eps)) G_1^2 ] + (1-\alpha)  \E[\cF_1(q(G_2, \eps)) G_1^2 ] \notag \\
  &= \alpha \E \left[ \frac{\cT(Y) G^2}{\lambda^*(\delta_1) - \cT(Y)} \right] + (1-\alpha) \E \left[ \frac{\cT(Y)}{\lambda^*(\delta_1) - \cT(Y)} \right] \notag \\
  &= \frac{1}{\delta} + \E \left[ \frac{\cT(Y)}{\lambda^*(\delta_1) - \cT(Y)} \right],
  \label{eq:SE_GY1}
 \end{align}
where the last equality holds because $\E \left[ \frac{\cT(Y) (G^2-1)}{\lambda^*(\delta_1) - \cT(Y)} \right]  = \frac{1}{\delta_1}$ from \Cref{eqn:def-lami-star-supcrit-explicit}, and $ \delta_1 = \alpha \delta$. 
Similarly, we obtain
\begin{align}
   &  \E[\cF_2(\tY)]  = \E[\cF_2(Y)] = 
 \E\left[ \frac{\cT(Y)}{\lambda^*(\delta_2) - \cT(Y)} \right], \quad \E[\cF_2(\tY)^2] =
  \E\left[ \frac{\cT(Y)^2}{ (\lambda^*(\delta_2) - \cT(Y))^2} \right ], \nonumber \\
  & \E[\cF_2(\tY)G_2^2] 
  = \frac{1}{\delta} + \E \left[ \frac{\cT(Y)}{\lambda^*(\delta_2) - \cT(Y)} \right].
   \label{eq:SE_GY21}
\end{align}
Using \Cref{eq:SE_GY1,eq:SE_GY21}, the formula for $\tbeta_i^2$ in \Cref{eq:beta1_def} becomes:
\begin{align}
     \tbeta_i^2 =  \frac{1}{\delta},  \qquad i\in\{1,2\}.
     \label{eq:tbeta1_def}
\end{align}
 We similarly obtain 
\begin{align}
    \tilde{\chi}_i = \frac{\rho_i^{\spec}}{\sqrt{\delta}}, \quad \tilde{\sigma}_1^2 = \frac{1- (\rho_i^{\spec})^2}{\delta}, 
     \qquad i\in\{1,2\}.
    \label{eq:tchi_tsigma_formulas}
\end{align}
where $\rho_1^{\spec}, \rho_2^{\spec}$ are defined in \Cref{eqn:def-asymp-corr}.

\paragraph{Proof heuristic} Let us revisit the heuristic sanity-check in \Cref{eq:evector_eqn}. For $i \in \{1,2\}$, under choice $i$  with $\cF= \cF_i$, $F= F_i\coloneqq\diag(\cF_i(y_1), \ldots, \cF_i(y_n))$, and $\beta_{\infty} = \tilde{\beta}_i$, by using  the  formulas above for  $\tilde{\beta}_i$ and $\E[\cF_i(\tY)]$, \Cref{eq:evector_eqn} becomes:
 \begin{equation}
    v^{\infty}\left( 1 +  \delta \E\left[ \frac{\cT(Y)}{\lambda^*(\delta_i) - \cT(Y)} \right] \right) = \Abar^{\top}F_i ( I_n + F_i)^{-1}\Abar  v^{\infty} = \frac{1}{\lambda^*(\delta_i)} \Abar^{\top} T \Abar v^\infty,
    \label{eq:evector_eqn_1}
\end{equation}
where we recall that $T= \diag(\cT(y_1), \ldots, \cT(y_n))$. 
Therefore, with choice $i$, \Cref{eq:evector_eqn_1} suggests that the GAMP iterate converges to an eigenvector of $\Dbar = \Abar^\top T \Abar $ corresponding to the eigenvalue  $\lambda^*(\delta_i)\left( 1 +  \delta \E\left[ \frac{\cT(Y)}{\lambda^*(\delta_i) - \cT(Y)} \right] \right)$. 
Moreover, when $\lambda^*(\delta_i) > \ol{\lambda}(\delta)$, \Cref{thm:main-thm-eigval} and \Cref{rk:explicit-formula-eigval} tell us that the leading eigenvalue of $\Dbar$ converges to:
\begin{equation}\label{eq:limeig1}
    \lambda_i( \Dbar ) \stackrel{d \to \infty}{\longrightarrow} \lambda^*(\delta_i)\left( 1 +  \delta \E\left[ \frac{\cT(Y)}{\lambda^*(\delta_i) - \cT(Y)} \right] \right).
\end{equation}
%
Therefore, \Cref{eq:evector_eqn_1} indicates that the GAMP iterates under each choices $1$ and $2$ converge to the eigenvectors corresponding to the two largest eigenvalues of $\Dbar$, when $\lambda^*(\delta_1) > \lambda^*(\delta_2) > \ol{\lambda}(\delta)$. We now make this claim rigorous.

\subsection{Proof of \Cref{thm:main-thm-joint-dist}} \label{subsec:proof_main}

Consider the GAMP iteration in \Cref{eq:newGAMP} for $t \ge 2$. By substituting the expression for $u^t$  in the $v^{t+1}$ update,  the iteration can be rewritten as follows:
\begin{align}
 u^t &= \frac{1}{\sqrt{\delta} \, \beta_t} \paren{ \Abar  v^t \, -  \, F u^{t-1} }, 
\qquad  v^{t+1}  = \frac{1}{\sqrt{\delta} \beta_t}\sqrbrkt{ \paren{\Abar^\top F \Abar   - \delta \E[\cF(\tY)] \, I_d} v^t \, - \, \Abar^\top F^2 u^{t-1} }. \label{eq:GAMPrewrite1}
\end{align}
In the remainder of the proof, we will assume that $t \geq 2$. Define
\begin{align}
        e_1^t &= u^{t}-u^{t-1},\label{eq:err1}\\
        e_2^t &= v^{t+1}-v^{t}.\label{eq:err2}
\end{align}
By combining \Cref{eq:err1} with \Cref{eq:GAMPrewrite1}, we have
\begin{equation}\label{eq:ut1new}
    u^{t-1}=(F+\sqrt{\delta}\beta_tI_n)^{-1}\Abar v^t-\sqrt{\delta}\beta_t(F+\sqrt{\delta}\beta_tI_n)^{-1}e_1^t.
\end{equation}
Substituting the expression for $u^{t-1}$  in \Cref{eq:ut1new} into \Cref{eq:GAMPrewrite1} and recalling from \Cref{eq:SE_GY1,eq:SE_GY21} that $\expt{\cF(\wt{Y})} = \expt{\cF(Y)}$, we obtain:
\begin{align}
    v^{t+1}  &= \left(\Abar^\top F(F+\sqrt{\delta}\beta_t I_n)^{-1} \Abar   - \frac{\sqrt{\delta} \E[\cF(Y)]}{\beta_t} \, I_d \right) v^t +\Abar^\top F^2(F+\sqrt{\delta}\beta_tI_n)^{-1}e_1^t \nonumber \\
    &= \left(\Abar^\top F(F+I_n)^{-1} \Abar   -\delta \E[\cF(Y)] \, I_d \right) v^t  \notag \\
    &\quad +(1-\sqrt{\delta}\beta_t)\Abar^\top F(F+I_n)^{-1}(F+\sqrt{\delta}\beta_tI_n)^{-1} \Abar v^t  \nonumber  \\
& \quad   + \delta\E[\cF(Y)]\left(1-\frac{1}{\sqrt{\delta}\beta_t}\right) v^t+\Abar^\top F^2(F+\sqrt{\delta}\beta_t I_n)^{-1}e_1^t.
\label{eq:GAMPrewrite3}
\end{align}
Let
\begin{equation}\label{eq:GAMPrewrite4}
    e_3^t =  \left(\Abar^\top F(F+I_n)^{-1} \Abar   -(\delta \E[\cF(Y)]+1) \, I_d \right) v^t.
\end{equation}
Using this in \Cref{eq:GAMPrewrite3} along with \Cref{eq:err2}, we obtain
\begin{equation}\label{eq:deferr3}
    \begin{split}
        e_3^t&=e_2^t-(1-\sqrt{\delta}\beta_t)\Abar^\top F(F+I_n)^{-1}(F+\sqrt{\delta}\beta_tI_n)^{-1} \Abar v^t \\
& \quad   - \delta\E[\cF(Y)]\left(1-\frac{1}{\sqrt{\delta}\beta_t}\right) v^t-\Abar^\top F^2(F+\sqrt{\delta}\beta_tI_n)^{-1}e_1^t.
    \end{split}
\end{equation}

We now prove the two claims of \Cref{thm:main-thm-joint-dist} via choices 1 and 2, respectively. 
All the limits in the remainder of the proof hold almost surely, so we won't specify this explicitly.

\subsubsection{Proof of \Cref{eq:psiX1joint}}
Consider the GAMP algorithm with choice 1, as defined in \Cref{eq:F1x1}. With $\cF(y) = \cF_1(y)$, we have:
\begin{equation}
    F(F + I_n)^{-1} = \frac{1}{\lambda^*(\delta_1)}  T , \qquad \E[\cF(Y)] = \E\left[ \frac{\cT(Y)}{ \lambda^*(\delta_1) - \cT(Y)} \right].
    \label{eq:choice_1_params}
\end{equation}
Recalling the notation $\Dbar = \Abar^{\top} T \Abar $, let us decompose $v^t$ into a component in the direction of $\eigvone$ plus an orthogonal component $r_1^t$: 
\begin{equation}\label{eq:projvt1}
    v^t = \xi_{1,t}\,\eigvone+r_1^t,
\end{equation}
where $\xi_{1,t}=\la v^t, \eigvone \ra$.  Substituting \Cref{eq:projvt1} in the definition of $e_3^t$ in \Cref{eq:GAMPrewrite4} and using \Cref{eq:choice_1_params}, we obtain
\begin{multline}
  \left( \frac{\Dbar}{\lambda^*(\delta_1)} -   \E \left[   \frac{ \delta \cT(Y)}{ \lambda^*(\delta_1) - \cT(Y)}  +  1 \right] I_d \right) r_1^t \\
  =  e_3^t \, + \, \xi_{1,t} \left(   \delta \E \left[  \frac{\cT(Y)}{ \lambda^*(\delta_1) - \cT(Y)} \right] +  1 - \frac{\lambda_1(\Dbar)}{\lambda^*(\delta_1)}\right) \eigvone .
  \label{eq:r1t_exp}
\end{multline}

The idea of the proof is to prove that $ \lim\limits_{t \to \infty}\lim\limits_{d \to \infty} \| r_1^t \|_2^2/d = 0$, which from \Cref{eq:projvt1} implies that  the GAMP iterate is aligned with $\eigvone$ in the limit. To show this, we first claim that for all sufficiently large $n$:%
\begin{align}
     \normtwo{ \left( \frac{\Dbar}{\lambda^*(\delta_1)} -   \E \left[   \frac{ \delta \cT(Y)}{ \lambda^*(\delta_1) - \cT(Y)}  +  1 \right] I_d \right) r_1^t}  \ge C  \normtwo{r_1^t},
\label{eq:r1t_LB}
\end{align}
for some  constant $C >0$ that does not depend on $n$. We then consider the right side of \Cref{eq:r1t_exp} and show that under choice 1:
\begin{align}
    \lim_{t \to \infty} \lim_{d \to \infty} \, \frac{1}{\sqrt{d}} \normtwo{e_3^t \, + \, \xi_{1,t} \left(   \delta \E \left[  \frac{\cT(Y)}{ \lambda^*(\delta_1) - \cT(Y)} \right] +  1 - \frac{\lambda_1(\Dbar)}{\lambda^*(\delta_1)}\right) \eigvone}  =0.
    \label{eq:xi1_e3_lim}
\end{align}
We now derive the result in \Cref{eq:psiX1joint}
using  \Cref{eq:r1t_LB,eq:xi1_e3_lim}, deferring the proofs of  these claims to the end of the section. Using \Cref{eq:r1t_LB,eq:xi1_e3_lim} in \Cref{eq:r1t_exp}, we have that 
\begin{equation}
    \lim_{t \to \infty}\lim_{d \to \infty} \frac{ \normtwo{ r_1^t }^2}{d} =0.
    \label{eq:r1t_lim}
\end{equation}
From the decomposition of $v^t$ in \Cref{eq:projvt1}, we have
\begin{equation}
\normtwo{v^t}^2 = \xi_{1,t}^2 \, + \, \normtwo{r_{1}^t}^2 \, ,
\label{eq:vtnorm_decomp}
\end{equation}
since $r_{1}^t$ is orthogonal to $\eigvone$ and $\normtwo{\eigvone}=1$. From \Cref{prop:GAMP_SE},  we have
\begin{align}
 \lim_{d \to \infty}   \frac{ \normtwo{v^t}^2}{d} = \E[V_t^2]=\beta_t^2, \qquad t\ge1.
\end{align}
Moreover, from \Cref{lem:SE_fixed_pts} and \Cref{eq:tbeta1_def}, under choice 1,  $\lim\limits_{t \to \infty} \beta_t^2 = \tbeta_1^2 = \frac{1}{\delta}$. Therefore,
\begin{equation}
    \lim_{t \to \infty}  \lim_{d \to \infty} \frac{ \normtwo{v^t}^2}{d} = \frac{1}{\delta}.
    \label{eq:vt_def}
\end{equation}
Combining this with \Cref{eq:vtnorm_decomp,eq:r1t_lim} yields
\begin{equation}
   \lim_{t \to \infty}  \lim_{d \to \infty} \frac{\xi_{1,t}^2}{d} = \frac{1}{\delta}.
   \label{eq:xi1lim}
\end{equation}
 Using \Cref{eq:xi1lim,eq:r1t_lim} in \Cref{eq:projvt1}, and recalling the definition of $x^{\spec}_1$ from the statement of \Cref{thm:main-thm-joint-dist}, we have
\begin{align}
 \lim_{t\to \infty} \lim_{d \to \infty} \frac{\| \sqrt{\delta} \, v^t -  x_1^{\spec} \|_2}{\sqrt{d}} =0. 
  \label{eq:vt_x1spec_diff}
\end{align}
  For any $\PL(2)$ function $\Psi: \reals^3 \to \reals$, by an application of Cauchy-Schwarz inequality,  we have that \cite[Lemma 7.24]{Feng22AMPTutorial} 
\begin{align}
   &  \abs{\frac{1}{d} \sum_{i=1}^d \Psi( \ol{x}^*_{1,i}, x^{\lin}_i,  x^{\spec}_{1,i})
    - \frac{1}{d} \sum_{i=1}^d \Psi( \ol{x}^*_{1,i}, x^{\lin}_i,  \sqrt{\delta} v^{t}_{i})}  \nonumber \\
    & \le C 
    \frac{\|\sqrt{\delta} v^t - x_1^{\spec}   \|_2}{\sqrt{d}}
    \left(1 + \frac{\| \xone \|_2}{\sqrt{d}}   + \frac{\| x^{\lin} \|_2}{\sqrt{d}} + \frac{\| x_1^{\spec} \|_2}{\sqrt{d}} + \frac{\| v^t \|_2}{\sqrt{d}} \right).
    \label{eq:Psi_diff1}
\end{align}
We have that $\| \xone \|_2 = \| x^{\lin} \|_2 = \| x_1^{\spec} \|_2  =\sqrt{d}$, by the definitions in the theorem statement. Therefore, using \Cref{eq:vt_def,eq:vt_x1spec_diff} in \Cref{eq:Psi_diff1}, we obtain:
\begin{align}
    \lim_{d \to \infty} \, \abs{\frac{1}{d} \sum_{i=1}^d \Psi( \ol{x}^*_{1,i}, x^{\lin}_i,  x^{\spec}_{1,i})
    - \frac{1}{d} \sum_{i=1}^d \Psi( \ol{x}^*_{1,i}, x^{\lin}_i, \sqrt{\delta} v^{t}_{i})} = 0.
\end{align}
Recall from \Cref{eq:newGAMP} that the GAMP iterate $v^1 = \xlin$, and $x^{\lin} = \sqrt{d} \, \xlin/\normtwo{\xlin}$.  From   \Cref{prop:GAMP_SE}, we have that
$\lim\limits_{d \to \infty} \frac{\normtwo{\xlin}}{d} = \sqrt{\E[V_1^2]}$. Using
 \Cref{prop:GAMP_SE} again, we have that 
\begin{align}
    \lim_{d \to \infty} \frac{1}{d} \sum_{i=1}^d \Psi\left( \ol{x}^*_{1,i}, x^{\lin}_i, \sqrt{\delta} v^{t}_{i} \right) =   \expt{\Psi\left(X_1, \, \frac{V_1}{\sqrt{\E[V_1^2]}}
 ,\,  \sqrt{\delta} V_t \right)  } .
 \label{eq:Psi_conv1}
\end{align}
From the definitions of $V_1, V_t$ in \Cref{eq:UtVt_def}, and the state evolution equations for choice 1 in \Cref{eq:SElin_init_x1,eq:mu1_chi1_lin}, we have 
\begin{align}
    & V_1 = \chi_{1,1} X_1 + \chi_{2,1} X_2 + W_{V,1}, \quad 
V_t= \chi_{1,t} X_1 \,   +  \, W_{V,t},  \quad t \ge 2.
\end{align}
Here $W_{V,1} \sim \normal(0, \, \delta \E[\cL(\tY)^2 ] )$ and $W_{V,t} \sim \normal(0, \sigma_{V,t}^2)$ are independent of $(X_1, X_2)$, and from \Cref{eq:WV_corr,eq:ft_gt_choice}, their covariance is given by 
 \begin{equation}
     \begin{split}
             \expt{W_{V,1}W_{V,t}}= \delta \expt{ U_{t-1} \cL(\tY) \cF(\tY)} = \delta \alpha \mu_{1, t-1}  \expt{ G \cL(Y) \cF(Y)},
     \end{split}
 \end{equation}
 where in the last line we have used that $\mu_{2,t-1} =0$ under choice 1.
Hence, for $t \ge 2$, \Cref{eq:Psi_conv1} becomes
\begin{align}
    & \lim_{d \to \infty} \frac{1}{d} \sum_{i=1}^d \Psi\left( \ol{x}^*_{1,i}, x^{\lin}_i, \sqrt{\delta} v^{t}_{i} \right) \notag \\
    &= \expt{\Psi\left(X_1, \, \frac{\chi_{1,1} X_1 + \chi_{2,1} X_2 + W_{V,1}}{\sqrt{\chi_{1,1}^2 + \chi_{2,1}^2 + \delta \E[\cL(Y)^2]}} ,\, \sqrt{\delta} (\chi_{1,t} X_1   +  W_{V,t}) \right) }.
 \label{eq:Psi_conv2}
\end{align}
To obtain the result in \Cref{eq:psiX1joint}, we take $t \to \infty$  on both sides above and show that 
\begin{equation}
\begin{split}
      &   \lim_{t \to \infty} \expt{\Psi\left(X_1, \, \frac{\chi_{1,1} X_1 + \chi_{2,1} X_2 + W_{V,1}}{\sqrt{\chi_{1,1}^2 +  \chi_{2,1}^2 + \delta \E[\cL(Y)^2  ]}}
 ,\, \sqrt{\delta}(\chi_{1,t} X_1   +  W_{V,t}) \right)  } \\
&  = 
 \expt{\Psi\left(X_1, \, \frac{\chi_{1,1} X_1 + \chi_{2,1} X_2 + W_{V,1}}{\sqrt{\chi_{1,1}^2 +  \chi_{2,1}^2 + \delta \E[\cL(Y)^2  ]}}
 ,\, \sqrt{\delta}(\tilde{\chi}_{1} X_1   +  W_{V, \infty}) \right)  },
\end{split}
\label{eq:limExpfinal}
\end{equation}
where $(W_{V,1}, W_{V, \infty})$ are jointly Gaussian with
\begin{equation}
W_{V, \infty} \sim \normal( 0, \tilde{\sigma}_{1}^2),  \qquad 
    \expt{W_{V,1} W_{V, \infty} } = \tilde{\chi}_1 \delta \alpha  \expt{ G \cL(Y) \cF(Y)}.
\end{equation}
Here $\tilde{\chi}_{1}$ and $\tilde{\sigma}_{1}^2$ are given by 
\Cref{eq:tchi_tsigma_formulas}. Using \Cref{lem:SE_fixed_pts}, we have
\begin{equation}
    \begin{split}
       &  \lim_{t \to \infty} \expt{W_{V,t}^2} = \lim_{t \to \infty} \sigma_{V,t}^2 = \tilde{\sigma}_1^2 = \expt{W_{V, \infty}^2},  \\
        &  \lim_{t \to \infty} \expt{W_{V,1}W_{V,t}}  = \delta \alpha  \expt{ G \cL(Y) \cF(Y)} \lim_{t \to \infty} \mu_{1, t-1}
        =       \tilde{\chi}_1   \delta \alpha  \expt{ G \cL(Y) \cF(Y)}  = \expt{W_{V,1} W_{V, \infty} } , 
    \end{split}
    \label{eq:rv_conv}
\end{equation}
where in the second line, we have used the formula for $\mu_{1, t-1}$ from \Cref{eq:mu1_chi1_lin} and that $\lim\limits_{t \to \infty} \beta_t =1/\sqrt{\delta}$ (from \Cref{eq:tbeta1_def}).  \Cref{eq:rv_conv} implies that the sequence of zero mean jointly Gaussian pairs $(W_{V,1}, W_{V,t})_{t \geq 1}$ converges in distribution to the jointly Gaussian pair $(W_{V,1}, W_{V, \infty})$. To show \Cref{eq:limExpfinal}, we use Lemma 4.5 in \cite{dumbgen2011}. We apply this result taking $Q_t$  to be the distribution of   
$$\left(X_1, \, \frac{\chi_{1,1} X_1 + \chi_{2,1} X_2 + W_{V,1}}{\sqrt{\chi_{1,1}^2 +  \chi_{2,1}^2 + \delta \E[\cL(Y)^2  ]}}
 ,\, \sqrt{\delta}( \chi_{1,t} X_1   +  W_{V,t}) \right), \ 
 $$
 and  $Q$ to be the distribution of
$$\left(X_1, \, \frac{\chi_{1,1} X_1 + \chi_{2,1} X_2 + W_{V,1}}{\sqrt{\chi_{1,1}^2 +  \chi_{2,1}^2 + \delta \E[\cL(Y)^2  ]}}
 ,\, \sqrt{\delta}(\tilde{\chi}_{1} X_1   +  W_{V, \infty}) \right). $$ 
Since $\chi_{1,t} \to \tilde{\chi}_1$ and the limits in \Cref{eq:rv_conv} hold, the sequence  $(Q_t)_{t \geq 2}$ converges weakly to $Q$. In our case, $\Psi: \reals^3 \to \reals$ is PL($2$),  and therefore $\Psi(a,b,c) \leq C'(1 + \abs{a}^2 + \abs{b}^2+\abs{c}^2)$, for all $(a,b,c) \in \reals^3$ for some constant $C'$. Choosing $h(a,b,c) = \abs{a}^2 + \abs{b}^2 + \abs{c}^2$, we have $\frac{\abs{\Psi}}{1 + h} \leq C'$. Furthermore, $\int h \, \de Q_t$ is a linear combination of $\{  \chi_{1,t}^2, \sigma_{V,t}^2\}$, with coefficients that do not depend on $t$. The integral $\int h \, \de Q$ has the same form, except that  $\chi_{1,t}, \sigma_{V,t}$ are replaced by $\tilde{\chi}_{1}, \tilde{\sigma}_{1} $, respectively. Since $\chi_{1,t} \to \tilde{\chi}_{1}$, $\sigma_{V,t} \to \tilde{\sigma}_{1}$, we have that $\lim_{t \to \infty} \int h \, \de Q_t = \int h \, \de Q$.
Therefore, by applying Lemma 4.5 in \cite{dumbgen2011}, we have that
\begin{equation}
    \lim_{t \to \infty} \int \Psi \, \de Q_t = \int \Psi \, \de Q,
    \label{eq:Psi_Qint_conv}
\end{equation}
which is equivalent to  \Cref{eq:limExpfinal}. 
From \Cref{eq:tchi_tsigma_formulas}, we recall  that $\tilde{\chi}_1 = \frac{\rho_1^{\spec}}{\sqrt{\delta}}$, $\tilde{\sigma}^2_1 = \frac{1 - (\rho_1^{\spec})^2 }{\delta}$. Using these and the formulas for $\chi_{1,1}, \chi_{2,1}$  from \Cref{eq:SElin_init_x1} in  \Cref{eq:limExpfinal}, and taking $t \to \infty$ in \Cref{eq:Psi_conv2}  yields the result in \Cref{eq:psiX1joint}. 

It remains to prove \Cref{eq:r1t_LB,eq:xi1_e3_lim}.

\paragraph{Proof of \Cref{eq:r1t_LB}} We recall from \Cref{thm:main-thm-eigval} and \Cref{rk:explicit-formula-eigval} that when $\lambda^*(\delta_1)  > \ol{\lambda}(\delta)$,
the top eigenvalue of $\Dbar \coloneqq \Abar^\top T \Abar $ converges almost surely to 
\begin{equation}
\begin{split}
        & \lim_{d \to \infty} \lambda_1( \Dbar )  = \lambda^*(\delta_1)\left( 1 +  \delta \E\left[ \frac{\cT(Y)}{\lambda^*(\delta_1) - \cT(Y)} \right] \right).  \\
        %
\end{split}
\label{eq:eigD_lims}
\end{equation}
Moreover, when $\lambda^*(\delta_1)  > \ol{\lambda}(\delta)$, \Cref{thm:main-thm-eigval}
also guarantees a strict separation between the first and second eigenvalues, i.e., 
\begin{equation}
    \lim_{d \to \infty} \lambda_1( \Dbar ) >   \lim_{d \to \infty} \lambda_2( \Dbar) = \zeta(\lambda^*(\delta_2); \delta) \ge  \lim_{d \to \infty} \lambda_3( \Dbar ) = \zeta(\ol{\lambda}(\delta); \delta).
    \label{eq:eigD_lims2}
\end{equation}
Let 
\begin{align}
    M_1 \coloneqq \frac{\Dbar}{\lambda^*(\delta_1)} -   \E \left[   \frac{ \delta \cT(Y)}{ \lambda^*(\delta_1) - \cT(Y)}  +  1 \right] I_d .
    \label{eq:M1def}
\end{align}
As $M_1$ is symmetric, it can be written as $M_1 = Q \Lambda Q^{\top}$, with $Q$  an orthogonal matrix consisting of the eigenvectors of $M_1$ and $\Lambda$ a diagonal matrix with the eigenvalues. Note that the eigenvectors of $M_1$ are the same as those of $\Dbar$ and its eigenvalues are:
\begin{equation}
\lambda_i(M_1) = \frac{\lambda_i(\Dbar)}{\lambda^*(\delta_1)} -   \E \left[   \frac{ \delta \cT(Y)}{ \lambda^*(\delta_1) - \cT(Y)}  +  1 \right], \quad i=1, \ldots,d.
\label{eq:M1_eigs}
\end{equation}
Since $r_1^t$ is orthogonal to $\eigvone = v_1(M_1)$, we have   $M_1 r_1^t = Q \Lambda' Q^{\top}r_1^t$, where $\Lambda'$ is obtained from $\Lambda$ by replacing $\lambda_1(M_1)$ with any other value. Here we replace $\lambda_1(M_1)$ by $\lambda_2(M_1)$. We therefore have
\begin{align}
    \normtwo{M_1 r_1^t}^2 & = \normtwo{ Q \Lambda' Q^{\top} r_1^t}^ 2 
     \ge \normtwo{r_1^t}^2 \, \min_{ s\in\bbS^{d-1} } \normtwo{ Q \Lambda' Q^{\top} s}^2 \notag\\
    &= \normtwo{r_1^t}^2 \, \min_{ s\in\bbS^{d-1} }\langle s, Q (\Lambda')^2 Q^{\top} s \rangle = \normtwo{r_1^t}^2  \, \lambda_d( Q (\Lambda')^2 Q^{\top}),
    \label{eq:norm_M1r1}
\end{align}
where the last equality follows from the variational characterization of the smallest eigenvalue of a symmetric matrix (Courant--Fischer theorem). Note that 
\begin{equation}
    \lambda_d( Q (\Lambda')^2 Q^{\top}) 
    =\lambda_d( (\Lambda')^2 ) = \min_{i \in \{2, \ldots, d\}} \lambda_i(M_1)^2.
    \label{eq:min_lambdai_M1}
\end{equation}
From the formula for $\lambda_i(M_1)$ in \Cref{eq:M1_eigs} and the limiting eigenvalues of $\Dbar$ in \Cref{eq:eigD_lims,eq:eigD_lims2}, we have
\begin{equation}
    \lim_{d \to \infty} \lambda_1(M_1) =0, \qquad 
    \lim_{d \to \infty}  \, \min_{i\in \{2, \ldots, d\}} \lambda_i(M_1)^2 = C >0,
    \label{eq:lim_lambdai_M1}
\end{equation}
for a universal constant $C$. Combining \Cref{eq:norm_M1r1,eq:min_lambdai_M1,eq:lim_lambdai_M1} shows that the lower bound in  \Cref{eq:r1t_LB} holds for all sufficiently large $n$.

\paragraph{Proof of \Cref{eq:xi1_e3_lim}} Since $\normtwo{\eigvone}=1$, by the triangle inequality we have
\begin{align}
    & \normtwo{e_3^t \, + \, \xi_{1,t} \left(   \delta \E \left[  \frac{\cT(Y)}{ \lambda^*(\delta_1) - \cT(Y)} \right] +  1 - \frac{\lambda_1(\Dbar)}{\lambda^*(\delta_1)}\right) \eigvone} \nonumber \\
    & \le \normtwo{e_3^t} \, + \, 
     \abs{\xi_{1,t}} \cdot \abs{   \delta \E \left[  \frac{\cT(Y)}{ \lambda^*(\delta_1) - \cT(Y)} \right] +  1 - \frac{\lambda_1(\Dbar)}{\lambda^*(\delta_1)}}. \label{eq:trineq1}
\end{align}
From \Cref{eq:xi1lim,eq:eigD_lims} we have:
\begin{align}
    \lim_{t \to \infty} \lim_{d \to \infty} \, \frac{\abs{\xi_{1,t}}}{\sqrt{d}} = \frac{1}{\sqrt{\delta}}, \qquad
    \lim_{d \to \infty} \, \abs{   \delta \E \left[  \frac{\cT(Y)}{ \lambda^*(\delta_1) - \cT(Y)} \right] +  1 - \frac{\lambda_1(\Dbar)}{\lambda^*(\delta_1)}}=0.
    \label{eq:xit_lim}
\end{align}
Therefore, the second term in \Cref{eq:trineq1} converges to $0$.
For the term $\normtwo{e_3^t}$, using the triangle inequality in the expression  in \Cref{eq:deferr3} we obtain:
\begin{equation}\label{eq:e3_trianeq}
    \begin{split}
         \normtwo{e_3^t}   \le &  \, \normtwo{e_2^t} + \abs{1-\sqrt{\delta}\beta_t} \normop{\Abar } \normop{F(F+I_n)^{-1}(F+\sqrt{\delta}\beta_tI_n)^{-1}} \normop{\Abar } \normtwo{v^t} \\
&    + \delta \abs{\E[\cF(Y)]\left(1-\frac{1}{\sqrt{\delta}\beta_t}\right)} \normtwo{v^t} + \normop{\Abar } \normop{F^2(F+\sqrt{\delta}\beta_tI_n)^{-1}} \normtwo{e_1^t}.
    \end{split}
\end{equation}
Here we have used the fact that $\normtwo{Mv} \le \normop{M} \normtwo{v}$ for any matrix $M$ and vector $v$, and  that the operator norm is sub-multiplicative.

Since $\Abar $ is i.i.d.\ Gaussian, its operator norm is bounded almost surely as $n$ grows. With $\cF(y)$ given by choice 1 in \Cref{eq:F1x1}, the diagonal matrices in \Cref{eq:e3_trianeq} become:
\begin{align}
   &  F(F+I_n)^{-1}(F+\sqrt{\delta}\beta_t I_n)^{-1}=
    \frac{1}{ \lambda^*(\delta_1)}\,  T[\lambda^*(\delta_1) I_n - T][\lambda^*(\delta_1)\sqrt{\delta} \beta_t I_n  + (1- \sqrt{\delta}\beta_t)T]^{-1}, \nonumber \\
   &  F^2(F+\sqrt{\delta}\beta_tI_n)^{-1} =
   T^2[\lambda^*(\delta_1) I_n - T]^{-1}[\lambda^*(\delta_1)\sqrt{\delta} \beta_t I_n  + (1- \sqrt{\delta}\beta_t)T]^{-1}. \label{eq:Fdiag_matrices}
\end{align}
Recalling that $\beta_t >0$ for $t > 0$, $\lim\limits_{t \to \infty} \sqrt{\delta}\beta_t =1$ and that $\cT(\cdot)$ is bounded, the operator norms of  the diagonal matrices in \Cref{eq:Fdiag_matrices} are both bounded. 

\Cref{prop:GAMP_SE} and \Cref{lem:SE_fixed_pts} together imply that 
\begin{align}
    \label{eq:vt_lim}
     \lim_{t \to \infty} \lim_{d \to \infty} \frac{\normtwo{v^t}^2}{d} = \lim_{t \to \infty} \beta_t^2 = \frac{1}{\delta}.
\end{align}
Recalling that $e_1^t= u^t-u^{t-1} $ and $e_2^t=v^{t+1}-v^t$, using \Cref{prop:GAMP_SE} and \Cref{lem:SE_fixed_pts} we can also show \cite[Lemma 5.3]{mondelli2021optimalcombination} that
\begin{align}
    \lim_{t \to \infty} \lim_{n \to \infty} \frac{\normtwo{e_1^t}^2}{n} = 0, \qquad     \lim_{t \to \infty} \lim_{d \to \infty} \frac{\normtwo{e_2^t}^2}{d} = 0.
    \label{eq:e1t_e2t_lims}
\end{align}
Using \Cref{eq:Fdiag_matrices,eq:vt_lim,eq:e1t_e2t_lims} in \Cref{eq:e3_trianeq} shows that $ \lim\limits_{t \to \infty} \lim\limits_{d \to \infty} \normtwo{e_3^t}/\sqrt{d}=0$ which, together with \Cref{eq:trineq1,eq:xit_lim}, completes the proof of \Cref{eq:xi1_e3_lim}.

\subsubsection{Proof of \Cref{eq:psiX2joint}} 

With choice 2, as defined in \Cref{eq:F2x2}, we have $\cF(y) = \cF_2(y)$, which yields
\begin{equation}
    F(F + I_n)^{-1} = \frac{1}{\lambda^*(\delta_2)}  T, \qquad \E[\cF(Y)] = \E\left[ \frac{\cT(Y)}{ \lambda^*(\delta_2) - \cT(Y)} \right].
    \label{eq:choice_2_params}
\end{equation}
We decompose the GAMP iterate $v^t$ into a component in the direction of $\eigvtwo$ plus an orthogonal component $r_2^t$:
\begin{equation}\label{eq:projvt2}
    v^t = \xi_{2,t} \eigvtwo+r_2^t,
\end{equation}
where $\xi_{2,t}=\la v^t, \eigvtwo \ra$. Substituting \Cref{eq:projvt2} in the definition of $e_3^t$ in \Cref{eq:GAMPrewrite4} and using \Cref{eq:choice_2_params}, we obtain
\begin{multline}
   \left( \frac{\Dbar}{\lambda^*(\delta_2)} -   \E \left[   \frac{ \delta \cT(Y)}{ \lambda^*(\delta_2) - \cT(Y)}  +  1 \right] I_d \right) r_2^t \\
   =  e_3^t \,  + \,   \xi_{2,t} \left( 1+ \delta \E \left[  \frac{\cT(Y)}{ \lambda^*(\delta_2) - \cT(Y)} \right] 
   - \frac{\lambda_2(\Dbar)}{\lambda^*(\delta_2)} \right) \eigvtwo .
   \label{eq:r2t_exp}
\end{multline}
We  show that $ \lim\limits_{t \to \infty}\lim\limits_{d \to \infty} \| r_2^t \|^2/d = 0$, which from \Cref{eq:projvt2} implies that  the GAMP iterate is aligned with $\eigvtwo$ in the limit. To show this, we first claim that for all sufficiently large $n$:%
\begin{align}
     \normtwo{ \left( \frac{\Dbar}{\lambda^*(\delta_2)} -   \E \left[   \frac{ \delta \cT(Y)}{ \lambda^*(\delta_2) - \cT(Y)}  +  1 \right] I_d \right) r_2^t}  \ge C  \normtwo{r_2^t},
\label{eq:r2t_LB}
\end{align}
for some  constant $C >0$. We then show that under choice 2:
\begin{align}
    \lim_{t \to \infty} \lim_{d \to \infty} \, \frac{1}{\sqrt{d}} \normtwo{e_3^t \, + \, \xi_{2,t} \left(   \delta \E \left[  \frac{\cT(Y)}{ \lambda^*(\delta_2) - \cT(Y)} \right] +  1 - \frac{\lambda_2(\Dbar)}{\lambda^*(\delta_2)}\right) \eigvtwo}  =0.
    \label{eq:xi2_e3_lim}
\end{align}
Given the claims in  \Cref{eq:r2t_LB,eq:xi2_e3_lim}, the result in \Cref{eq:psiX2joint} is obtained using the same steps as \Cref{eq:r1t_lim} to \eqref{eq:Psi_Qint_conv}, by replacing $\xone, x_1^{\spec}, r_{1}^t, \alpha, \xi_{1,t}, \tilde{\chi}_1, \tilde{\sigma}_1$ with $\xtwo, x_2^{\spec}, r_{2}^t, (1-\alpha), \xi_{2,t}, \tilde{\chi}_2, \tilde{\sigma}_2$, respectively.

The proof of \Cref{eq:r2t_LB} is along the same lines as  that of \Cref{eq:r1t_LB}. Other than replacing  notation as above, the only change is in the argument from \Cref{eq:M1def} to \eqref{eq:lim_lambdai_M1}. Here we define the matrix $M_2 \coloneqq \frac{\Dbar}{\lambda^*(\delta_2)} -   \E \left[   \frac{ \delta \cT(Y)}{ \lambda^*(\delta_2) - \cT(Y)}  +  1 \right]I_d$, which can be written as $M_2 = Q \bar{\Lambda} Q^{\top}$, where  $\bar{\Lambda}$ is a diagonal matrix containing the eigenvalues of $M_2$, and $Q$ is an orthogonal matrix with the  eigenvectors. The eigenvectors of $M_2$ are the same as those of $\Dbar$ and its eigenvalues are:
\begin{equation}
\lambda_i(M_2) = \frac{\lambda_i(\Dbar)}{\lambda^*(\delta_2)} -   \E \left[   \frac{ \delta \cT(Y)}{ \lambda^*(\delta_2) - \cT(Y)}  +  1 \right], \quad i=1, \ldots,d.
\label{eq:M2_eigs}
\end{equation}
\Cref{thm:main-thm-eigval} and \Cref{rk:explicit-formula-eigval} guarantee that when $\lambda^*(\delta_1) > \lambda^*(\delta_2)  > \ol{\lambda}(\delta)$, there is strict separation between the top three
three eigenvalues of $\Dbar \coloneqq \Abar^\top T \Abar $. The limits of these eigenvalues are: 
\begin{equation}
\begin{split}
         \lim_{d \to \infty} \lambda_1( \Dbar )  &= \lambda^*(\delta_1)\left( 1 +  \delta \E\left[ \frac{\cT(Y)}{\lambda^*(\delta_1) - \cT(Y)} \right] \right)  \\
    &   >   \lim_{d \to \infty} \lambda_2( \Dbar ) = \lambda^*(\delta_2)\left( 1 +  \delta \E\left[ \frac{\cT(Y)}{\lambda^*(\delta_2) - \cT(Y)} \right] \right) \\
        &   > \lim_{d \to \infty} \lambda_3( \Dbar ) = \bar{\lambda}(\delta)\left( 1 +  \delta \E\left[ \frac{\cT(Y)}{\bar{\lambda}(\delta) - \cT(Y)} \right] \right).
\end{split}
\label{eq:eigD_lims3}
\end{equation}
Since $r_2^t$ is orthogonal to $\eigvtwo = v_2(M_2)$, we have   $M_2 r_2^t = Q \bar{\Lambda}' Q^{\top}r_1^t$, where $\bar{\Lambda}'$ is obtained from $\bar{\Lambda}$ by replacing the \emph{second}  eigenvalue $\lambda_2(M_2)$ with any other value. Here we replace $\lambda_2(M_2)$ by $\lambda_1(M_2)$. Then, using the eigenvalue limits in \Cref{eq:eigD_lims3} together with arguments analogous to \Cref{eq:norm_M1r1,eq:min_lambdai_M1,eq:lim_lambdai_M1}, we obtain:
\begin{equation}
    \lim_{d \to \infty} \lambda_2(M_2) =0, \qquad 
    \lim_{d \to \infty}  \, \min_{i \ne 2} \lambda_i(M_2)^2 = C >0,
    \label{eq:lim_lambdai_M2}
\end{equation}
for a universal constant $C$.

The proof of \Cref{eq:xi2_e3_lim} is essentially identical to that of \Cref{eq:xi1_e3_lim}, and is omitted. This completes the proof of \Cref{thm:main-thm-joint-dist}.


\begin{remark}[Adapting the argument to $\alpha = 1/2$]
\label{rk:alpha-half-amp}
To obtain the result mentioned in \Cref{rk:alpha-half-master}, we
analyze a pair of GAMP algorithms with the same design of denoisers and initializers as in choices 1 and 2. In particular, one can show that $v^t$ converges to a pair of linearly independent vectors in the span of $ v_1(\Dbar) $ and $ v_2(\Dbar) $ under choices 1 and 2. 
To prove the claim, under choice 1, we decompose $ v^t $ into the projection onto $ \spn\curbrkt{v_1(\Dbar), v_2(\Dbar)} $ and the orthogonal component $ r_1^t $. 
Via a similar analysis, \Cref{eq:r1t_lim} can be shown to hold, provided that $ \lambda^*(\delta/2) > \ol\lambda(\delta) $, since this condition ensures the existence of a spectral gap (cf.\ \Cref{rk:alpha-half-rmt}). 
Hence, $ v^t $ converges to a vector in $ \spn\curbrkt{v_1(\Dbar), v_2(\Dbar)} $. 
Furthermore, \Cref{eq:Psi_conv1} continues to hold by state evolution (\Cref{prop:GAMP_SE}). 
As a result, $ v^t $ converges to a vector $ \tv_1 $ whose limiting empirical distribution has the law of $ \tilde{\chi}_1 X_1 + W_{V,\infty} $, with $ W_{V,\infty} $ independent of $ (X_1,X_2) $. 
Similarly, under choice 2, $ v^t $ converges to another vector $ \tv_2 $ in $ \spn\curbrkt{v_1(\Dbar), v_2(\Dbar)} $ whose limiting empirical distribution has the law of $ \tilde{\chi}_2 X_2 + W_{V,\infty}' $, with $ W_{V,\infty}' $ independent of $ (X_1,X_2) $. By recognizing that the inner products of $ \tv_1 , \tv_2 $ with $x_1^*$ differ, 
one readily obtains that $ \tv_1,\tv_2 $ are linearly independent.
Therefore, in the high-dimensional limit, the GAMP iterates recover $ \spn\curbrkt{v_1(\Dbar),v_2(\Dbar)} $. 
At this point, we can find the vector in $\spn\curbrkt{v_1(\Dbar),v_2(\Dbar)}$ with the desired limiting joint law as in \Cref{eq:psiX1joint,eq:psiX2joint} by matching its correlation with the linear estimator $ x^{\lin} $ via 
\Cref{eqn:xspec-alpha-half}. 
This last step of grid search can be effectively carried out when the overlap attained by $x^{\lin}$ is non-zero, i.e., $ \expt{G\cL(Y)}\ne0 $. 
\end{remark}

\section{Bayes-optimal combination (proof of \Cref{cor:opt-combo})}\label{app:bayes-opt-comb}

In the proof, we only consider combined estimators for the estimation of $x_1^*$. 
The arguments for the estimation of $x_2^*$ are similar, and therefore omitted. 
For any $ C_1\in\cC_1 $ (the latter set is defined in \Cref{eqn:set-comb-fun-1}), consider the combined estimator $ x_1^{\comb} \coloneqq C_1(x^{\lin}, x_1^{\spec}) $. 
By \Cref{thm:main-thm-joint-dist}, we have
\begin{align}
    \lim_{d\to\infty} \frac{\abs{\inprod{x_1^{\comb}}{x_1^*}}}{\normtwo{x_1^{\comb}}\normtwo{x_1^*}} &= \frac{\abs{\expt{X_1 C_1(X^{\lin}, X_1^{\spec})}}}{\sqrt{\expt{C_1(X^{\lin}, X_1^{\spec})^2}}} . \notag 
\end{align}
almost surely. 
Here we use the fact that $ \expt{X_1^2} = 1 $ which follows from $ \xone\in\sqrt{d}\,\bbS^{d-1} $. 
Now, the optimality of the conditional expectation function $ C_1^* $ follows from the Cauchy--Schwarz inequality:
\begin{align}
    \frac{\abs{\expt{X_1 C_1(X^{\lin}, X_1^{\spec})}}}{\sqrt{\expt{C_1(X^{\lin}, X_1^{\spec})^2}}} 
    &= \frac{\abs{\expt{\expt{X_1 \condon X^{\lin}, X_1^{\spec}} C_1(X^{\lin}, X_1^{\spec})}}}{\sqrt{\expt{C_1(X^{\lin}, X_1^{\spec})^2}}} \notag \\
    &\le \sqrt{\expt{\expt{X_1 \condon X^{\lin}, X_1^{\spec}}^2}}, \label{eqn:bayes-opt-bound} 
\end{align}
with equality in  \Cref{eqn:bayes-opt-bound} if $C_1=C_1^*$.


We then compute $ \expt{X_1 \condon X^{\lin}, X_1^{\spec}} $ from the joint distribution of $ (X_1, X^{\lin}, X^{\spec}) $ given by \Cref{eqn:def-asymp-corr,eqn:def-rv-xlin-xspec}. 
Under \Cref{itm:assump-signal-distr}, we have $ (X_1,X_2)\sim\cN(0,1)^{\ot2} $. 
Using this it can be verified that $(X_1, X^{\lin}, X^{\spec})$ are jointly Gaussian with zero mean and the following covariance matrix:
\begin{align}
    \begin{bmatrix}
    1 & \rho_1^{\lin} & \rho_1^{\spec} \\
    \rho_1^{\lin} & 1 & \rho_1^{\lin}\rho_1^{\spec} + \expt{W^{\lin}W_1^{\spec}} \\
    \rho_1^{\spec} & \rho_1^{\lin}\rho_1^{\spec} + \expt{W^{\lin}W_1^{\spec}} & 1
    \end{bmatrix} . \notag 
\end{align}
Let  $\nu_1 = \rho_1^{\lin} \rho_1^{\spec} + \expt{W^{\lin} W_1^{\spec}}$. 
Using the covariance structure above, we obtain that $ X_1 $ conditioned on $ (X^{\lin}, X_1^{\spec}) $ is a Gaussian random variable with mean 
\begin{align}
    \wt{\mu} &\coloneqq \begin{bmatrix}
    \rho_1^{\lin} & \rho_1^{\spec}
    \end{bmatrix} \begin{bmatrix}
    1 & \nu_1 \\ \nu_1 & 1
    \end{bmatrix}^{-1} \begin{bmatrix}
    X^{\lin} \\ X_1^{\spec}
    \end{bmatrix} \notag 
\end{align}
and variance
\begin{align}
    \wt{\sigma}^2 &\coloneqq 1 - \begin{bmatrix}
    \rho_1^{\lin} & \rho_1^{\spec}
    \end{bmatrix} \begin{bmatrix}
    1 & \nu_1 \\ \nu_1 & 1
    \end{bmatrix}^{-1} \begin{bmatrix}
    \rho_1^{\lin} \\ \rho_1^{\spec}
    \end{bmatrix} . \notag 
\end{align}
Therefore, the Bayes-optimal combined estimator is given by
\begin{align}
    \expt{X_1\condon X^{\lin}, X_1^{\spec}} &= \wt{\mu}
    = \frac{1}{1 - \nu_1^2} \sqrbrkt{\paren{\rho_1^{\lin} - \rho_1^{\spec}\nu_1} X^{\lin} + \paren{\rho_1^{\spec} - \rho_1^{\lin}\nu_1} X_1^{\spec}} , \notag 
\end{align}
which agrees with the expression in \Cref{eqn:opt-combinator-1}. 
Finally, the explicit formulas of the overlaps given by the Bayes-optimal combined estimators can be obtained from \Cref{eqn:bayes-opt-bound,eqn:opt-combinator-1} via elementary algebraic manipulations. 

\section{Additional proofs for linear estimator (proof of \Cref{lem:linear-optimal-overlap})}
\label{sec:linear-estimator-pf}



With the characterization of the limiting overlaps   of the linear estimator 
in \Cref{lem:linear-overlap}, we can  maximize them over the choice of the preprocessing function $\cL\colon\bbR\to\bbR $.   For $ i\ge0 $, let
\begin{align}
    m_i(y) &\coloneqq \expt{G^i p(y|G)} , \label{eqn:def-mi}
\end{align}
where $ G\sim\cN(0,1) $ and $ p(y|g) $ is defined in \Cref{eq:condlaw}. 
Using $m_0,m_1$, the squared limiting overlap between $\xlin$ and $x_1^*$ in \Cref{eqn:overlap-lin-1-main} can be expressed in the following way:
\begin{multline}
    \frac{\alpha^2 \expt{G\cL(Y)}^2}{(\alpha^2+(1-\alpha)^2) \expt{G\cL(Y)}^2 + \expt{\cL(Y)^2}/\delta} 
    = \paren{\frac{\alpha^2+(1-\alpha)^2}{\alpha^2} + \frac{1}{\alpha^2\delta} \cdot \frac{\expt{\cL(Y)^2}}{\expt{G\cL(Y)}^2}}^{-1} \\ 
    = \paren{\frac{\alpha^2+(1-\alpha)^2}{\alpha^2} + \frac{1}{\alpha^2\delta} \cdot \frac{\int_{\supp(Y)} m_0(y) \cL(y)^2 \diff y}{\paren{\int_{\supp(Y)} m_1(y) \cL(y) \diff y}^2}}^{-1} , \label{eqn:overlap-using-m0m1} 
\end{multline}
provided $ \int_{\supp(Y)} m_1(y) \cL(y) \diff y \ne0 $ and $ \expt{|G\cL(Y)|}<\infty $. 
The optimization of overlap can be formalized as the following maximization problem:
\begin{align}
    \paren{\ollinone}^2 &\coloneqq \sup_{\cL\colon\bbR\to\bbR} \paren{\frac{\alpha^2+(1-\alpha)^2}{\alpha^2} + \frac{1}{\alpha^2\delta} \cdot \frac{\int_{\supp(Y)} m_0(y) \cL(y)^2 \diff y}{\paren{\int_{\supp(Y)} m_1(y) \cL(y) \diff y}^2}}^{-1} \notag \\
    &\phantom{\coloneqq} \suchthat \int_{\supp(Y)} m_1(y) \cL(y) \diff y \ne0 \label{eqn:ol-lin-constraint-1} \\
    &\phantom{\coloneqq \suchthat} \expt{|G\cL(Y)|}<\infty . \label{eqn:ol-lin-constraint-2} 
\end{align}
Therefore, maximizing $\ollinone$ is equivalent to solving the following minimization problem:
\begin{align}
    \optlin &\coloneqq \inf_{\cL\colon\bbR\to\bbR} \frac{\int_{\supp(Y)} m_0(y) \cL(y)^2 \diff y}{\paren{\int_{\supp(Y)} m_1(y) \cL(y) \diff y}^2} \quad  \suchthat \quad \text{\Cref{eqn:ol-lin-constraint-1,eqn:ol-lin-constraint-2}} . \notag 
\end{align}
This optimization problem has been studied in \cite[Appendix C.1]{mondelli2021optimalcombination}. 
In particular, under the condition 
\begin{align}
    \int_{\supp(Y)} \frac{m_1(y)^2}{m_0(y)} \diff y \in(0,\infty) \label{eqn:lin-regularity-condition}
\end{align}
(which is equivalent to \Cref{eqn:cond-lin-eff}), 
we have
$  \optlin = \paren{\int_{\supp(Y)} \frac{m_1(y)^2}{m_0(y)} \diff y}^{-1}$,
attained by $ \cL^*\colon\bbR\to\bbR $ defined as (cf.\ \cite[Eqn.\ (C.4)]{mondelli2021optimalcombination}) 
\begin{align}
    \cL^*(y) = \frac{m_1(y)}{m_0(y)} , \label{eqn:opt-lin-preprocessor}
\end{align}
which satisfies $ \int_{\supp(Y)} m_1(y) \cL^*(y) \diff y > 0 $ and $ \expt{|G\cL^*(Y)|}<\infty $. 

Therefore, the value of the original problem $ \ollinone $ we are interested in is given by 
\begin{align}
    \paren{\ollinone}^2 &= \paren{\frac{\alpha^2+(1-\alpha)^2}{\alpha^2} + \frac{1}{\alpha^2\delta} \cdot \frac{1}{\int_{\supp(Y)} \frac{m_1(y)^2}{m_0(y)} \diff y}}^{-1} . \label{eqn:opt-overlp-lin-1}
\end{align}
Analogously, the optimal (over the choice of $\cL\colon\bbR\to\bbR$) limiting overlap between $\xlin$ and $x_2^*$ equals
\begin{align}
    \paren{\ollintwo}^2 &= \paren{\frac{\alpha^2+(1-\alpha)^2}{(1-\alpha)^2} + \frac{1}{(1-\alpha)^2\delta} \cdot \frac{1}{\int_{\supp(Y)} \frac{m_1(y)^2}{m_0(y)} \diff y}}^{-1} , \label{eqn:opt-overlp-lin-2}
\end{align}
which is also achieved by $\cL^*\colon\bbR\to\bbR$ defined in \Cref{eqn:opt-lin-preprocessor}.

\section{Additional proofs for spectral estimator (\Cref{thm:opt-spec})}\label{sec:spec-estimator-pf}

\subsection{Optimization of spectral threshold}
\label{sec:opt-threshold}

Let us consider weak recovery of $x_1^*$. 
(The analysis for the recovery of $x_2^*$ is completely analogous and is therefore omitted.) For a given preprocessing function $\cT\colon\bbR\to\bbR$, we know from \Cref{thm:overlap-spectral} and \Cref{rk:cnvan} that  weak recovery of $ x_1^* $ is possible (i.e., $\rho_1^{\spec} >0$) when $ \lambda^*(\delta_1) > \ol\lambda(\delta) $. 
This condition is equivalent to $ \phi(\ol\lambda(\delta)) > \zeta(\ol\lambda(\delta); \delta_1) = \psi(\ol\lambda(\delta); \delta_1) $, or more explicitly, 
\begin{align}
    \expt{\frac{Z}{\ol\lambda(\delta) - Z}(G^2 - 1)} &> \frac{1}{\delta_1} = \frac{1}{\alpha\delta}. \label{eqn:optthr-cond1} 
\end{align}
Here we recall that $Z=\cT(Y)$, and $ \ol\lambda(\delta) $ satisfies $\psi'(\ol\lambda(\delta);\delta) = 0$, or equivalently (see \Cref{rk:explicit-formulas}), it is the solution to
\begin{align}
    \expt{\paren{\frac{Z}{\ol\lambda(\delta) - Z}}^2} &= \frac{1}{\delta} . \label{eqn:optthr-cond2}
\end{align}
\Cref{eqn:optthr-cond1}  assumes that 
$ \ol\lambda(\delta)>0 $, which is satisfied as $ \ol\lambda(\delta)>\sup\supp(\cT(Y)) $ by definition (cf.\ \Cref{eqn:def-lam-bar}) and the RHS is strictly positive by \Cref{itm:assump-preproc-spec}.  
Due to scaling invariance, we claim that $ \ol\lambda(\delta) $ can be assumed to be $1$. 
Indeed, both \Cref{eqn:optthr-cond1,eqn:optthr-cond2} depend on $ \cT\colon\bbR\to\bbR $ only through the ratio $ \frac{\cT(Y)}{\ol\lambda(\delta) - \cT(Y)} $,  therefore 
any given $\cT$ and the corresponding $ \ol\lambda(\delta) $ derived via \Cref{eqn:optthr-cond2} can be replaced with $ \cT/\ol\lambda(\delta) $ and $1$, respectively. 
\Cref{eqn:optthr-cond1,eqn:optthr-cond2} then become
\begin{align}
    \expt{\frac{Z}{1 - Z}(G^2 - 1)} > \frac{1}{\alpha\delta} ,  \qquad   \expt{\paren{\frac{Z}{1 - Z}}^2} = \frac{1}{\delta} .\label{eqn:optthr-cond1-2} 
\end{align}
%
For convenience, let $ \Gamma(y) \coloneqq \frac{\cT(y)}{1 - \cT(y)}$.
The expectations in \Cref{eqn:optthr-cond1-2} can then be written as
\begin{align}
    \expt{\paren{\frac{Z}{1 - Z}}^2}
    = \expt{\paren{\frac{\cT(Y)}{1 - \cT(Y)}}^2} 
    = \expt{\Gamma(Y)^2} 
    &= \int_{\supp(Y)} \expt{p(y|G)}\cdot\Gamma(y)^2 \diff y , \label{eqn:expand1}
\end{align}
where $G \sim \cN(0,1)$. Similarly
\begin{align}
    \expt{\frac{Z}{1 - Z}(G^2 - 1)}
    &= \int_{\supp(Y)} \expt{p(y|G)(G^2 - 1)}\cdot\Gamma(y) \diff y . \label{eqn:expand2}
\end{align}
So the conditions in \Cref{eqn:optthr-cond1-2} can be further written as
\begin{align}
    \int_{\supp(Y)} \expt{p(y|G)}\cdot\Gamma(y)^2 \diff y &= \frac{1}{\delta} , \label{eqn:optthr-cond1-3}
\end{align}
and
\begin{align}
    \int_{\supp(Y)} \expt{p(y|G)(G^2 - 1)}\cdot\Gamma(y) \diff y &> \frac{1}{\alpha\delta} . \label{eqn:optthr-cond2-3}
\end{align}
In the above form, one can apply Cauchy--Schwarz inequality to \Cref{eqn:optthr-cond2-3} given the equality condition in \Cref{eqn:optthr-cond1-3}. Hence, we can upper bound the LHS of \Cref{eqn:optthr-cond2-3} as
\begin{align}
    & \int_{\supp(Y)} \expt{p(y|G)(G^2 - 1)}\cdot \Gamma(y) \diff y \notag \\
    &= \int_{\supp(Y)} \sqrt{\expt{p(y|G)}}\cdot\Gamma(y) \cdot \frac{\expt{p(y|G)(G^2 - 1)}}{\sqrt{\expt{p(y|G)}}} \diff y \notag \\
    &   \le \sqrt{\int_{\supp(Y)} \expt{p(y|G)}\cdot\Gamma(y)^2 \diff y} \cdot\sqrt{\int_{\supp(Y)} \frac{\expt{p(y|G)(G^2 - 1)}^2}{\expt{p(y|G)}} \diff y}  \notag \\
    & = \frac{1}{\sqrt{\delta}} \cdot \frac{1}{\alpha\sqrt{\delta^*_1}}, \label{eqn:optthr-cond3-3}
\end{align}
where the last equality is obtained  using \Cref{eqn:optthr-cond1-3} and by defining
\begin{align}
    \delta^*_1 &\coloneqq \frac{1}{\alpha^2\int_{\supp(Y)} \frac{\expt{p(y|G)(G^2 - 1)}^2}{\expt{p(y|G)}} \diff y} \label{eqn:def-delta-1-star} 
\end{align}
Combining \Cref{eqn:optthr-cond2-3,eqn:optthr-cond3-3}, we obtain the condition $ \delta > \delta^*_1 $.

So far we have shown that if $ \lambda^*(\delta_1) > \ol\lambda(\delta) $ under a  preprocessing function $ \cT\colon\bbR\to\bbR $, then it must be the case that $ \delta > \delta^*_1 $. 
In what follows, we show that whenever $ \delta>\delta^*_1 $, one can find a preprocessing function $ \wt\cT_1^*\colon\bbR\to\bbR $ that achieves equality in both \Cref{eqn:optthr-cond1-3,eqn:optthr-cond3-3}, guaranteeing that $ \lambda^*(\delta_1) > \ol\lambda(\delta) $. 

Consider any $ \delta > \delta^*_1 $. 
Our goal is to find a function $ \wt\cT^*_1\colon\bbR\to\bbR $ satisfying
\begin{align}
    \expt{\paren{\frac{\wt\cT_1^*(Y)}{1 - \wt\cT_1^*(Y)}}^2} &= \frac{1}{\delta} , \label{eqn:find-optthr-1}
\end{align}
and
\begin{align}
    \expt{\frac{\wt\cT_1^*(Y)}{1 - \wt\cT^*_1(Y)}(G^2 - 1)} = \frac{1}{\alpha\sqrt{\delta}\sqrt{\delta^*_1}}  > \frac{1}{\alpha\delta}. \label{eqn:find-optthr-2-strong}
\end{align}
Note that according to \Cref{eqn:find-optthr-1}, the corresponding $ \ol\lambda(\delta) $ associated with $ \wt\cT^*_1 $ (to be constructed) is equal to $1$. 
%
Constructing such a $ \wt\cT_1^* $ is equivalent to constructing a function $\cT_1^*$, which we  define via
\begin{align}
    \frac{\wt\cT_1^*(y)}{1 - \wt\cT_1^*(y)} &= \sqrt{\frac{\delta^*_1}{\delta}}\cdot\frac{\cT_1^*(y)}{1 - \cT_1^*(y)} \label{eqn:tilde-t1star} 
\end{align}
for every $ y\in\bbR $. 
Now \Cref{eqn:find-optthr-1,eqn:find-optthr-2-strong} are equivalent to 
\begin{align}
    \expt{\paren{\frac{\cT_1^*(Y)}{1 - \cT_1^*(Y)}}^2} &= \frac{1}{\delta^*_1} , \qquad  \expt{\frac{\cT_1^*(Y)}{1 - \cT_1^*(Y)}(G^2 - 1)} = \frac{1}{\alpha\delta^*_1}  .\label{eqn:find-optthr-1-2}
\end{align}
Before proceeding, let us define the following functions for convenience: 
\begin{align}
    \Gamma_1^*(y) &\coloneqq \frac{\cT_1^*(y)}{1 - \cT_1^*(y)} ,  \qquad    
    m_0(y) \coloneqq \expt{p(y|G)}, \qquad m_2(y) 
    \coloneqq \expt{p(y|G)G^2}. \label{eqn:gamma1star} 
\end{align}
With the above notation, $ \delta^*_1 $ in \Cref{eqn:def-delta-1-star} can be written as
\begin{align}
    \delta^*_1 &= \frac{1}{\alpha^2\int_{\supp(Y)} \frac{(m_2(y) - m_0(y))^2}{m_0(y)} \diff y} , \label{eqn:def-delta1-star-pf}
\end{align}
and the LHSs of \Cref{eqn:find-optthr-1-2} can be written as 
\begin{align}
    \expt{\paren{\frac{\cT_1^*(Y)}{1 - \cT_1^*(Y)}}^2} &= \int_{\supp(Y)} \Gamma_1^*(y)^2 m_0(y) \diff y \notag
\end{align}
and
\begin{align}
    \expt{\frac{\cT_1^*(Y)}{1 - \cT_1^*(Y)}(G^2 - 1)} &= \int_{\supp(Y)} \Gamma_1^*(y) (m_2(y) - m_0(y)) \diff y \notag 
\end{align}
respectively, using \Cref{eqn:expand1,eqn:expand2}. 
Then, \Cref{eqn:find-optthr-1-2} becomes
\begin{align}
    \int_{\supp(Y)} \Gamma_1^*(y)^2 m_0(y) \diff y &= \alpha^2\int_{\supp(Y)} \frac{(m_2(y) - m_0(y))^2}{m_0(y)} \diff y \label{eqn:find-optthr-1-3}
\end{align}
and 
\begin{align}
    \int_{\supp(Y)} \Gamma_1^*(y) (m_2(y) - m_0(y)) \diff y &= \alpha\int_{\supp(Y)} \frac{(m_2(y) - m_0(y))^2}{m_0(y)} \diff y .
    \label{eqn:find-optthr-2-3}
\end{align}
By inspecting \Cref{eqn:find-optthr-1-3,eqn:find-optthr-2-3}, we conclude that the following choice of $ \Gamma_1^*\colon\bbR\to\bbR $ meets both conditions:
\begin{align}
    \Gamma_1^*(y) &\coloneqq \alpha\cdot\frac{m_2(y) - m_0(y)}{m_0(y)} . \notag 
\end{align}
By \Cref{eqn:gamma1star}, this gives a choice of $ \cT_1^* $: 
\begin{align}
    \cT_1^*(y) &= \frac{\Gamma_1^*(y)}{1 + \Gamma_1^*(y)}
    = \frac{\alpha\cdot\frac{m_2(y) - m_0(y)}{m_0(y)}}{1 + \alpha\cdot\frac{m_2(y) - m_0(y)}{m_0(y)}}
    = 1 - \frac{1}{\alpha\cdot\frac{m_2(y)}{m_0(y)} + (1-\alpha)} . \label{eqn:opt-spec-T-thr} 
\end{align}
Using the relation in \Cref{eqn:tilde-t1star}, we can then determine $ \wt\cT_1^*(y) $:
\begin{align}
\wt\cT_1^*(y) &= 1 - \frac{1}{1 + \sqrt{\frac{\delta^*_1}{\delta}}\cdot\frac{\cT_1^*(y)}{1 - \cT_1^*(y)}} 
= 1 - \frac{1}{\sqrt{\frac{\delta^*_1}{\delta}} \paren{\alpha\cdot\frac{m_2(y)}{m_0(y)} + (1-\alpha)} + \paren{1 - \sqrt{\frac{\delta^*_1}{\delta}}\cdot}} . \label{eqn:opt-spec-tilde-T}
\end{align}

We have found a candidate function $ \wt\cT_1^*\colon\bbR\to\bbR $ that satisfies both \Cref{eqn:find-optthr-1,eqn:find-optthr-2-strong}. 
It remains to verify that this function meets \Cref{itm:assump-preproc-spec}. 
In \Cref{sec:opt-spec-overlap} (see \Cref{eqn:inf-lb,eqn:sup-lb,eqn:sup-ub}), we will show  that $\cT_1^*$ satisfies
\begin{align}
    0 < \sup_{y\in\supp(Y)} \cT_1^*(y) < 1 , \quad 
    \inf_{y\in\supp(Y)} \cT_1^*(y) > -\infty . \label{eqn:T-prop-to-prove}
\end{align}
Since $ \cT_1^*(y) < 1 $ for every $y\in\supp(Y)$, writing 
\begin{align}
    \wt{\cT}_1^*(y) &= 1 - \frac{1}{\frac{\sqrt{\delta_1^*/\delta}}{1 - \cT_1^*(y)} + 1 - \sqrt{\delta_1^*/\delta}} , \notag 
\end{align} 
we see that $ \wt{\cT}_1^*(y) $ increases as $ \cT_1^*(y) $ increases. 
Therefore, \Cref{eqn:T-prop-to-prove} implies
\begin{align}
    0 < \sup_{y\in\supp(Y)} \wt{\cT}_1^*(y) < 1 , \quad 
    \inf_{y\in\supp(Y)} \wt{\cT}_1^*(y) > 1 - \frac{1}{1 - \sqrt{\delta_1^*/\delta}} ,
\end{align}
which certifies that $\wt{\cT}_1^*$ satisfies \Cref{itm:assump-preproc-spec}. 

\subsection{Optimization of spectral overlap}
\label{sec:opt-spec-overlap}

We will now identify the optimal preprocessing function that achieves the largest asymptotic overlap with $ x_1^* $ whenever $ \lambda^*(\delta_1) > \ol\lambda(\delta) $.
We again only focus on $ x_1^* $, as the argument for $ x_2^* $ is analogous. 
Recall from \Cref{thm:overlap-spectral} 
that the squared overlap induced by a generic preprocessing function converges almost surely to
\begin{align}
    \lim_{d\to\infty} \frac{\inprod{v_1(D)}{x_1^*}^2}{\normtwo{v_1(D)^2}^2 \normtwo{x_1^*}^2}
    &
    = \frac{1}{\alpha}\cdot\frac{\frac{1}{\delta} - \expt{\paren{\frac{Z}{\lambda^*(\delta_1) - Z}}^2}}{\frac{1}{\delta_1} + \expt{\paren{\frac{Z}{\lambda^*(\delta_1) - Z}}^2(G^2 - 1)}}= \frac{1}{\alpha}\cdot\frac{\psi'(\lambda^*(\delta_1);\delta)}{\psi'(\lambda^*(\delta_1);\delta_1) - \phi'(\lambda^*(\delta_1))}, \notag
\end{align}
where the second equality readily follows after some manipulations.  Let $ \cF $ denote the set of feasible preprocessing functions:
\begin{align}
    \cF &\coloneqq \curbrkt{\cT\colon\bbR\to\bbR : 
    \inf_{y\in\supp(Y)} \cT(y) > -\infty , \;
    0 < \sup_{y\in\supp(Y)} \cT(y) < \infty , \;
    \prob{\cT(Y) = 0} < 1
    } , \label{eqn:feasible-t} 
\end{align}
where $Y = q(G,\eps)$ and $ (G,\eps)\sim\cN(0,1)\otimes P_\eps $. 
The variational problem we would like to solve is 
\begin{align}
    \mathsf{OL}_1^2 &= \sup_{\cT_1\in\cF} \, \frac{1}{\alpha}\cdot\frac{\psi'(\lambda^*(\delta_1);\delta)}{\psi'(\lambda^*(\delta_1);\delta_1) - \phi'(\lambda^*(\delta_1))} \notag \\
    & \quad\suchthat\quad \zeta(\lambda^*(\delta_1); \delta_1) = \phi(\lambda^*(\delta_1)) \notag \\
    &= \frac{1}{\alpha} \sup_{\cT_1\in\cF} \frac{\psi'(\lambda^*(\delta_1);\delta)}{\psi'(\lambda^*(\delta_1);\delta_1) - \phi'(\lambda^*(\delta_1))} \notag \\
    & \phantom{=}\suchthat\quad \psi(\lambda^*(\delta_1);\delta_1) = \phi(\lambda^*(\delta_1)), \quad  \psi'(\lambda^*(\delta_1);\delta) > 0 \label{eqn:cond-opt-overlap-for-rk} 
\end{align}
where $ \mathsf{OL}_1^2 $ is the squared overlap for the first signal. 
The first constraint in \Cref{eqn:cond-opt-overlap-for-rk} is by the definition of $ \lambda^*(\delta_1) $ and the second one is an equivalent characterization of the spectral threshold. 
We claim that $ \lambda^*(\delta_1) $ can be assumed without loss of generality to be $1$. 
To see this, according to the explicit formulas for $\phi'(\cdot)$, $ \psi'(\cdot; \delta_1) $, and $ \lambda^*(\delta_1) $ (see \Cref{eqn:phi-deriv,eqn:derivative-psi,eqn:def-lami-star-supcrit-explicit}), we note that both the objective and constraints of the optimization problem $ \mathsf{OL}_1^2 $ depend on $ \cT_1\colon\bbR\to\bbR $ only through the ratio $ \frac{\cT_1(Y)}{\lambda^*(\delta_1) - \cT_1(Y)} $. 
Therefore, any $ \cT_1 $ and the corresponding $ \lambda^*(\delta_1) $ computed via \Cref{eqn:derivative-psi} can be replaced without affecting anything with $ \cT_1/\lambda^*(\delta_1) $ and $1$, respectively. 
Recall that $\Gamma_1\colon\bbR\to\bbR$ is defined as
\begin{align}
    \Gamma_1(y) &\coloneqq \frac{\cT_1(y)}{1 - \cT_1(y)} = \frac{1}{1 - \cT_1(y)} - 1 . \label{eqn:gamma-t} 
\end{align}
Since $ \lambda^*(\delta_1) = 1 $ and $ \lambda^*(\delta_1) $ is the solution to \Cref{eqn:def-lami-star-supcrit-explicit} in $ (\sup\supp(Z),\infty) $, 
we need an additional constraint on $ \cT_1 $: 
\begin{align}
    1 > \sup\supp(\cT_1(Y)) = \sup_{y\in\supp(Y)} \cT_1(y) , \label{eqn:sup-t1-ub} 
\end{align}
which, in light of \Cref{eqn:gamma-t}, translates to the following constraint on $ \Gamma_1 $:
\begin{align}
    \sup_{y\in\supp(Y)} \Gamma_1(y) &> -1 . \label{eqn:gamma-cstr-3}
\end{align}
Let $ \cG $ be the feasible set of $ \Gamma_1\colon\bbR\to\bbR $. 
To give a precise definition of $\cG$, we now translate the constraints on $ \cT_1 $ to constraints on $\Gamma_1$. 
According to \Cref{eqn:sup-t1-ub} and the second equality in \Cref{eqn:gamma-t}, $\Gamma_1(y)$ increases as $\cT_1(y)$ increases in $(-\infty,1)$. 
Therefore, the constraints $ \sup\limits_{y\in\supp(Y)} \cT(y) \in (0,\infty) $ and  $ \inf\limits_{y\in\supp(Y)} \cT_1(y) > -\infty $ translate to 
\begin{align}
    \sup_{y\in\supp(Y)} \Gamma_1(y) &\in (0,\infty), \qquad   \inf_{y\in\supp(Y)} \Gamma_1(y) > -1 .\label{eqn:gamma-cstr-1}
\end{align}
\Cref{eqn:gamma-cstr-1,eqn:gamma-cstr-3} yield the following definition of $\cG$:
\begin{align}
    \cG &\coloneqq \curbrkt{\Gamma\colon\bbR\to\bbR : 
    \inf_{y\in\supp(Y)} \Gamma(y) > -1 , \; 
    0 < \sup_{y\in\supp(Y)} \Gamma(y) < \infty , \;
    \prob{\Gamma(Y) = 0} < 1
    } . \label{eqn:def-feasible-set-gamma} 
\end{align}

Using the explicit representations of various functionals and variables, we can then write the optimization problem $ \mathsf{OL}_1^2 $ as
\begin{align}
    \mathsf{OL}_1^2 &= \sup_{\cT_1\in\cF} \frac{\frac{1}{\delta} - \expt{\paren{\frac{Z}{\lambda^*(\delta_1) - Z}}^2}}{\frac{1}{\delta} + \alpha \expt{\paren{\frac{Z}{\lambda^*(\delta_1) - Z}}^2(G^2 - 1)}} \notag \\
    &= \sup_{\Gamma_1\in\cG} \paren{\frac{\frac{1}{\delta} + \alpha\int_{\supp(Y)} \Gamma_1(y)^2 m_2(y) \diff y - \alpha\int_{\supp(Y)} \Gamma_1(y)^2 m_0(y) \diff y}{\frac{1}{\delta} - \int_{\supp(Y)}\Gamma_1(y)^2 m_0(y) \diff y}}^{-1} \notag \\
    &= \sup_{\Gamma_1\in\cG} \paren{\frac{\frac{1-\alpha}{\delta} + \alpha \int_{\supp(Y)} \Gamma_1(y)^2 m_2(y) \diff y}{\frac{1}{\delta} - \int_{\supp(Y)} \Gamma_1(y)^2 m_0(y) \diff y} + \alpha}^{-1} , \notag 
\end{align}
subject to the conditions
\begin{align}
    \expt{\frac{Z}{\lambda^*(\delta_1) - Z}(G^2 - 1)} &= \frac{1}{\alpha\delta} , \qquad 
    \frac{1}{\delta} > \expt{\paren{\frac{Z}{\lambda^*(\delta_1) - Z}}^2} , \label{eqn:two-cond} 
\end{align}
which can be alternatively written as follows using the notation in \Cref{sec:opt-threshold}:
\begin{align}
\frac{1}{\alpha\delta} &= \int_{\supp(Y)} \Gamma_1(y) (m_2(y)-m_0(y)) \diff y , \qquad 
\frac{1}{\delta} > \int_{\supp(Y)} \Gamma_1(y)^2 m_0(y) \diff y . \label{eqn:cond-opt-overlap}
\end{align}
Note that in the first identity in \Cref{eqn:two-cond}, we use $ \lambda^*(\delta_1)\ne0 $ which holds since $ \lambda^*(\delta_1)>\sup\supp(Z)>0 $ by \Cref{itm:assump-preproc-spec}.  
We observe that to solve $ \mathsf{OL}_1^2 $, it suffices to solve the following minimization problem: 
\begin{align}
    \mathsf{OPT}_1 &\coloneqq  \inf_{\Gamma_1\in\cG} \,  \paren{\frac{\frac{1-\alpha}{\delta} + \alpha \int_{\supp(Y)} \Gamma_1(y)^2 m_2(y) \diff y}{\frac{1}{\delta} - \int_{\supp(Y)} \Gamma_1(y)^2 m_0(y) \diff y}}  \notag \\
    &\phantom{\coloneqq} \suchthat\quad \text{\Cref{eqn:cond-opt-overlap}} . \notag 
\end{align}
Recall the standard fact that the minimum value of a function (subject to constraints) is the smallest level $\beta$ such that the $\beta$-level set (after taking the intersection with the constraint set) is non-empty. 
For any given $\beta>0$, define $ \cL_\beta $ as the set of $\Gamma_1\in\cG$ which induces an objective value at most $\beta$ and satisfies all conditions of $ \mathsf{OPT}_1 $: 
\begin{align}
    \cL_\beta &\coloneqq \curbrkt{\Gamma_1\in\cG : \begin{array}{l}
        \frac{\frac{1-\alpha}{\delta} + \alpha \int_{\supp(Y)} \Gamma_1(y)^2 m_2(y) \diff y}{\frac{1}{\delta} - \int_{\supp(Y)} \Gamma_1(y)^2 m_0(y) \diff y} \le \beta \\
        \frac{1}{\alpha\delta} = \int_{\supp(Y)} \Gamma_1(y) (m_2(y)-m_0(y)) \diff y \\
        \frac{1}{\delta} > \int_{\supp(Y)} \Gamma_1(y)^2 m_0(y) \diff y
    \end{array}} . \notag 
\end{align}
The first constraint describes the $\beta$-level set of the original objective of $\mathsf{OPT}_1$. 
Then $\mathsf{OPT}_1$ can be further written as 
\begin{align}
\mathsf{OPT}_1 &\coloneqq 
\inf\curbrkt{\beta>0 : \cL_\beta \ne\emptyset} . 
\label{eqn:beta-opt1}
\end{align}

The first constraint in $\cL_\beta$ is equivalent to:
\begin{align}
    \frac{1-\alpha}{\delta} + \alpha\int_{\supp(Y)} \Gamma_1(y)^2 m_2(y) \diff y &\le \frac{\beta}{\delta} - \beta \int_{\supp(Y)} \Gamma_1(y)^2 m_0(y) \diff y, \notag 
\end{align}
or
\begin{align}
    \int_{\supp(Y)} \Gamma_1(y)^2 (\alpha m_2(y) + \beta m_0(y)) \diff y &\le \frac{\beta - (1 - \alpha)}{\delta}, \label{eqn:before-divide-beta} 
\end{align}
since the third condition guarantees that the denominator of the LHS of the first inequality is positive. 
We also claim that $ \beta-(1-\alpha)>0 $ and divide both sides by it to obtain
\begin{align}
    \int_{\supp(Y)} \Gamma_1(y)^2 \frac{\alpha m_2(y) + \beta m_0(y)}{\beta - (1 - \alpha)} \diff y &\le \frac{1}{\delta} . \label{eqn:levelset-cond1-equiv} 
\end{align}
To see why $ \beta-(1-\alpha)>0 $, recall that $\beta$ is a possible value of $ \mathsf{OPT}_1 $ which in turn has the following relation to $ \mathsf{OL}_1^2 $:
\begin{align}
    \mathsf{OL}_1^2 &= \frac{1}{\mathsf{OPT}_1 + \alpha} . \label{eqn:ol-opt} 
\end{align}
Since $ 0\le \mathsf{OL}_1^2\le 1 $, this implies $ \beta\ge1-\alpha $. 
Furthermore, if $ \beta = 1-\alpha $, \Cref{eqn:before-divide-beta} implies that $\Gamma_1(y) = 0$ for almost every $y\in\supp(Y)$. 
However, since $ \Gamma_1\in\cG $, this cannot happen according to the third condition in the definition of $\cG$ (cf.\ \Cref{eqn:def-feasible-set-gamma}). 
Therefore $ \beta>1 - \alpha $. 

The following observation can further simplify the description of the set $ \cL_\beta $. 
The third (inequality) constraint can be dropped since $ m_2(y) , m_0(y)\ge0 $ and 
\begin{align}
\int_{\supp(Y)} \Gamma_1(y)^2 \frac{\alpha m_2(y) + \beta m_0(y)}{\beta - (1 - \alpha)} \diff y
\ge \int_{\supp(Y)} \Gamma_1(y)^2 m_0(y) \diff y . \notag 
\end{align}
Therefore the third constraint has already been guaranteed to be true given the first (inequality) constraint which is, as we just argued, equivalent to \Cref{eqn:levelset-cond1-equiv}. 

Now, with the above observations, we arrive at the following equivalent description of the $\beta$-level set $ \cL_\beta $:
\begin{align}
    \cL_\beta &= \curbrkt{\Gamma_1\in\cG : \begin{array}{l}
        \int_{\supp(Y)} \Gamma_1(y)^2 \frac{\alpha m_2(y) + \beta m_0(y)}{\beta - (1 - \alpha)} \diff y \le \frac{1}{\delta} \\
        \frac{1}{\alpha\delta} = \int_{\supp(Y)} \Gamma_1(y) (m_2(y)-m_0(y)) \diff y
    \end{array}} . \notag 
\end{align}
We observe that for any fixed $\beta>0$, $\cL_\beta\ne\emptyset$ if and only if $ \mathsf{INF}_1^{(\beta)} \le \frac{1}{\delta} $ where $ \mathsf{INF}_1^{(\beta)} $ is the value of the following constrained minimization problem:
\begin{align}
    \mathsf{INF}_1^{(\beta)} &\coloneqq \inf_{\Gamma_1\in\cG} \int_{\supp(Y)} \Gamma_1(y)^2 \frac{\alpha m_2(y) + \beta m_0(y)}{\beta - (1 - \alpha)} \diff y \notag \\
    &\phantom{\coloneqq} \suchthat\quad  \frac{1}{\alpha\delta} = \int_{\supp(Y)} \Gamma_1(y) (m_2(y)-m_0(y)) \diff y . \notag 
\end{align}

We turn to compute the value of $\mathsf{INF}_1^{(\beta)}$. 
Though $\mathsf{INF}_1^{(\beta)}$ appears as a constrained variational problem over function space, one can obtain its optimum (and the corresponding optimizer) from a different perspective by casting it as a \emph{linear program} over a certain Hilbert space. 
Specifically, motivated by the form of the objective of $\mathsf{INF}_1^{(\beta)}$, we define the following inner product on the function space: 
\begin{align}
    \inprod{f}{g}_{\beta, \alpha} &\coloneqq \int_{\supp(Y)} f(y)g(y) \frac{\alpha m_2(y) + \beta m_0(y)}{\beta - (1 - \alpha)} \diff y . \notag 
\end{align}
This induces the norm $\norm{\beta,\alpha}{f} = \sqrt{\inprod{f}{f}_{\beta,\alpha}}$
With the above notation, $ \mathsf{INF}_1^{(\beta)} $ can be written as 
\begin{align}
    \mathsf{INF}_1^{(\beta)} &= \inf_{\Gamma_1\in\cG} \norm{\beta,\alpha}{\Gamma_1}^2 \notag \\
    &\phantom{=} \suchthat \quad \inprod{\Gamma_1}{\frac{m_2 - m_0}{\frac{\alpha m_2 + \beta m_0}{\beta - (1-\alpha)}}}_{\beta,\alpha} = \frac{1}{\alpha\delta} . \notag 
\end{align}
In words, $\mathsf{INF}_1^{(\beta)}$ outputs the smallest norm of $\Gamma_1$ whose linear projection onto a given vector (viewed as an element in the defined Hilbert space) 
$ \frac{m_2 - m_0}{\frac{\alpha m_2 + \beta m_0}{\beta - (1-\alpha)}} $ is fixed. 
It is now geometrically clear that the minimizer $\Gamma_1^{(\beta)}$ must be aligned with the vector $\frac{m_2 - m_0}{\frac{\alpha m_2 + \beta m_0}{\beta - (1-\alpha)}}$.
Therefore $\Gamma_1^{(\beta)} = a^* \cdot \frac{m_2 - m_0}{\frac{\alpha m_2 + \beta m_0}{\beta - (1-\alpha)}}  $,
where the scalar $a^*\in\bbR$ is uniquely determined from the equality condition:
\begin{align}
    \int_{\supp(Y)} a^* \cdot \frac{m_2(y) - m_0(y)}{\frac{\alpha m_2(y) + \beta m_0(y)}{\beta - (1-\alpha)}} \cdot (m_2(y)-m_0(y)) \diff y = \frac{1}{\alpha\delta} , \notag
\end{align}
i.e., 
\begin{align}
    a^* &= \paren{\alpha\delta(\beta - (1-\alpha)) \int_{\supp(Y)} \frac{(m_2(y) - m_0(y))^2}{\alpha m_2(y) + \beta m_0(y)} \diff y}^{-1} . \notag
\end{align}
Let
\begin{align}
    f_\alpha(\beta) &\coloneqq (\beta - (1-\alpha)) \int_{\supp(Y)} \frac{(m_2(y) - m_0(y))^2}{\alpha m_2(y) + \beta m_0(y)} \diff y . \notag  
\end{align}
We have found the minimizer of $\mathsf{INF}_1^{(\beta)}$:
\begin{align}
    \Gamma_1^{(\beta)} &= \frac{1}{\alpha\delta f_\alpha(\beta)} (\beta - (1-\alpha)) \frac{m_2 - m_0}{\alpha m_2 + \beta m_0} , \notag 
\end{align}
and the resulting minimum value is given by
\begin{align}
    \mathsf{INF}_1^{(\beta)} &= \int_{\supp(Y)} \Gamma_1^{(\beta)}(y)^2 \frac{\alpha m_2(y) + \beta m_0(y)}{\beta - (1 - \alpha)} \diff y \notag \\
    &= \int_{\supp(Y)} \paren{\frac{1}{\alpha\delta f_\alpha(\beta)} (\beta - (1-\alpha)) \frac{m_2(y) - m_0(y)}{\alpha m_2(y) + \beta m_0(y)}}^2 \frac{\alpha m_2(y) + \beta m_0(y)}{\beta - (1 - \alpha)} \diff y \notag \\
    &= \frac{1}{(\alpha\delta f_\alpha(\beta))^2} \int_{\supp(Y)} (\beta - (1-\alpha)) \frac{(m_2(y) - m_0(y))^2}{\alpha m_2(y) + \beta m_0(y)} \diff y \,  = \,  \frac{1}{\alpha^2\delta^2f_\alpha(\beta)} . \notag 
\end{align}
It follows that 
\begin{align}
    \cL_\beta\ne\emptyset & \quad \iff \quad \mathsf{INF}_1^{(\beta)} \le\frac{1}{\delta} \quad \iff \quad \frac{1}{\alpha^2\delta^2f_\alpha(\beta)} \le \frac{1}{\delta} 
   \quad  \iff \quad f_\alpha(\beta) \ge \frac{1}{\alpha^2\delta}. \notag 
\end{align}
Recalling \Cref{eqn:beta-opt1}, the value of $\mathsf{OPT}_1$ is therefore equal to 
\begin{align}
    \mathsf{OPT}_1 &\coloneqq \inf\curbrkt{\beta>0 : f_\alpha(\beta) \ge \frac{1}{\alpha^2\delta}} . \notag 
\end{align}
Writing $f_\alpha$ as
\begin{align}
    f_\alpha(\beta) &= (\beta - (1-\alpha)) \int_{\supp(Y)} \frac{(m_2(y) - m_0(y))^2}{\alpha m_2(y) + \beta m_0(y)} \diff y 
    = \int_{\supp(Y)} \frac{(m_2(y) - m_0(y))^2}{\frac{\alpha m_2(y)}{\beta - (1-\alpha)} + \frac{m_0(y)}{1 - \frac{1-\alpha}{\beta}}} \diff y , \notag 
\end{align}
we see that $f_\alpha$ is increasing in $\beta$. 
This implies that $\mathsf{OPT}_1$ is equal to the critical $\beta_1^*(\delta,\alpha)$ that solves the following equation
\begin{align}
    f_\alpha(\beta_1^*(\delta,\alpha)) &= \frac{1}{\alpha^2 \delta} . \label{eqn:beta-fixed}
\end{align}

Putting our findings together, we have shown that the value of $\mathsf{OPT}_1$ equals $\beta_1^*(\delta,\alpha)$ satisfying \Cref{eqn:beta-fixed}, or more explicitly, 
\begin{align}
    (\beta_1^*(\delta,\alpha) - (1-\alpha)) \int_{\supp(Y)} \frac{(m_2(y) - m_0(y))^2}{\alpha m_2(y) + \beta_1^*(\delta,\alpha) m_0(y)} \diff y = \frac{1}{\alpha^2 \delta} , \label{eqn:opt-val-beta-star-fp} 
\end{align}
and the corresponding optimizer is given by
\begin{align}
    \Gamma_1^* = \Gamma_1^{(\beta_1^*(\delta,\alpha))} &= \frac{1}{\alpha\delta f_\alpha(\beta_1^*(\delta,\alpha))} (\beta_1^*(\delta,\alpha) - (1-\alpha)) \frac{m_2 - m_0}{\alpha m_2 + \beta_1^*(\delta,\alpha) m_0} \notag \\
    &= \alpha (\beta_1^*(\delta,\alpha) - (1-\alpha)) \frac{m_2 - m_0}{\alpha m_2 + \beta_1^*(\delta,\alpha) m_0} , \label{eqn:opt-gamma}
\end{align}
where the second equality follows from \Cref{eqn:opt-val-beta-star-fp}. 

In light of the relation between $\mathsf{OL}_1^2$ and $\mathsf{OPT}_1$ (cf.\ \Cref{eqn:ol-opt}) and the relation between $ \Gamma_1 $ and $ \cT_1 $ (cf.\ \Cref{eqn:gamma-t}), it is straightforward to translate results in \Cref{eqn:opt-val-beta-star-fp,eqn:opt-gamma} to the original problem $\mathsf{OL}_1^2$ we are interested in. 
Indeed, the value of $\mathsf{OL}_1^2$ equals
\begin{align}
    \mathsf{OL}_1^2 &= \frac{1}{\beta_1^*(\delta,\alpha) + \alpha} \label{eqn:opt-overlap-formula} 
\end{align}
and is achieved by
\begin{align}
    \cT_1^* &= \frac{\Gamma_1^*}{1 + \Gamma_1^*} = \frac{\alpha (\beta_1^*(\delta,\alpha) - (1-\alpha)) \frac{m_2 - m_0}{\alpha m_2 + \beta_1^*(\delta,\alpha) m_0}}{1 + \alpha (\beta_1^*(\delta,\alpha) - (1-\alpha)) \frac{m_2 - m_0}{\alpha m_2 + \beta_1^*(\delta,\alpha) m_0}} \notag \\
    &= \frac{\alpha(\beta_1^*(\delta,\alpha) - (1-\alpha)) (m_2 - m_0)}{\alpha (\beta_1^*(\delta,\alpha) + \alpha) m_2 + (1-\alpha)(\beta_1^*(\delta,\alpha) + \alpha) m_0} . \notag 
\end{align}
Recall that a multiplicative scaling of $\cT_1^*$ does not change its performance in terms of spectral threshold and overlap. 
Therefore, we multiply the above expression by $\frac{\beta_1^*(\delta,\alpha) + \alpha}{\beta_1^*(\delta,\alpha) - (1-\alpha)}$ and redefine $\cT_1^*$ as
\begin{align}
    \cT_1^* &= \frac{m_2 - m_0}{m_2 + \frac{1-\alpha}{\alpha} m_0} 
    = 1 - \frac{1}{\alpha\cdot\frac{m_2}{m_0} + (1-\alpha)} . \label{eqn:opt-preproc-formula} 
\end{align}
 We observe that $\cT_1^*$ in \Cref{eqn:opt-preproc-formula} above is the same as that in \Cref{eqn:opt-spec-T-thr} obtained in \Cref{sec:opt-threshold}.  This is discussed in \Cref{rk:tilting} below.

Finally, we claim that the supremum in $ \mathsf{OL}_1^2 $ can be \emph{achieved}, by verifying that $ \cT_1^* $ meets \Cref{itm:assump-preproc-spec}. 
Indeed, letting $ (G,\eps)\sim\cN(0,1)\otimes P_\eps $, we have
\begin{align}
    \inf_{y\in\supp(\cT_1^*(q(G,\eps)))} \cT_1^*(y) &\ge 1 - \frac{1}{1-\alpha} = -\frac{\alpha}{1-\alpha} > -\infty , \label{eqn:inf-lb}
\end{align}
provided $ \alpha<1 $. 
Also, it trivially holds that
\begin{align}
    \sup_{y\in\supp(\cT_1^*(q(G,\eps)))} \cT_1^*(y) &< 1 . \label{eqn:sup-ub} 
\end{align}
We then verify $ \prob{\cT_1^*(q(G,\eps)) = 0} < 1 $. 
To this end, observe that $m_2$ cannot be identically equal to $m_0$, otherwise $ \delta_1^* = \delta_2^* = \infty $ (cf.\ \Cref{eqn:def-delta1-star-pf}). 
Therefore $ m_2/m_0 $ is not constantly $1$ and $ \cT_1^* $ is not constantly $0$. 
Finally, we verify that 
\begin{align}
    \sup_{y\in\supp(\cT_1^*(q(G,\eps)))} \frac{m_2(y)}{m_0(y)} &> 1 , \label{eqn:m2-larger-than_m0}
\end{align}
which implies 
\begin{align}
    \sup_{y\in\supp(\cT_1^*(q(G,\eps)))} \cT_1^*(y) &> 1 - \frac{1}{\alpha + (1-\alpha)} = 0 . \label{eqn:sup-lb} 
\end{align}
\Cref{eqn:m2-larger-than_m0} follows since $ m_2\not\equiv m_0 $ and by \Cref{eqn:m2-m0-int-1}, 
\begin{align}
    \int_{\sup(\cT_1^*(q(G,\eps)))} (m_2(y) - m_0(y) )\diff y = 0 , \notag 
\end{align}
hence there must exist $y\in\sup(\cT_1^*(q(G,\eps)))$ such that $ m_2(y)>m_0(y) $. 

\begin{remark}[Coincidence of $\cT_1^*$ in \Cref{sec:opt-threshold,sec:opt-spec-overlap}]
\label{rk:tilting}
As noted in the proofs, \Cref{eqn:opt-spec-T-thr,eqn:opt-preproc-formula} coincide. 
The first part of \Cref{sec:opt-threshold}  exhibits a lower bound $ \delta_1^{\spec} \ge \delta_1^* $ (defined in \Cref{eqn:def-delta-1-star}), whereas the second part (from \Cref{eqn:find-optthr-1} onward) shows $ \delta_1^{\spec} \le \delta_1^* $. 
The upper bound can be alternatively obtained in \Cref{sec:opt-spec-overlap} by substituting in the condition $ \psi'(\lambda^*(\delta_1);\delta)>0 $ (cf.\ \Cref{eqn:cond-opt-overlap-for-rk}) the function $ \cT_1^* $ and recognizing that the condition is equivalent to $ \delta > \delta_1^* $. 
This recognition is, however, not entirely obvious and we find it more transparent to directly derive the upper bound in the second part of \Cref{sec:opt-threshold}. 
The price is that a slightly tilted version of $\cT_1^*$ (see $\wt\cT_1^*$ in \Cref{eqn:opt-spec-tilde-T}) is constructed. 
The tilting is an artifact of the proof technique and we expect that $\cT_1^*$ simultaneously minimizes the spectral threshold and maximizes the limiting overlap above the threshold. 
\end{remark}

\section{Universal lower bound on spectral threshold (missing proof in \Cref{rk:univ-lb-spec-thr})}
\label{sec:lower-bound-spec-thr}

We show that for $i \in \{1,2\}$, we have $ \delta^{\spec}_i \ge \frac{1}{2\alpha_i^2}$ for \emph{any} mixed GLM. 
This follows from the upper bound
\begin{align}
    \int_{\supp(Y)} \frac{(m_2(y) - m_0(y))^2}{m_0(y)} \diff y &\le 2,\label{eqn:upperbound-to-prove}
\end{align}
  in view of \Cref{eqn:spec-bound-1}. Here the functions $m_2$ and $m_0$ are defined in \Cref{eqn:gamma1star}.

A similar argument has been made in \cite[Appendix A]{lu2020phase} for the non-mixed \emph{complex} generalized linear model where the output $y_i$ depends on the linear measurement $ g_i = \inprod{a_i}{x^*} $ through its \emph{modulus}: $ y_i \sim p(\cdot\mid |g_i|) $. 
Here we provide a similar argument for the real case without requiring the modulus. 
Recalling the definition of $m_2$ and $m_0$ (cf.\ \Cref{eqn:gamma1star}), we have
\begin{align}
    \int_{\supp(Y)} \frac{(m_2(y) - m_0(y))^2}{m_0(y)} \diff y
    &= \int_{\supp(Y)} \frac{m_2(y)^2}{m_0(y)} \diff y - 2 \int_{\supp(Y)} m_2(y) \diff y + \int_{\supp(Y)} m_0(y) \diff y \notag \\
    &= \int_{\supp(Y)} \frac{m_2(y)^2}{m_0(y)} \diff y - 1 , \notag 
\end{align}
since 
\begin{align}
\begin{split}
    \int_{\supp(Y)} m_2(y) \diff y &= \int_{\supp(Y)} \expt{p(y|G)G^2} \diff y = \expt{G^2} = 1 , \\
    \int_{\supp(Y)} m_0(y) \diff y &= \int_{\supp(Y)} \expt{p(y|G)} \diff y = 1 . 
\end{split}
\label{eqn:m2-m0-int-1}
\end{align}
To show \Cref{eqn:upperbound-to-prove}, it then suffices to show $ \int \frac{m_2^2}{m_0}\le3 $. 
Using the Cauchy--Schwarz inequality, we can bound $m_2^2$ as follows:
\begin{align}
    m_2(y)^2 = \paren{\int_\bbR f(g)p(y|g)g^2 \diff g}^2 
   &  = \paren{\int_\bbR \sqrt{f(g)p(y|g)}g^2 \cdot \sqrt{f(g)p(y|g)} \diff g}^2 \notag \\
    &\le \paren{\int_\bbR f(g)p(y|g) g^4 \diff g} \cdot \paren{\int_\bbR f(g)p(y|g) \diff g} \notag \\
    &= \paren{\int_\bbR f(g)p(y|g) g^4 \diff g} \cdot m_0(y) , \notag 
\end{align}
where $ f(g) $ denotes the standard Gaussian density function. 
The desired bound then follows:
\begin{align}
    \int_{\supp(Y)} \frac{m_2(y)^2}{m_0(y)} \diff y
    &\le \int_{\supp(Y)}\int_\bbR f(g)p(y|g) g^4 \diff g\diff y \notag \\
    &= \int_\bbR \sqrbrkt{f(g)\paren{\int_{\supp(Y)} p(y|g) \diff y}g^4} \diff g = \expt{G^4} = 3 . \notag 
\end{align}

\section{State evolution of GAMP for mixed GLMs (proof of \Cref{prop:GAMP_SE})} \label{sec:se-gamp-mixed-glm}

Recall the GAMP iteration in \Cref{eq:gamp-eqn-general}:
\begin{align}
\begin{split}
    u^t &= \frac{1}{\sqrt{\delta}} \Abar f_t(v^t; \, \xone, \xtwo) - \sfb_t \, g_{t-1}(u^{t-1};y),  \qquad
    v^{t+1} = \frac{1}{\sqrt{\delta}} \Abar^\top g_{t}(u^t;y) - \sfc_t  \, f_t(v^t;\, \xone, \xtwo), 
\end{split}
\label{eq:gamp-eqn-det}
\end{align}
where the memory coefficients  are $\sfb_t =\frac{1}{n}\sum_{i=1}^d f_t'(v_i^t;\, \ol{x}_{1,i}^*, \ol{x}_{2,i}^*)$ and  
$\sfc_t = \frac{1}{n}\sum_{i=1}^n g_t'(u_i^t; y_i)$. Using the initialization $\tv^0$, the iteration starts with  $u^0= \frac{1}{\sqrt{\delta}}\Abar \tv^0$.

To prove the proposition, we rewrite \Cref{eq:gamp-eqn-det} as an AMP iteration with matrix-valued iterates. This matrix-valued AMP is designed to be a special case of an abstract AMP iteration for which a state evolution result has been established in \cite{javanmard2013state, Feng22AMPTutorial}. This state evolution result is then translated to obtain the results in \Cref{eq:psiG,eq:psiX}.
 (See also \cite{GBAMP21} for an analysis of a general graph-based AMP iteration that includes AMP with matrix-valued iterates.) 
 Given the iteration in \Cref{eq:gamp-eqn-det},  for $t \ge 1$, let
\begin{align}
  e^t \coloneqq \begin{bmatrix} g_1 & g_2 & u^t  \end{bmatrix} , \qquad   
  h^{t+1} \coloneqq  \, & v^{t+1} - \chi_{1,t+1}\xone   - \chi_{2,t+1}\xtwo \, ,
  \label{eq:etht1_def}
\end{align}
where we recall that $g_1 = \Abar \xone$, $g_2 = \Abar \xtwo$, and $\chi_{1,t}, \chi_{2,t}$ are the state evolution parameters computed via the recursion in \Cref{eq:UtVt_def,eq:muU_update,eq:G1G2Y_joint}. 
We also introduce the functions $ \brf_t: \reals^3 \to \reals^3$ and $\brg_t:\reals^5 \to \reals$, defined as:
\begin{align}
   & \brf_t(h^t; \, \xone , \xtwo) = \begin{bmatrix} \sqrt{\delta} \xone, & \sqrt{\delta} \xtwo, & f_t(h^t+ \chi_{1,t}\xone   + \chi_{2,t}\xtwo; \xone,\xtwo )  \end{bmatrix}, \\
    &  \brg_t(e^t; \, \veta, \veps) = g_t(e^t_3; \, q(\veta \odot e^t_1 + (1-\veta) \odot e^t_2,  \veps)).
\end{align}
Here, $\odot$ denotes element-wise multiplication, $\brf_t$ and $\brg_t$  act row-wise on their matrix-valued inputs and $e^t_{j}$ denotes the $j$-th column of $e^t \in \reals^{n \times 3}$. We also recall the notation $\veta=(\eta_1, \ldots, \eta_n)$, $\veps=(\eps_1, \ldots, \eps_n)$ and that $y = q( \veta \odot g_{1} + (1-\veta) \odot g_2, \,  \veps) =  q(\veta \odot e^t_1 + (1-\veta) \odot e^t_2, \,  \veps) $.  With these definitions, we claim that the AMP iteration in \Cref{eq:gamp-eqn-det} is equivalent to the following one:
\begin{equation}
\begin{split}
          & e^t  = \frac{1}{\sqrt{\delta}} \Abar \brf_t(h^t; \, \xone , \xtwo) \, - \, \brg_{t-1}(e^{t-1}; \, \veta, \veps) \sfB_t^{\top}, \\
      & h^{t+1}  =  \frac{1}{\sqrt{\delta}} \Abar^{\top} \brg_t(e^t; \, \veta, \veps)
      \, - \, \brf_t(h^t; \, \xone , \xtwo) \sfC_t^{\top}, 
\end{split}
\label{eq:et_ht1_def}
\end{equation}
where $\sfB_t \in \reals^{3 \times 1}$ and $\sfC_t \in \reals^{1 \times 3}$ are given by:
\begin{align}
    & \sfB_t = \begin{bmatrix} 0  & 0 &  \frac{1}{n}\sum_{i=1}^d f_t'(h_i^t + \chi_{1,t}\ol{x}^*_{1,i} + \chi_{2,t}\ol{x}^*_{2,i}; \ol{x}_{1,i}^*,\ol{x}_{2,i}^*) \end{bmatrix} ^{\top}, \nonumber \\
    &  \sfC_t = \begin{bmatrix}
    \E[\partial_{1} g_t(U_t; \, q(\eta G_1 + (1-\eta) G_2,  \eps))]  \\ 
    \E[\partial_{2} g_t(U_t; \, q(\eta G_1 + (1-\eta) G_2, \eps))]  \\ 
    \frac{1}{n}\sum_{i=1}^n g_t'(u_i^t; \, q(\eta_i g^{1}_i + (1-\eta_i) g^{2}_i, \epsilon_i) )]
    \end{bmatrix}^\top . \label{eq:sBsC_def}
\end{align}
Here $\partial_{1} g_t$ and $\partial_{2} g_t$ refer to the partial derivatives of $g_t(u; \, q( \eta g_1 + \eta g_2, \eps) )$ with respect to  $g_1$ and $g_2$, respectively, $(G_1, G_2, \eta, \eps) \sim \normal(0,1) \otimes \normal(0,1) \otimes \bern(\alpha) \otimes P_{\epsilon} $, and $U_t$ is defined as in \Cref{eq:UtVt_def}.  The iteration in \Cref{eq:et_ht1_def} is initialized with $e^0= [g_1 \quad g_2 \quad \frac{1}{\sqrt{\delta}}\Abar\tv^0 ]$  .

The equivalence between \Cref{eq:gamp-eqn-det,eq:et_ht1_def} can be seen by substituting in \Cref{eq:et_ht1_def}   the definitions of $e^t$ and $h^{t+1}$ from \Cref{eq:etht1_def}, and the fact that by Stein's lemma,  $\chi_{1,t+1}$ and $\chi_{2,t+1}$  defined in \Cref{eq:UtVt_def,eq:muU_update,eq:G1G2Y_joint} can be expressed as \cite[Lemma 4.1]{Feng22AMPTutorial}:
\begin{equation}
  \chi_{1,t+1} = \sqrt{\delta}\, \E[\partial_{1} g_t(U_t; q(\eta G_1 + (1-\eta) G_2,  \eps))] , 
  \;
  \chi_{2,t+1} = \sqrt{\delta}\, \E[\partial_{2} g_t(U_t; q(\eta G_1 + (1-\eta) G_2,  \eps))].
\end{equation}
The recursion in \Cref{eq:et_ht1_def} is a special case of the abstract AMP recursion with matrix-valued iterates for which a state evolution result has been established \cite{javanmard2013state}, \cite[Sec.\ 6.7]{Feng22AMPTutorial}. The standard form of the abstract AMP recursion uses empirical estimates (instead of expected values) for the first two entries of $\sfC_t$ in \Cref{eq:sBsC_def}. However,  the state evolution result remains valid for the recursion in \Cref{eq:et_ht1_def} (see Remark 4.3 of \cite{Feng22AMPTutorial}). This result states that the empirical distributions of the rows of $e^t$ and $h^{t+1}$ converge to the Gaussian distributions  $\normal(0, \Sigma_t)$ and $\normal(0, \Omega_{t+1})$, respectively.  The covariances $\Sigma_t \in \reals^{3 \times 3}$ and $\Omega_{t+1}   \in \reals$ are defined by the following state evolution recursion:
\begin{align}
    &\Sigma_{t} = \frac{1}{\delta} \E[\brf_t(G^{\omega}_t; \,  X_1, X_2) \brf_t(G^{\omega}_t; \,  X_1, X_2)^{\top}] \label{eq:Sigmat_rec} \\
    &= {\footnotesize
    \begin{bmatrix}
    1 & 0 &  \frac{\E[ X_1 f_t( G^{\omega}_t + \chi_{1,t}X_1 + \chi_{2,t}X_2; X_1,X_2) ]}{\sqrt{\delta}} \\
    0 &  1 & \frac{\E[ X_2 f_t( G^{\omega}_t + \chi_{1,t}X_1 + \chi_{2,t}X_2; X_1,X_2) ]}{\sqrt{\delta}} \\
    \frac{\E[ X_1 f_t( G^{\omega}_t + \chi_{1,t}X_1 + \chi_{2,t}X_2; X_1,X_2) ]}{\sqrt{\delta}}  & \frac{\E[ X_2 f_t( G^{\omega}_t + \chi_{1,t}X_1 + \chi_{2,t}X_2; X_1,X_2) ]}{\sqrt{\delta}}  & \frac{\E[(f_t( G^{\omega}_t + \chi_{1,t}X_1 + \chi_{2,t}X_2; X_1,X_2))^2]}{\delta}  
    \end{bmatrix}
    } , \nonumber \\
    &\Omega_{t+1} = \E[ (\brg_t(G^{\sigma}_t; \, \eta, \epsilon))^2 ]
     = \E[(g_t(G^{\sigma}_{t,3}; \, 
     q(G^{\sigma}_{t,1}, G^{\sigma}_{t,2}, \eta, \epsilon)))^2], \label{eq:Omegat_rec}
\end{align}
where $G^{\sigma}_t \equiv (G^{\sigma}_{t,1}, G^{\sigma}_{t,2}, G^{\sigma}_{t,3} ) \sim \normal(0, \Sigma_t)$ is independent of $(\eta, \eps) \sim  \bern(\alpha) \otimes P_{\epsilon} $,
and $(G^{\omega}_t, X_1, X_2) \sim \normal(0, \Omega_t) \otimes \normal(0,1) \otimes \normal(0,1)$. The recursion is initialized with
\begin{equation}
    \Sigma_0 =  \begin{bmatrix}
    1 & 0 & \frac{1}{\sqrt{\delta}}\E[  F_0(X_1, X_2) X_1 ] \\
    0 & 1 & \frac{1}{\sqrt{\delta}}\E[  F_0(X_1, X_2) X_2 ]  \\
    \frac{1}{\sqrt{\delta}}\E[  F_0(X_1, X_2) X_1 ]  & \frac{1}{\sqrt{\delta}}\E[  F_0(X_1, X_2) X_2 ]  & 
    \frac{1}{\delta} \E[  (F_0(X_1, X_2))^2 ]
    \end{bmatrix}.
    \label{eq:Sigma0}
\end{equation}

The sequences $(G^{\sigma}_t)_{t \geq 0}$ and 
$(G^{\omega}_{t+1})_{t \geq 0}$ are each jointly Gaussian with the following covariance structure:
\begin{equation}
\begin{split}
    &G^{\sigma}_{t,1} = G_1, \qquad 
    G^{\sigma}_{t,2} = G_2, \quad \forall \, t \ge 0 \quad \text{ where } (G_1, G_2) \sim \normal(0,1) \otimes \normal(0,1), \\
    &\expt{G^{\sigma}_{0,3}G^{\sigma}_{t,3}} = \frac{1}{\delta}\expt{F_0(X_1, X_2) f_t(G^{\omega}_t + \chi_{1,t}X_1 + \chi_{2,t}X_2; X_1,X_2)}, \quad t \ge 1,
    \end{split}
    \label{eq:Gsigma_Gomega}
\end{equation}
and for $r, t \ge 1$:
\begin{subequations}
\begin{multline}
    \expt{G^{\omega}_{r} G^{\omega}_t } =
    \E[g_{r-1}(G^{\sigma}_{r-1,3}; \, 
     q(\eta G^{\sigma}_{r-1,1} + (1-\eta) G^{\sigma}_{r-1,2},  \epsilon) \\
     \times g_{t-1}(G^{\sigma}_{t-1,3}; \, 
     q(\eta G^{\sigma}_{t-1,1} + (1 - \eta) G^{\sigma}_{t-1,2}, \epsilon)], 
\end{multline} 
\begin{multline}
    \expt{G^{\sigma}_{r,3}G^{\sigma}_{t,3}} =\frac{1}{\delta} \expt{f_r(G^{\omega}_r + \chi_{1,r}X_1 + \chi_{2,r}X_2; X_1,X_2) \, f_t(G^{\omega}_t + \chi_{1,t}X_1 + \chi_{2,t}X_2; X_1,X_2)}.
\end{multline}
\label{eq:Grt_corr}
\end{subequations}

The following proposition follows from the state evolution result in \cite{javanmard2013state}, \cite[Sec.\ 6.7]{Feng22AMPTutorial} for AMP with  matrix-valued iterates.  
\begin{proposition}[State Evolution]
With setup and assumptions of \Cref{prop:GAMP_SE}, consider the AMP recursion in \Cref{eq:et_ht1_def}.  The following holds almost surely for  any $\PL(2)$ functions $\Psi: \reals^{t+1}  \to \reals$, $\Phi: \reals^{t+3}  \to \reals$, for $t \geq 0$:
\begin{align}
& \lim_{n \to \infty}\frac{1}{n} \sum_{i=1}^n \Psi( e^t_i, e^{t-1}_i, \ldots, e^0_{i} ) = \E[ \Psi( G^{\sigma}_t, \, G^{\sigma}_{t-1}, \, \ldots, G^{\sigma}_0) ], \label{eq:psiGsigma} \\
& \lim_{d \to \infty} \frac{1}{d} \sum_{i=1}^d \Phi( \ol{x}^*_{1,i}, \ol{x}^*_{2,i}, h^{t+1}_i, h^{t}_i, \ldots, h^1_i) = \E[ \Phi(X_1, X_2, G^{\omega}_{t+1}, \, G^{\omega}_{t}, \, \ldots, G^{\omega}_1) ], \label{eq:psiGomega}
\end{align}
where the joint distributions of $(G^{\sigma}_t, \, G^{\sigma}_{t-1}, \, \ldots, G^{\sigma}_0)$ and $(G^{\omega}_{t+1}, \, G^{\omega}_{t}, \, \ldots, G^{\omega}_1))$ are as given in \Cref{eq:Sigmat_rec,eq:Omegat_rec,eq:Sigma0,eq:Gsigma_Gomega,eq:Grt_corr}.
\label{prop:GAMP_matrix_SE}
\end{proposition}

Recall the definitions of $e^t, h^{t+1}$ from \Cref{eq:etht1_def}. Then, the results in \Cref{eq:psiGsigma,eq:psiGomega} imply that for any $\PL(2)$ function $\Phi: \reals^{t+3} \to \reals$, we have for $t \ge 0$:
\begin{align}
    & \lim_{d \to \infty} \frac{1}{d} \sum_{i=1}^d \Phi( \ol{x}^*_{1,i}, \ol{x}^*_{2,i}, v^{t+1}_i,  \ldots, v^1_i) \nonumber \\
    & \qquad = \E[ \Phi(X_1, X_2, \chi_{1,t+1}X_1 + \chi_{2,t+1}X_2 + G^{\omega}_{t+1}, \,  \ldots, \, 
    \chi_{1,1}X_1 + \chi_{2,1}X_2  +G^{\omega}_1) ] \label{eq:thm_result_vi} \\
    & \lim_{n \to \infty}\frac{1}{n} \sum_{i=1}^n \Phi( g_{1,i}, g_{2,i},  u^t_{i}, \ldots, u^0_{i} )
    =\expt{\Phi(G_1, G_2, G^{\sigma}_{t,3}, \ldots, G^{\sigma}_{0,3} )}. \label{eq:thm_result_ui}
\end{align}
Recalling $G^{\sigma}_{t,1}=G_1$ and $G^{\sigma}_{t,2}=G_2$ for $t \ge 0$, using \Cref{eq:Sigmat_rec} we can write
\begin{multline}
    G^{\sigma}_{t,3} =
    \expt{G^{\sigma}_{t,3} \mid G^{\sigma}_{t,1}, G^{\sigma}_{t,2}} \, + \,  W^{\sigma}_t  = \frac{1}{\sqrt{\delta}} \expt{X_1 f_t( G^{\omega}_t + \chi_{1,t}X_1 + \chi_{2,t}X_2; X_1,X_2)}  G_1 \notag \\
    +  \,  \frac{1}{\sqrt{\delta}} 
    \expt{ X_2 f_t( G^{\omega}_t + \chi_{1,t}X_1 + \chi_{2,t}X_2; X_1,X_2) } G_2 \, + \, W^{\sigma}_t, \notag
\end{multline} 
where $W^{\sigma}_t$ is a zero-mean Gaussian independent of $(G_1, G_2)$ with variance 
\begin{multline}
    \expt{(W^{\sigma}_t)^2} = \frac{1}{\delta}  \E[(f_t( G^{\omega}_t + \chi_{1,t}X_1 + \chi_{2,t}X_2; X_1,X_2))^2] \\
    - \frac{1}{\delta}  (\expt{X_1 f_t( G^{\omega}_t + \chi_{1,t}X_1 + \chi_{2,t}X_2; X_1,X_2)} )^2 
    -  \frac{1}{\delta}  ( \expt{ X_2 f_t( G^{\omega}_t + \chi_{1,t}X_1 + \chi_{2,t}X_2; X_1,X_2) })^2. \notag
\end{multline}  
To complete the proof,  for $t \ge 0$, we define:
\begin{equation}
\begin{split}
    & W_{U,t} \coloneqq W^{\sigma}_t, \quad  W_{V,t+1}\coloneqq G^{\omega}_{t+1}, \\
    & \mu_{1,t} \coloneqq \frac{1}{\sqrt{\delta}} \E[ X_1 f_t( G^{\omega}_t + \chi_{1,t}X_1 + \chi_{2,t}X_2; X_1,X_2) ], \\
    & \mu_{2,t} \coloneqq  \frac{1}{\sqrt{\delta}} \E[ X_2 f_t( G^{\omega}_t + \chi_{1,t}X_1 + \chi_{2,t}X_2; X_1,X_2) ],  \\
    & U_t\coloneqq G^{\sigma}_{3,t} = \mu_{1,t} G_1 + \mu_{2,t} G_2 + W_{U,t}, \quad V_{t+1} = \chi_{1,t}X_1 + \chi_{2,t}X_2 + W_{V, t+1}.
    \end{split} \notag
\end{equation}
Using these in the convergence statements in \Cref{eq:thm_result_vi,eq:thm_result_ui} gives the  result of \Cref{prop:GAMP_SE}.

\section{Auxiliary lemmas} \label{sec:aux}


\begin{lemma}
\label{lem:properties-phi-psi}
Consider the setting of \Cref{sec:prelim} and let \Cref{itm:assump-preproc-spec} hold. 
Then the following properties of $ \phi(\lambda),\psi(\lambda;\Delta)$ hold. 
\begin{enumerate}
    \item \label{itm:property-1} $ \phi(\cdot) $ is strictly decreasing;
    \item \label{itm:property-2} For any $ \Delta>0 $, $ \psi(\cdot;\Delta) $ is strictly convex in the first argument; 
    \item \label{itm:property-3} For any $ \lambda>\sup\supp(Z) $, $ \psi(\lambda;\cdot) $ is strictly decreasing in the second argument.
\end{enumerate}
\end{lemma}

\begin{proof}
The proof follows by checking the derivatives. 
For $ \phi $, we have
\begin{align}
\frac{\diff}{\diff\lambda}\phi(\lambda) &= \expt{\frac{ZG^2}{\lambda-Z}} - \lambda\expt{\frac{ZG^2}{(\lambda-Z)^2}}
= - \expt{\paren{\frac{ZG}{\lambda-Z}}^2} < 0 , \label{eqn:phi-deriv}
\end{align}
The last strict inequality holds since $Z$ is not almost surely zero, i.e., $\prob{Z = 0}<1$ in \Cref{itm:assump-preproc-spec}. 
\Cref{itm:property-1} of \Cref{lem:properties-phi-psi} then follows. 
For $ \psi $, we have
\begin{align}
\frac{\partial}{\partial\lambda}\psi(\lambda;\Delta) &= \frac{1}{\Delta} + \expt{\frac{Z}{\lambda - Z}} - \lambda \expt{\frac{Z}{(\lambda-Z)^2}} 
= \frac{1}{\Delta} - \expt{\frac{Z^2}{(\lambda-Z)^2}} , \label{eqn:derivative-psi} 
\end{align}
and
\begin{align}
\frac{\partial^2}{\partial\lambda^2}\psi(\lambda;\Delta) &= 2\expt{\frac{Z^2}{(\lambda-Z)^3}} . \notag 
\end{align}
Since $ \lambda>\sup\supp(Z) $ and $Z$ is not almost surely zero, $ \psi''(\cdot;\Delta)>0 $ and therefore \Cref{itm:property-2} of \Cref{lem:properties-phi-psi} holds. 
Finally, \Cref{itm:property-3} of \Cref{lem:properties-phi-psi} is obvious, provided $ \lambda>\sup\supp(Z)>0 $ where the second inequality is guaranteed by \Cref{itm:assump-preproc-spec}. 
\end{proof}


\begin{lemma}[Explicit formulas]
\label{rk:explicit-formulas}
Fix any $\Delta>0$. 
The parameter $ \ol\lambda(\Delta) $ satisfies 
\begin{align}
    \expt{\paren{\frac{Z}{\ol\lambda(\Delta) - Z}}^2} &= \frac{1}{\Delta} . 
    \label{eqn:def-lam-bar-explicit}
\end{align}
If $ \lambda^*(\Delta) > \ol\lambda(\Delta) $, the parameter $ \lambda^*(\Delta) $ satisfies
\begin{align}
    \expt{\frac{Z(G^2 - 1)}{\lambda^*(\Delta) - Z}} &= \frac{1}{\Delta} . 
    \label{eqn:def-lami-star-supcrit-explicit}
\end{align}

\end{lemma}

\begin{proof}
Since $ \ol\lambda(\Delta) $ (cf.\ \Cref{eqn:def-lam-bar}) is the minimum point of $ \psi(\cdot;\Delta) $, it satisfies 
\begin{align}
    \left.\frac{\partial}{\partial\lambda}\psi(\lambda;\Delta)\right|_{\lambda = \ol\lambda(\Delta)} &= 0 . \notag
\end{align}
This gives \Cref{eqn:def-lam-bar-explicit} according to \Cref{eqn:derivative-psi}. 

Under the condition $\lambda^*(\Delta) > \ol\lambda(\Delta)$, we have $ \zeta(\lambda^*(\Delta);\Delta) = \psi(\lambda^*(\Delta);\Delta) $ and $ \lambda^*(\Delta) $ satisfies the fixed point equation $ \psi(\lambda^*(\Delta);\Delta) = \phi(\lambda^*(\Delta)) $ which in turn can be written more explicitly as \Cref{eqn:def-lami-star-supcrit-explicit}. 
\end{proof}

\begin{remark}
\Cref{eqn:def-lami-star-supcrit-explicit} is often used with $ \Delta = \delta_i $ ($i\in\{1,2\}$) under the condition $ \lambda^*(\delta_i) > \ol\lambda(\delta) $. 
This is legitimate since $ \ol\lambda(\delta) > \ol\lambda(\delta_i) $ (see, e.g., \Cref{eqn:lambda-bar-order}) and the latter condition is stronger than the one in \Cref{rk:explicit-formulas}. 
\end{remark}